\documentclass[11pt]{article}
\usepackage{jheppub}

\makeatletter
\def\@fpheader{\relax}
\makeatother

\usepackage{amssymb,amsmath,amsfonts,stmaryrd,mathscinet}
\usepackage{xcolor}

\usepackage{tikz}
\usetikzlibrary{decorations.markings,arrows,decorations.pathreplacing}

\usepackage{multicol}
\usepackage{multirow}

\newcommand\arXiv[1]{\href{http://arxiv.org/abs/#1}{\nolinkurl{arXiv:#1}}}
\newcommand\MRnumber[1]{\href{http://www.ams.org/mathscinet-getitem?mr=#1}{\nolinkurl{MR#1}}}
\newcommand\DOI[1]{\href{http://dx.doi.org/#1}{\nolinkurl{DOI:#1}}}

\newcommand\myurl[1]{\href{http://#1}{\texttt{#1}}}

\usepackage{amsthm}

\newtheorem{dummy}{Dummy}[section]
\newtheorem{theorem}[dummy]{Theorem}
\newtheorem{proposition}[dummy]{Proposition}
\newtheorem{conjecture}[dummy]{Conjecture}
\newtheorem{corollary}[dummy]{Corollary}
\newtheorem{definition}[dummy]{Definition}
\newtheorem{lemma}[dummy]{Lemma}
\newtheorem{question}[dummy]{Question}

\usepackage{dsfont}
\renewcommand\mathbb\mathds

\newcommand\bB{\mathbb B}
\newcommand\bC{\mathbb C}

\newcommand\bF{\mathbb F}

\newcommand\bK{\mathbb K}

\newcommand\bQ{\mathbb Q}
\newcommand\bR{\mathbb R}

\newcommand\bX{\mathbb X}

\newcommand\bZ{\mathbb Z}

\newcommand\cC{\mathcal C}
\newcommand\cD{\mathcal D}
\newcommand\cE{\mathcal E}
\newcommand\cF{\mathcal F}

\newcommand\cH{\mathcal H}

\newcommand\cJ{\mathcal J}

\newcommand\cM{\mathcal M}
\newcommand\cN{\mathcal N}
\newcommand\cO{\mathcal O}
\newcommand\cP{\mathcal P}

\newcommand\cS{\mathcal S}

\newcommand\cV{\mathcal V}

\newcommand\cX{\mathcal X}

\newcommand\rA{\mathrm A}
\newcommand\rB{\mathrm B}

\newcommand\rH{\mathrm H}

\newcommand\rJ{\mathrm J}

\newcommand\rL{\mathrm L}
\newcommand\rM{\mathrm M}

\newcommand\rO{\mathrm O}

\newcommand\rU{\mathrm U}

\usepackage[mathscr]{euscript}

\DeclareMathOperator\homology{H}
\renewcommand\H{\homology}

\renewcommand\d{\mathrm d}

\newcommand\MF{\mathrm{MF}}

\newcommand\TMF{\mathrm{TMF}}
\newcommand\Tmf{\mathrm{Tmf}}
\newcommand\tmf{\mathrm{tmf}}
\newcommand\mf{\mathrm{mf}}
\newcommand\ku{\mathrm{ku}}
\newcommand\SQM{\mathrm{SQM}}

\newcommand\pt{\mathrm{pt}}

\newcommand\GL{\mathrm{GL}}
\newcommand\SL{\mathrm{SL}}

\DeclareMathOperator\SH{SH}
\DeclareMathOperator\tr{tr}

\newcommand\longto\longrightarrow
\newcommand\mono\hookrightarrow
\newcommand\epi\twoheadrightarrow
\newcommand\isom{\overset\sim\to}
\newcommand\<\langle
\renewcommand\>\rangle
\newcommand\sminus\smallsetminus

\DeclareMathOperator\Fer{Fer}

\DeclareMathOperator\im{im}
\DeclareMathOperator\Sq{Sq}
\DeclareMathOperator\Aut{Aut}

\DeclareMathOperator\SQFT{SQFT}

\newcommand\Perm{\mathbf{24}}

\newcommand\Co{\mathrm{Co}}
\newcommand\SO{\mathrm{SO}}
\newcommand\Spin{\mathrm{Spin}}

\newcommand\define[1]{\emph{#1}}
\newcommand\cat[1]{\mathbf{#1}}

\title{Topological Mathieu Moonshine}
\author{Theo Johnson-Freyd}
\affiliation{Department of Mathematics, Dalhousie University, Halifax, NS, CANADA \\Perimeter Institute for Theoretical Physics, Waterloo, ON, CANADA}
\emailAdd{tjohnsonfreyd@perimeterinstitute.ca}
\abstract{We explore the Atiyah--Hirzebruch spectral sequence for the  $\operatorname{tmf}^\bullet[\frac12]$-cohomology of the classifying space $B\mathrm{M}_{24}$ of the largest Mathieu group $\mathrm{M}_{24}$, twisted by a class $\omega \in \mathrm{H}^4(B\mathrm{M}_{24};\mathbb{Z}[\frac12]) \cong \bZ_3$. Our exploration includes detailed computations of the $\mathbb{F}_3$-cohomology of $\mathrm{M}_{24}$ and of the first few differentials in the AHSS. We are specifically interested in the value of $\operatorname{tmf}^\bullet_\omega(B\mathrm{M}_{24})[\frac12]$ in cohomological degree~$-27$. Our main computational result is that $\operatorname{tmf}^{-27}_\omega(B\rM_{24})[\frac12] = 0$ when $\omega \neq 0$. For comparison, the restriction map $\operatorname{tmf}^{-3}_\omega(B\mathrm{M}_{24})[\frac12]\to \operatorname{tmf}^{-3}(\mathrm{pt})[\frac12] \cong \bZ_3$ is surjective for one of the two nonzero values of $\omega$.
 
Our motivation comes from Mathieu Moonshine. Assuming a well-studied conjectural relationship between $\operatorname{TMF}$ and supersymmetric quantum field theory, there is a canonically-defined $\mathrm{Co}_1$-twisted-equivariant lifting $[\overline{V^{f\natural}}]$ of the class $\{24\Delta\} \in \operatorname{TMF}^{-24}(\mathrm{pt})$, for a specific value $\omega$ of the twisting, where $\mathrm{Co}_1$ denotes Conway's largest sporadic group. We conjecture that the product $[\overline{V^{f\natural}}] \nu$, where $\nu \in \operatorname{TMF}^{-3}(\mathrm{pt})$ is the image of the generator of  $\operatorname{tmf}^{-3}(\mathrm{pt}) \cong \mathbb{Z}_{24}$, does not vanish $\mathrm{Co}_1$-equivariantly, but that its restriction to $\mathrm{M}_{24}$-twisted-equivariant $\mathrm{TMF}$ does vanish. We explain why this conjecture answers some of the questions in Mathieu Moonshine: it implies the existence of a minimally supersymmetric quantum field theory with $\mathrm{M}_{24}$ symmetry, whose twisted-and-twined partition functions have the same mock modularity as in Mathieu Moonshine. Our AHSS calculation establishes this conjecture ``perturbatively'' at odd primes.
 
An appendix included mostly for entertainment purposes discusses ``$\ell$-complexes,'' in which the differential $D$ satisfies $D^\ell=0$ rather than $D^2=0$, and their relation to $\mathrm{SU}(2)$ Verlinde rings. The case $\ell=3$ is used in our AHSS calculations.}

\preprint{}
\begin{document}

\maketitle

%\theo{add reference to Baker, On the homotopy type of the spectrum representing elliptic cohomology, Proc. AMS 107 (1989) \cite{MR982399}}

%\theo{add reference to \cite{MR3729259}}

\section{Introduction}\label{sec.intro}

By writing the elliptic genus of an $\cN=(4,4)$ K3 sigma model in terms of characters of the chiral $\cN=4$ superalgebra, Eguchi,  Ooguri, and  Tachikawa~\cite{MR2802725} discovered a specific weight-$\frac12$ mock modular form (for $\Gamma = \mathrm{SL}_2(\bZ)$) with shadow $24 \eta(\tau)^3$:
$$ H(\tau) = 2q^{-1/8}\bigl( -1 + 45q + 231q^2 + 770q^3 + 2277q^4 +\dots\bigr)$$
Physics readily explains the mock modularity and integrality of $H$. It does not, however, explain why the coefficients of $H$ are dimensions of representations of Mathieu's largest group $\rM_{24}$ \cite{MR3539377}, and more generally raises the following mysteries:
\begin{question}\label{question.modularity}
Whenever a finite group $G$ acts on a K3 sigma model preserving $\cN=(4,4)$ supersymmetry, the elliptic genus can be twisted and twined by a commuting pair of elements $g,h \in G$. This produces twisted-twined versions $H_{g,h}(\tau)$ of $H(\tau)$ with interesting (mock) modularity properties, with multiplier that depends on the 't~Hooft anomaly of $G$. The group $G = \rM_{24}$ {does not} act nontrivially on any K3 sigma model~\cite{MR2955931}, but nevertheless the functions $H_{g,h}(\tau)$ exist for all commuting pairs $g,h \in \rM_{24}$~\cite{MR3108775}. Why?
\end{question}
\begin{question}\label{question.positivity}
 A priori, the supertrace in the elliptic genus allows for a large cancelation of bosonic and fermionic modes. In particular, the coefficients of $g \mapsto H_g(\tau) = H_{e,g}(\tau)$ are automatically virtual characters of $G$, but have no reason to be honest characters. Nevertheless, except for the constant term $-1$, these coefficients are honest characters~\cite{MR3539377}. Why?
 \end{question}
\begin{question} \label{question.growth}
The functions $H_{g,h}(\tau)$ enjoy a mock-modular analogue of the ``genus-zero property'' from monstrous moonshine~\cite{MR3021323}. Why?
\end{question}
The goal of this note is to suggest a solution to Question~\ref{question.modularity}. We will not provide a complete solution---some calculations are too hard---but our suggestion will at least answer what type of quantum field theory it is that can produce the functions $H_{g,h}(\tau)$.
 We will have nothing to say about Question~\ref{question.positivity}. We will briefly comment in Conjecture~\ref{conjecture.growth} about Question~\ref{question.growth}.

The first step is to recast the problem as a question in stable homotopy theory, and in particular in elliptic cohomology. This follows the spirit of \cite{MR2500561,MR2681787} to explain aspects of moonshine in elliptic cohomological terms, but we believe that many aspects of the specific approach here are new.

As explained in Section~\ref{sec.sqft}, compact minimally supersymmetric $(1{+}1)$-dimensional quantum field theories are the cocycles for an extraordinary cohomology theory $\SQFT^\bullet$. This statement is not mathematically rigorous: even the set of ``$(1{+}1)$-dimensional quantum field theories'' is not mathematically defined (although~\cite{MR2742432} comes close), and topologizing this set will surely be subtle, but the construction is physically straightforward. This cohomology theory connects directly with mock modularity~\cite{GJFmock}: if $\cS$ is an SQFT of cohomological degree $1-4k$ representing the trivial class in $\SQFT^{1-4k}(\pt)$, then any nullhomotopy of $\cS$ determines a (generalized) mock modular form with shadow determined by $\cS$. We will call the theory $\cS = \partial \cF$ the \define{boundary} of its nullhomotopy $\cF$. (Note that this is not a ``boundary condition,'' where the boundary is on the worldsheet. Rather, it should be thought of as a boundary in ``field space'' or ``target space,'' because if $\cF$ is a sigma model with target $M$, then $\partial \cF$ is a sigma model with target $\partial M$.)

If the boundary SQFT $\cS$ furthermore admits an action by a finite group $G$ of flavour symmetries, and if the nullhomotopy is $G$-equivariant, then the same construction produces mock modular forms depending on commuting pairs $(g,h)$. 
The level structure depends on the orders of $g$ and $h$, and the multiplier system depends on the 't~Hooft anomaly $\omega \in \H^3(\rM_{24}; \rU(1)) \cong \H^4(\rM_{24}; \bZ)$ of the $G$-action. 
(For the purposes of this introduction, we will ignore the fact that 't~Hooft anomalies for fermionic QFTs live in ``supercohomology'' and not in ordinary cohomology.)
In algebrotopological language, the fact that it makes sense to talk about deformations of SQFTs with $G$-flavour symmetry and anomaly $\omega$ means that the
cohomology theory $\SQFT^\bullet$ has a \define{twisted equivariant} enhancement, allowing us to define twisted equivariant cohomology groups $\SQFT^\bullet_\omega(\bB G)$ for any finite group $G$ and anomaly $\omega \in \H^4(G;\bZ)$. Here and throughout, we will write $\bB G$ for the classifying stack of $G$; a more standard name for $\SQFT^\bullet_\omega(\bB G)$ is $\SQFT^\bullet_{G,\omega}(\pt)$.

For example, the direct sum $\overline{\Fer(3)}{}^{\oplus 24}$ of $24$ copies of the antiholomorphic superconformal field theory $\overline{\Fer(3)}$ (three antichiral Majorana--Weyl fermions, with supersymmetry encoding the structure constants of $\mathfrak{su}(2)$) is nullhomotopic~\cite{WittenTMF}, and the corresponding mock modular form is $H(\tau)$. We can let $\rM_{24}$ act on $\overline{\Fer(3)}{}^{\oplus 24}$ by permuting the summands. Writing $\Perm$ for the standard degree-$24$ permutation representation of $\rM_{24}$, we will call the corresponding $\rM_{24}$-equivariant SCFT $\Perm \otimes \overline{\Fer(3)}$.
Because the  $\rM_{24}$-symmetry spontaneously breaks to $\rM_{23}$, and because $\H^4(\rM_{23};\bZ) = 0$, we can think of the $\rM_{24}$ action on $\Perm \otimes \overline{\Fer(3)}$as having any 't~Hooft anomaly that we want (see \S\ref{subsec.anomalies}). Thus we have classes $[\Perm \otimes \overline{\Fer(3)}] = [\Perm] \otimes [\overline{\Fer(3)}] \in \SQFT^{-3}_\omega(\bB \rM_{24})$ for every $\omega$. If one of them were nullhomotopic, then the nullhomotopy, with its  corresponding mock modular forms, might explain Mathieu Moonshine.

Unfortunately, we will show in Proposition~\ref{prop.permDNW} that $\Perm \otimes \overline{\Fer(3)}$ is not $\rM_{24}$-equivariantly nullhomotopic (for any value of the 't~Hooft anomaly). Rather, the boundary SQFT that we will use is $\cS = \overline{V^{f\natural}} \otimes \overline{\Fer(3)}$, where $V^{f\natural}$ is the holomorphic SCFT constructed in~\cite{MR2352133}. The automorphism group of $V^{f\natural}$ is Conway's largest sporadic group $\Co_1$, which contains $\rM_{24}$ as a subgroup; the computations in~\cite{JFT} show that the anomaly $\omega$ of the corresponding $\rM_{24}$-action on $\cS$ agrees with the anomaly for Mathieu Moonshine computed in~\cite{MR3108775}. Cohomological degrees in $\SQFT^\bullet$ are determined by the central charges of the representing QFTs, and this $\cS$ represents a class in cohomological degree $-27$. Our suggested answer to Question~\ref{question.modularity} is:

\begin{conjecture}\label{conjecture.main}
  The antiholomorphic SCFT $\cS = \overline{V^{f\natural}} \otimes \overline{\Fer(3)}$ represents the trivial class $[\cS] = 0$ in $\SQFT^{-27}_{\omega}(\bB \rM_{24})$.
\end{conjecture}

Without further information about $\SQFT^\bullet$, it seems impossible to test this conjecture. But in fact there is a rather clear idea of the structure of $\SQFT^\bullet$, with evidence continuing to amass~\cite{MR992209,MR2101224,MR2079378,CheungThesis,MR2330519,MR2742432,DBE2015,1610.00747,BET1,GJFIII,BET2,WittenTMF,GJFmock} in favour of the following conjecture:

\begin{conjecture}\label{conjecture.ST}
  The spectrum $\SQFT^\bullet$ represents the universal elliptic cohomology theory $\TMF^\bullet$ of ``topological modular forms'' described in~\cite{MR2597740,MR3223024}. 
\end{conjecture}

Under this equivalence, the class $[\overline{\Fer(3)}] \in \SQFT^{-3}(\pt)$ corresponds to the class usually denoted $\nu \in \TMF^{-3}(\pt) = \pi_3 \TMF$, the image under the Hurewicz map of the 3-sphere $S^3 = \mathrm{SU}(2)$ with its Lie group framing~\cite{WittenTMF}, and the class $[\overline{V^{f\natural}}] \in \SQFT^{-24}(\pt)$ is $\{24 \Delta\} \in \TMF^{-24}(\pt)$ \cite{GJFIII}, where $\Delta = (c_4^3 - c_6^2) / 1726$ is the usual modular discriminant. (The curly braces around $\{24\Delta\}$ are there because $\Delta$ itself is not a class in $\TMF^{-24}(\pt)$.)

Recently a complete definition of equivariant TMF has become available~\cite{EllipticIII,2004.10254}.
Assuming Conjecture~\ref{conjecture.ST}, Conjecture~\ref{conjecture.main} becomes:

\begin{conjecture}\label{conjecture.mainTMF}
  There is a distinguished refinement of $\{24\Delta\} \in \TMF^{-24}(\pt)$ to a class in $\TMF^{-24}_\omega(\bB \Co_1)$, and, after multiplying by $\nu$ and restricting along $\rM_{24} \subset \Co_1$, the class
  $\{24 \Delta\} \nu$ vanishes in $\TMF^{-27}_\omega(\bB \rM_{24})$.
\end{conjecture}

Note that the $\rM_{24}$-action on $V^{f\natural}$, and hence on $\cS = \overline{V^{f\natural}} \otimes \overline{\Fer(3)}$, extends to a $\Co_1$-action. However, we do not believe that $[\cS] = \{24\Delta\}\nu$ vanishes $\Co_1$-equivariantly. It is worth emphasizing that, in order to define $\{24 \Delta\} \nu \in \TMF^{-27}_\omega(\bB \rM_{24})$, one would need to show that the nonequivariant class $\{24\Delta\} \in \TMF^{-24}(\pt)$ admits an equivariant refinement to a class in $\TMF^{-24}_\omega(\bB \Co_1)$. The existence of such a refinement is implied by Conjecture~\ref{conjecture.ST}, but it has not been shown mathematically rigorously.
Furthermore, in \S\ref{subsec.tmf} we will suggest that an answer to Question~\ref{question.growth} might come from proving:
\begin{conjecture}
  $\{24 \Delta\}$ refines to a class in $\operatorname{Tcf}^{-24}_\omega(\bB\Co_1)$, the space of (twisted) $\Co_1$-equivariant \define{topological cusp forms}, and the restriction of $\{24 \Delta\}\nu$ vanishes in $\operatorname{Tcf}^{-27}_\omega(\bB\rM_{24})$.
\end{conjecture}

Unfortunately, this author is not aware of techniques for computing twisted equivariant $\TMF^\bullet$ (let alone $\operatorname{Tcf}^\bullet$) groups. Instead, as evidence in favour of Conjecture~\ref{conjecture.mainTMF}, we will attempt to compute the related group $\tmf^{-27}_\omega(B \rM_{24})$. There are two changes involved. First, we have replaced the genuinely equivariant problem with the Borel-equivariant one. Any group $G$ has a classifying space $B G$, and for any cohomology theory $\cE^\bullet$,  \define{Borel-equivariant $\cE^\bullet$-cohomology} studies cohomology of $B G$ in place of $\bB G$.
As with Atiyah--Segal completion for K-theory \cite{MR2172633,MR2307274},
 one expects in general that $\cE^\bullet(BG)$ is an approximation of $\cE^\bullet(\bB G)$, but the latter may include more information than the former. (In fact, the ``completion'' story for TMF is subtle, and typically fails for Lie groups \cite{2004.10254}, but seems to hold for finite groups.)
Second, we have replaced the spectrum $\TMF^\bullet$ by the related spectrum $\tmf^\bullet$. Speaking very roughly (see \S\ref{subsec.tmf} for an important correction), $\tmf^\bullet$ corresponds to the modular forms which are bounded at the cusp $\tau = i\infty$ and $\TMF^\bullet$ corresponds to the modular forms which are meromorphic at the cusp; on homotopy groups, $\TMF^\bullet(\pt) = \tmf^\bullet[\Delta^{-24}]$, and if a class in $\tmf^\bullet$ vanishes, then its image in $\TMF^\bullet$ also vanishes.
There is no known physical description of $\tmf^\bullet$, and there is not expected to be one.

 Actually, computing $\tmf^{\bullet}_\omega(B\rM_{24})$ is still too hard, because the $2$-local structure of $\tmf^\bullet$ is complicated and the $2$-local cohomology of $\rM_{24}$ is not known. So we will attempt only $\tmf^{\bullet}_\omega(B \rM_{24})[\frac12]$. 
After further inverting $3$, the spectrum $\tmf^\bullet[\frac16]$ becomes the spectrum called ``$E\ell\ell$'' in~\cite{MR1263724}, where it is shown that $\tmf^\bullet_\omega(B\rM_{24})[\frac16]$ (which is independent of $\omega$) is supported only in even degrees. As such, our computation is interesting only at the prime~$3$. 
After a detailed study of $\H^\bullet(\rM_{24}; \bF_3)$ in Section~\ref{sec.HM24}, in Section~\ref{sec.AHSS} we investigate the Atiyah--Hirzebruch spectral sequence for
$\tmf^{\bullet}_\omega(B \rM_{24})[\frac12]$. 
Note that we are particularly interested in the groups $\tmf^{-27}_\omega(B\rM_{24})$, which houses the
image under the completion map $\tmf(\bB\rM_{24}) \to \tmf(B \rM_{24})$ of the equivariant enhancement of $\{24\Delta\}\nu$,
and $\tmf^{-3}_\omega(B \rM_{24})$, which houses the image under completion of $\Perm \nu$. Our main mathematical result is:

\begin{theorem}\label{thm.main}
  If $\omega \in \H^4(\rM_{24}; \bZ[\frac12]) \cong \bZ_3$ is nonzero, then $\tmf^{-27}_\omega(B\rM_{24}) = 0$.
   For comparison, for one of the two nonzero values of $\omega$, and not the other one, the restriction map $\tmf^{-3}_\omega(B \rM_{24})[\frac12] \to \tmf^{-3}(\pt)[\frac12]$ is nonzero. 
\end{theorem}

 Spectral sequences in general, and the Atiyah--Hirzebruch spectral sequence in particular, are the homotopy algebraist's version of perturbation theory. 
 Indeed, 
a physicist should think of the difference between $\TMF^\bullet_\omega(\bB G)$ and $\TMF^\bullet_\omega(B G)$ as the difference between nonperturbative and perturbative field theory. One can pull back along the map $BG \to \bB G$ to produce a map $\TMF^\bullet_\omega(\bB G) \to \TMF^\bullet_\omega(B G)$. The domain, hypothetically, encodes deformation classes of SQFTs with $G$-flavour symmetry, and in particular their behaviours on worldsheets equipped with arbitrary $G$-bundle, whereas the codomain remembers only the physics ``near the trivial $G$-bundle.''
(Since $G$ is finite, the stack of $G$-bundles has no perturbative structure ``over $\bC$,'' but it does have perturbative structure $p$-locally for any prime $p$ dividing the order of $G$.)

To end the paper, Appendix~\ref{appendix-3complex} describes some of the theory of ``chain complexes'' in which the ``differential'' does not satisfy $D^2=0$ but rather $D^\ell = 0$ for $\ell > 2$. Some of this theory, for $\ell=3$, is important in our calculations. The larger story connects in intriguing ways to the Verlinde ring for $\mathrm{SU}(2)$ at level $k=\ell-2$, and some readers may find it entertaining.

\subsection{Notation}\label{subsec.notation}

We will write $\bZ_n$ for the cyclic group of order $n$. This name is reasonably standard in the physics literature; mathematicians may prefer $C_n$ or $\bZ/n\bZ$.
For other finite groups, we generally follow ATLAS naming conventions~\cite{ATLAS}.
For a prime $p$, we will write $\bZ_{(p)}$ for the ring of $p$-local integers, i.e.\ the subring of $\bQ$ consisting of fractions  with denominator coprime to $p$. The finite field with $q = p^n$ elements is $\bF_q$; a generic field is~$\bK$.

We will always use cohomological degree conventions, with degrees always written as superscripts. For example, the \define{homotopy groups} of a spectrum $\cE^\bullet$ are $\cE^\bullet(\pt) = \pi_0\cE^\bullet = \pi_{-\bullet}\cE$.
If $\cE^\bullet$ is \define{connective} (e.g.\ $\tmf^\bullet$), then these groups are supported in \emph{nonpositive} cohomological degree. Without care, this paper would devolve into alphabet soup. So, for example, Bockstein maps will be denoted $\Box$ rather than $\beta$. We will sometimes write the group cohomology of a finite group $G$, with coefficients in an abelian group $A$, as $\H^\bullet(G;A)$, and sometimes as $\H^\bullet(BG;A)$, with ``$BG$'' denoting the classifying space of $G$. For an extraordinary cohomology theory $\cE^\bullet$, we will always use the latter name: $\cE^\bullet(BG)$ is the $\cE^\bullet$-cohomology of the space~$BG$. If $\cE^\bullet$ also admits an equivariant refinement, then we can evaluate $\cE^\bullet$ on the classifying stack $\bB G$ of $G$; by definition $\cE^\bullet(\bB G) = \cE^\bullet_G(\pt)$ is the $G$-equivariant cohomology of a point. When $\cE^\bullet = \H^\bullet(-;A)$ is ordinary cohomology, the groups $\H^\bullet(BG;A)$ and $\H^\bullet(\bB G;A)$ agree, justifying our use of simply $\H^\bullet(G;A)$.

\subsection{Acknowledgements}

I thank D.~Berwick-Evans for detailed and helpful comments on a draft of this paper. 
%I thank D.~Gaiotto for patiently reading very long emails in which I worked out this material during the 2020 worldwide work-from-home order, and I thank D.~Berwick-Evans for providing detailed comments on a draft of the paper. 
Sections~\ref{subsect.holomorphicanomaly} and~\ref{subsec.spectrum} are based on ideas developed by the author jointly with D.\ Gaiotto, whom I thank for ongoing discussions about this circle of ideas.
Research at Perimeter Institute is supported in part by the Government of Canada through the Department of Innovation, Science and Economic Development Canada and by the Province of Ontario through the Ministry of Colleges and Universities. The Perimeter Institute is in the Haldimand Tract, land promised to the Six Nations. Dalhousie University is in Mi`kma`ki, the ancestral and unceded territory of the Mi`kmaq.

\section{$\cN=(0,1)$ SQFTs}\label{sec.sqft}

The goal of this section is to describe a detailed but largely conjectural relationship between elliptic cohomology, mock modularity, and supersymmetry. As such, the language will drift between mathematics and theoretical physics, and we will leave some statements mathematically imprecise. Sections~\ref{sec.HM24} and~\ref{sec.AHSS} consist of rigorous mathematical calculations motivated by the conjectures in this section.

\subsection{A source of holomorphic anomalies} \label{subsect.holomorphicanomaly}

The starting point of our analysis is the following question: What type(s) of quantum field theories produce mock modular forms? The (an?)\ answer has been well-investigated for more than a decade~\cite{MR2313986,MR2679378,MR2680313,MR2821103,MR2949287,murthy13,Ashok:2013pya,Gupta:2017bcp,GJFmock,1905.05207,2004.05742,DPW2020}: A $(1{+}1)$-dimensional quantum field theory can produce mock modular instead of modular forms if it is noncompact. 

We will not in this paper attempt to define ``quantum field theory.'' We will always assume our QFTs to be unitary, so that we have access to Wick rotation to Euclidean signature (imaginary time). The physics literature does not seem to include a complete definition of ``compactness'' for a QFT, but the consensus is that it should be a ``spectral condition,'' since in the case of sigma-models what distinguishes compact from noncompact target is that the former lead to Hamiltonians with discrete spectrum, whereas the latter have continuous spectrum. We propose the following: a $(d{+}1)$-dimensional QFT is \define{compact} if its Wick-rotated partition function ``converges absolutely'' on all closed spacetimes: in Lagrangian formalism, we imagine an ``absolutely convergent path integral'' (in spite of the fact that not all QFTs have path integral descriptions, and most spaces of fields do not support measures of integration in the mathematical sense); in Hamiltonian formalism, we are asking that the Wick-rotated evolution operator $\tr(\exp(-\tau \hat H))$ should be trace-class. This latter condition occurs when the spectrum of the Hamiltonian $\hat H$ is bounded below, discrete, and does not grow too slowly. Compactness is a nontopological version of asking whether a functorial topological field theory is defined on all cobordisms, or if it is only partially defined.

Badly noncompact QFTs might even fail to assign Hilbert spaces of states to all closed $d$-dimensional spaces. 
The most mild type of noncompactness is when the Hilbert spaces are all well-defined, but the Wick-rotated partition function converges only conditionally. The value of a conditionally-convergent sum or integral can depend on the method used to evaluate it, and so the partition function of a mildly noncompact QFT is not quite well-defined. This is the origin of phenomena like mock modularity in noncompact QFTs: modular transformations may not be compatible with the chosen evaluation method.

Focusing on the case we care most about, let $\cF$ be a $(1{+}1)$-dimensional QFT, and write $Z(\cF)$ for its partition function on flat oriented $2$-dimensional tori (these being the only flat closed oriented $2$-manifolds). If $\cF$ has fermions, then $Z(\cF)$ depends on a choice of spin structure on the worldsheet. We will care most about the case of nonbounding spin structure, which is to say the Ramond spin structure along both the $A$- and $B$-cycles; we will thus call this the ``Ramond-Ramond''  partition function $Z_{RR}(\cF)$. The space of flat tori (with RR spin structure) is $3$-real dimensional: the local coordinates are the complex structure $(\tau,\bar\tau)$ and the area $a$. If $\cF$ is compact, then $Z(\cF)$ is a well-defined function of these three real variables, and we assume that it is real-analytic. (As with essentially all analytic questions about QFT, this is an assumption, and we must fold it into some aspect of the definition of ``compact QFT.'') Because different values of $(\tau,\bar\tau,a)$ describe the same torus, $a \mapsto Z_{RR}(\tau,\bar\tau,a)$ is a real-analytic family of real-analytic $\SL_2(\bZ)$-modular functions.
In the conformal case, of course, there is no $a$-dependence.

Now suppose that $\cF$ is not just a compact QFT, but also is equipped with an $\cN=(0,1)$ supersymmetry. A standard argument then says that $Z_{RR}(\cF)$ depends only on $\tau$~\cite{MR885560}. This argument is so familiar that we will not review it, except to make a few comments:
\begin{enumerate}
\item The statement only holds in the Ramond-Ramond spin structure. 
\item Let $\bar{Q}$ denote the supercurrent for the $\cN=(0,1)$ supersymmetry. (It is usually called ``$\bar{G}$,'' but we will soon want the letter $G$ to stand for a finite group.) This supercurrent is a worldsheet spinor, and so has two components, which explain the two nondependencies (on $\bar\tau$ and on $a$). Given coordinates $z,\bar z$ on the worldsheet, we can write the two components of $\bar{Q}$ as $\bar{Q}_z$ and $\bar{Q}_{\bar z}$. The former is the ``trace'' of $\bar{Q}$, and vanishes if $\cF$ is superconformal.
\item The growth rate of $Z_{RR}(\cF)(\tau)$ as $\tau \to i\infty$ is not worse than $\exp(\tau_2 c / 24)$, where $c$ is the central charge of $\cF$, and $\tau_2 = (\tau - \bar\tau)/2i$ is the imaginary part of $\tau$.  As such, $Z_{RR}(\cF)(\tau)$ is a \define{weakly holomorphic modular function,} meaning a modular function which is holomorphic for finite $\tau$, and meromorphic at the cusp $\tau = i\infty$.
\item $\cF$ has a \define{gravitational anomaly} if its left and right central charges $c_L,c_R$ do not match. The difference $w = c_R - c_L$ is always a half-integer. When $w \neq 0$, $Z_{RR}(\cF)$ suffers a multiplier under $T$-transformations: $$T[Z_{RR}(\cF)] = e^{-w\frac{2\pi i}{24}}Z_{RR}(\cF).$$ This multiplier can be absorbed by adjusting $$Z_{RR}(\cF) \leadsto Z'_{RR}(\cF) = Z_{RR}(\cF) \eta(\tau)^{2w}.$$ This \define{adjusted partition function} $Z'_{RR}(\cF)$ is then a weakly holomorphic modular form of weight $w$ and trivial multiplier. It matches better with the mathematics conventions for Witten genera. In 
\cite{GJFmock}, this adjustment (up to a factor of $i^w$ that we will ignore) is interpreted in terms of ``spectator'' Majorana--Weyl fermions that are added to $\cF$ to cure the anomaly.
 In \S\ref{subsec.spectrum} we will interpret the integer $n = -2w$ as the \define{cohomological degree} of $\cF$.
 \item The $q$-expansion of $Z_{RR}(\cF)$ is the index of the $S^1$-equivariant supersymmetric quantum mechanics model $\cH(\cF)$ formed by compactifying $\cF$ on an $A$-cycle with Ramond spin structure, thus explaining why $Z_{RR}(\cF) \in \bZ(\!(q)\!)$. The $q$-expansion of $Z'_{RR}(\cF)$ is built by adjusting the SQM model by some spectator fermions.
\end{enumerate}

When $\cF$ is not compact, the standard arguments can break down, as we have already indicated. 
Following~\cite{GJFmock}, we will focus on a particularly mild noncompactness, which is when $\cF$ has ``cylindrical ends.''

In order to give the definition, we will need the following construction.
 Let $\Phi$ be a self-adjoint operator in the SQFT $\cF$, thought of as a ``function'' $\Phi: \cF \to \bR$. There is a straightforward way to construct the ``fibre'' of $\Phi$ over $x \in \bR$, which we will denote $\cF(\Phi = x)$, or $\cF(x)$ when $\Phi$ is implicit. Namely, add to $\cF$ a chiral Majorana--Weyl fermion $\lambda$, which will serve as a Lagrange multiplier, to produce the QFT $\cF \otimes \Fer(1)$. Now deform the supersymmetry on $\cF \otimes \Fer(1)$ by adding a superpotential equal to
$$ W = \lambda (\Phi-x).$$
This results in an adjustment of the Lagrangian like $\lambda \bar{Q}[\Phi-x] + (\Phi-x)^2$. 
 In the IR, one expects this $\cF(\Phi = x)$ to flow to an SQFT in which $\Phi$ takes the constant value $x$. 
 
 Conversely, one expects to be able to recover $\cF$ from the $\bR$-family of SQFTs $x \mapsto \cF(\Phi = x)$ by dynamicalizing the parameter $x$. This dynamicalization procedure involves replacing the parameter $x$ by a scalar field $\phi$  and also introducing its (right-moving) superpartner $\psi$, so that together $(\phi,\psi)$ is a {scalar supermultiplet} for the $\cN=(0,1)$ algebra. We will write the result of this dynamicalization as
 $$ \cF(x) \mapsto \int_{\phi,\psi} \cF(\phi).$$
 
The SQFT $\cF$ is then said to have \define{cylindrical ends} if it can be equipped with a $\Phi$ such that the SQFTs $\cF(x)$ are all compact, and if supersymmetry is spontaneously  broken when $x \ll 0$, and if the theories $\cF(x)$ stabilize to some fixed SQFT $\partial \cF$ when $x \gg 0$. We will call $\partial \cF$ the \define{boundary} of $\cF$. Note that this is a boundary in field space, and not a ``boundary condition'' that can be assigned to boundaries of the worldsheet.

For example, if $\cF$ consists of a  massless scalar $\phi$ (i.e.\ a noncompact nonchiral boson), together with its superpartner $\psi$ (an antichiral Majorana--Weyl fermion), and if $\Phi = \phi^2$, then $\cF(x)$ picks up a quartic potential $(\phi^2 - x)^2$ and a Yukawa coupling $\lambda \phi \psi$. If $x<0$, $\cF(x)$ has spontaneous supersymmetry breaking, whereas if $x > 0$, $\cF(x)$ has two massive vacua, with fermion masses of opposite signs, and so the two vacua differ by a relative Arf invariant (c.f.\ \S2.1.1 of~\cite{1911.11780}). After fixing a convention about fermion masses (as in~\cite[Section 2]{10.21468/SciPostPhys.7.1.007}, for example), we see that the noncompact scalar multiplet, i.e.\ the sigma-model with target $\bR$, has cylindrical boundary equal to $(\text{trivial TQFT}) \oplus (\text{Arf TQFT})$.

Suppose that $\cF$ has cylindrical ends, parameterized by an operator $\Phi$. Then the partition function of $\cF$ has no reason to converge absolutely. But if the partition function of the boundary vanishes,
$$ Z_{RR}(\partial \cF) = 0,$$
then one expects that the path integral description of $Z_{RR}(\cF)$ will converge conditionally, because the end $\bR_{\gg 0} \times \partial \cF$ contributes a term like $0 \times \operatorname{vol}(\bR_{\gg 0})$, which we take to be $0$. In this way we can define $Z_{RR}(\cF)$. (The value of $Z_{RR}(\cF)$ might depend on the parameterization~$\Phi$.) 

Because we never chose coordinates on the worldsheet,  $Z_{RR}(\cF)(\tau,\bar\tau,a)$ is manifestly $\mathrm{SL}_2(\bZ)$-modular, and we will assume that it is real-analytic. However, it is not automatically holomorphic, because the standard argument for holomorphicity requires compactness. Rather, $Z_{RR}(\cF)(\tau,\bar\tau,a)$ can suffer a holomorphic anomaly. One of the main results of~\cite{GJFmock} is a precise formula for this holomorphic anomaly. The justification given in that paper is a combination of heuristic arguments about path integrals and applying Stokes' formula in field space, together with carefully-worked examples to fix the proportionality factors. 
A more detailed proof in the case of sigma models is given in~\cite{DPW2020}.
But in the general case the arguments from~\cite{GJFmock}
 would need further improvements in order to provide a ``theorem'' even at physicists' level of rigour, and so we will call it here a ``conjecture'':

\begin{conjecture}[\cite{GJFmock}]\label{conj.gjfA}
  Suppose $\cF$ is an $\cN=(0,1)$ SQFT with cylindrical ends and boundary $\partial \cF$, and that $Z_{RR}(\partial \cF) = 0$. Then deformations of $\cF$ that are ``compactly supported,'' i.e.\ that don't deform the end $\cF_{\gg 0}$, do not effect the value of $Z_{RR}(\cF)$.
Moreover,
 the $\bar\tau$- and $a$-dependence of the conditionally convergent partition function $Z_{RR}(\cF)(\tau,\bar\tau,a)$ are determined entirely by the boundary $\partial\cF$.
In particular, if $\partial \cF$ is superconformal, then $Z_{RR}(\partial \cF)$ has no $a$-dependence, and its $\bar\tau$-dependence is governed by the \define{holomorphic anomaly equation} (up to convention-dependent factors of $\sqrt[4]{-1}$):
  \begin{equation*}
    \sqrt{-8\tau_2} \eta(\tau) \frac{\partial }{\partial \bar\tau}Z_{RR}(\cF)  = (\text{torus one-point function of $\bar{Q}_{\bar z}$ in $\partial \cF$}) 
  \end{equation*}
\end{conjecture}

Thus, if $\partial\cF$ is superconformal, the adjusted partition function $Z'_{RR}(\cF) = Z_{RR}(\cF)\eta(\tau)^w$ is a real-analytic, but not holomorphic, modular form of weight $w$, where $w = c_R - c_L$ is the gravitational anomaly of $\cF$.
Any real-analytic modular form $\hat{f}(\tau,\bar\tau)$ has a \define{holomorphic part} $f(\tau)$, defined by analytically continuing and then taking a limit
$$ f(\tau) = \lim_{\bar\tau \to -i\infty} \hat{f}(\tau,\bar\tau),$$
assuming the limit exists. 
This is an example of a \define{generalized mock modular form}, with \define{shadow} the complex conjugate of $\sqrt{-8\tau_2}\frac{\partial \hat f}{\partial \bar\tau}$. 
(An excellent reference on many aspects of mock modularity is \cite{MR3729259}.)
It is honestly \define{mock modular} if the shadow is (weakly) holomorphic. 
In particular, suppose that $\partial \cF$ is a purely antiholomorphic SCFT (i.e.\ all of its fields are antichiral). Then the torus one-point function of $\bar{Q} = \bar{Q}_{\bar z}$ is antiholomorphic, and so
$$ f(\tau) = \eta(\tau) \lim_{\bar\tau \to -i\infty} Z_{RR}(\cF)(\tau,\bar\tau)$$
will be a weight-$1/2$ mock modular form (with multiplier), and shadow (the complex conjugate of) the torus one-point function of $\bar{Q}$ in $\partial\cF$.

The analysis in~\cite{GJFmock} furthermore suggests:
\begin{conjecture}[\cite{GJFmock}]\label{conj.gjfB}
  Suppose that $\cF$ as in Conjecture~\ref{conj.gjfA}, with $\partial \cF$ superconformal. Then the holomorphic part of $Z_{RR}(\cF)$ exists (the limit converges). Its $q$-expansion is the index of the $S^1$-equivariant SQM model $\cH(\cF)$ formed by compactifying $\cF$ on an appropriate Ramond circle called the \define{Tate curve}. (The compactification explicitly breaks $\SL_2(\bZ)$-modularity.) This index lives in $\bZ(\!(q)\!)$ up to a correction given by an Atiyah--Patodi--Singer invariant of $\partial\cF$. Because of the extra time-reversal symmetry of $\cH(\cF)$ (c.f.\ \S3.2.2 of~\cite{GPPV}), this APS invariant is just a half-integer related to a certain ``mod-2 index'' of $\partial\cF$.
\end{conjecture}

The following example was the primary motivation for~\cite{WittenTMF,GJFmock}. Take a K3 surface with $24$ punctures, and arrange a B-field on $M$ so that its flux near each puncture satisfies $\int_{S^3} H/2\pi = 1$. Now form an $\cN=(0,1)$ sigma model with target this noncompact $4$-manifold. (The $(0,1)$ worldsheet supersymmetry enhances to $(0,4)$ by using the hyperka\"ahler structure. The B-field is needed to cancel an anomaly that would otherwise be present because of the mismatched fermions \cite{MR796163}.) The result is a noncompact SCFT $\cF$ with cylindrical ends. The boundary theory $\partial \cF = \overline{\Fer(3)}{}^{\oplus 24}$ is a direct sum of $24$ copies of the same theory, one for each puncture. The contribution from each puncture is a purely antiholomorphic SCFT $\overline{\Fer(3)}$ consisting of three antichiral Majorana--Weyl fermions $\bar\psi_1,\bar\psi_2,\bar\psi_3$ and supersymmetry $\bar G = {:}\bar\psi_1\bar\psi_2\bar\psi_3{:}$, up to convention-dependent factors of $\sqrt[4]{-1}$. The torus one-point function of $\bar G$ in each summand is $\eta(\bar\tau)^3$.
Thus the K3 surface produces a mock modular form with shadow $24\eta(\tau)^3$, namely the function $H(\tau)$ that we started with in Section~\ref{sec.intro}.

$\overline{\Fer(3)}{}^{\oplus 24}$ is not the only possible boundary theory for producing $H(\tau)$, and is not the one we will end up using. There is a famous holomorphic SCFT called $V^{f\natural}$ constructed in~\cite{MR2352133}, with automorphism group $\Aut(V^{f\natural}) = \Co_1$ and central charge $c_L = 12$ (and $c_R = 0$). We will work instead with its reflection to an antiholomorphic SCFT $\overline{V^{f\natural}}$. The supersymmetry together with the antiholomorphicity imply that the Ramond-sector partition function $Z_{RR}(\overline{V^{f\natural}})$ simply counts the Ramond-sector ground states, of which there are $24$. Thus the torus one-point function of $\bar{Q}$ in $\overline{V^{f\natural}} \otimes \overline{\Fer(3)}$ is
$$ \langle \bar{Q} \rangle_{\overline{V^{f\natural}} \otimes \overline{\Fer(3)}} = \langle 1 \rangle_{\overline{V^{f\natural}}} \, \langle \bar{Q} \rangle_{\overline{\Fer(3)}} + \langle \bar{Q} \rangle_{\overline{V^{f\natural}}} \, \langle 1 \rangle_{\overline{\Fer(3)}}
= 24 \eta(\bar\tau)^3 + 0.
$$
As we will explain in \S\ref{subsec.tmf}, Conjecture~\ref{conjecture.ST} implies
 that $\overline{V^{f\natural}} \otimes \overline{\Fer(3)}$ is a boundary of an $\cN=(0,1)$ SQFT $\cF$ with cylindrical ends, thus providing another source of mock modular forms with shadow $24 \eta(\tau)^3$.

\subsection{$\SQFT^\bullet$ as an $\Omega$-spectrum}\label{subsec.spectrum}

The story in the previous section applies in the presence of a finite group $G$ of flavour symmetries. Namely, suppose that $\cF$ is a noncompact SQFT with cylindrical ends $\partial \cF$ and $G$-flavour symmetry. After averaging, we may assume that $\Phi$ is $G$-invariant. Then $G$ acts on $\partial \cF$, and so the right-hand side of the holomorphic anomaly equation~\ref{conj.gjfA} may be twisted and twined by elements of $G$, and we predict that it will be the holomorphic anomaly for the corresponding twisted-twined partition function of $\cF$.
After taking holomorphic parts, we would produce a (generalized) mock modular form valued in characters of $G$.

Thus we can answer Question~\ref{question.modularity} if we can produce an SQFT $\cS$ which is not just the boundary of an SQFT $\cF$ with cylindrical ends, but is such a boundary compatibly with an $\rM_{24}$-flavour symmetry. By exchanging $\cF$ with the $\bR$-family $x \mapsto \cF(x)$, we are equivalently asking whether $\cS$ can be deformed continuously, in an $\rM_{24}$-equivariant way, to an SQFT with 
 spontaneous supersymmetry breaking: whether $\cS$ is in the ``spontaneous-supersymmetry-breaking phase'' of SQFTs with $\rM_{24}$ flavour symmetry, or whether its supersymmetry is ``protected'' by the $\rM_{24}$-symmetry.

This question---whether some object can be continuously deformed into some other object---is the fundamental question of homotopy theory, and we will try to answer it by adopting homotopical techniques. Specifically, we will see that the space of SQFTs is not merely a topological space, but rather has extra structure making it into a ``spectrum.'' The algebraic topology of spectra is more rigid than the algebraic topology of spaces, and there are more tools available. The construction of a spectrum $\SQFT^\bullet$ described in this section is closely related to a construction in 
\cite{1610.00747} (see also
\cite{MR2079378,MR2742432,MR2742433}).

In order to describe this spectrum structure, we will need to discuss in a bit more detail the gravitational anomalies that $(1{+}1)$-dimensional (S)QFTs can enjoy. Specifically, we will distinguish two versions of the word ``quantum field theory,'' which we call ``absolute'' versus ``anomalous.''
For related recent discussion, see e.g.~\cite{MR3969923,me-TopologicalOrders}.

For a QFT to be \define{absolute}, it must come with extra data which is of debatable physical content. An absolute QFT has an absolutely-defined partition function, with no ambiguity about, say, the ``zero'' in the energy scale, or about the normalization of the path integral measure. An absolute QFT should have well-defined (super) Hilbert spaces, with no projectivity in the action by isometries: vectors in this Hilbert space have well-defined phases. Since the partition function is part of the data of an absolute QFT, symmetries of absolute QFTs never have 't~Hooft anomalies. The group of symmetries of an absolute quantum mechanics model is a subgroup of the unitary group rather than the projective unitary group. The usual functorial definition of topological QFTs, building on the original definition of~\cite{MR1001453}, is an attempt to model absolute TQFTs.

For comparison, an \define{anomalous} QFT is one that tolerates many phase ambiguities in its values. To tolerate an ambiguity in the meaning of ``zero'' energy, anomalous QFTs have ``partition functions'' that are not functions, but rather sections of possibly-nontrivial line bundles. To tolerate an ambiguity in the phase of a pure state, anomalous QFTs assign projective Hilbert spaces rather than honest Hilbert spaces. Symmetries of anomalous QFTs typically have nontrivial 't~Hooft anomalies. QFTs defined in terms of their algebras of operators are typically anomalous, and more data would to be needed in order to promote them to absolute QFTs. The simplest example of this is the Stone--von Neumann theorem, which says that an algebra of observables determines the Hilbert space functorialy only up to a phase ambiguity, i.e.\ only as a projective Hilbert space.

In the case of $(1{+}1)$-dimensional QFTs, the obstruction to promoting from anomalous to absolute is the \define{gravitational anomaly} $w = c_R - c_L$ mentioned in \S\ref{subsect.holomorphicanomaly}. If this anomaly vanishes, then there are still choices. For fermionic QFTs, it turns out that there are two choices (up to isomorphism), which differ by the parity of the Ramond-sector Hilbert space. (The parity of the Neveu--Schwarz sector is fixed by the state-operator correspondence.) There are further anomalies and choices if one wants to promote from anomalous to absolute in the presence of a symmetry. For example, symmetries of the operator algebra typically act only projectively, or more precisely via a spin lift, on the Ramond-sector, and there is the standard 't~Hooft anomaly cocycle. All together, the space of anomalies for fermionic $(1{+}1)$-dimensional QFTs is the $4$-layer spectrum  described in \S5.6 of~\cite{MR3978827} with homotopy groups $\bZ,\bZ_2,\bZ_2,0,\bZ$. This spectrum is called ``$\mathrm{fGP}^\times_{\leq 4}$'' in~\cite{MR3978827}, and the convention in that paper is that the homotopy groups live in degrees 
\begin{gather*}
(\mathrm{fGP}^\times_{\leq 4})^{3}(\pt) = \bZ,\qquad (\mathrm{fGP}^\times_{\leq 4})^{2}(\pt) = \bZ_2, \qquad (\mathrm{fGP}^\times_{\leq 4})^{1}(\pt) = \bZ_2,\\ (\mathrm{fGP}^\times_{\leq 4})^{0}(\pt) = 0, \qquad (\mathrm{fGP}^\times_{\leq 4})^{-1}(\pt) = \bZ,
\end{gather*} where by definition $(\mathrm{fGP}^\times_{\leq 4})^{\bullet}(\pt) = \pi_{-\bullet} \mathrm{fGP}^\times_{\leq 4}$, and, as mentioned already in \S\ref{subsec.notation}, we will try always to use cohomological degree conventions. The gravitational anomaly itself lives (after multiplication by $2$, since $c_R - c_L$ is a half-integer) in the $\bZ$ in cohomological degree $3 = (1{+}1) + 1$, and the $\bZ_2$ in cohomological degree $2$ corresponds to the two choices for promoting a QFT with vanishing gravitational anomaly to an absolute QFT.

The simplest example of a $(1{+}1)$-dimensional QFT  whose gravitational anomaly does not vanish is the theory $\Fer(1)$ of a single chiral Majorana--Weyl fermion. This is a holomorphic conformal field theory, with central charges $c_L = \frac12$ and $c_R = 0$. It can be made into an $\cN=(0,1)$ superconformal field theory by declaring that the supersymmetry operator acts trivially. Since $c_R \neq c_L$, this SCFT cannot be promoted to an absolute QFT. For example, its ``partition function'' is not a function, but rather a section of a nontrivial line bundle on the moduli space of spin Riemann surfaces called the \define{Pfaffian line}~\cite{MR915823,MR1186039}. (It is in fact best thought of as a bundle of superlines, with fibres isomorphic either to $\bC^{1|0}$ or $\bC^{0|1}$ depending on whether the spin Riemann surface is or is not the boundary of a spin 3-manifold.)
The tensor product (aka stacking) of $n$ copies of $\Fer(1)$ produces an $\cN=(0,1)$ SCFT $\Fer(n) = \Fer(1)^{\otimes n}$ with $(c_L,c_R) = (\frac n 2, 0)$.

\begin{definition}
$\SQFT^n$ is the space of compact unitary $\cN=(0,1)$ SQFTs whose anomaly is identified with the anomaly for $\Fer(n)$.
\end{definition}

For example, $\SQFT^0$ is the space of absolute SQFTs. The gravitational anomaly of $\cF \in \SQFT^n$ is $w = c_R - c_L = -\frac n 2$, but to give a point in $\SQFT^n$ requires more data than just an anomalous SQFT with this gravitational anomaly: one must give some ``parity'' information about the ``Ramond-sector Hilbert space'' of $\cF$, which is not a Hilbert space but rather an object of a possibly non-trivial category determined by $\Fer(n)$ (namely, the category of Ramond-sector vertex modules for the chiral algebra of $\Fer(n)$).

The symmetric group acts naturally on $\Fer(n)$ by permuting the constituent free fermions, and hence on $\SQFT^n$. Indeed, as an anomalous SQFT, $\Fer(n)$ carries an action by the orthogonal group $\rO(n)$. (The group acting on $\Fer(n)$ with trivialized 't~Hooft anomaly is $\Spin(n)$.) More generally, one can functorially define a holomorphic CFT $\Fer(V)$ for any real vector space $V$ with positive-definite inner product, and so we could have defined a space $\SQFT^V$ for any $V$, which is noncanonically isomorphic to $\SQFT^{\dim V}$.

There is a canonical isomorphism \cite{MR2742433}
$$ \Fer(V) \otimes \Fer(W) \cong \Fer(V\oplus W).$$
This implies that tensor product (stacking) of SQFTs provides a commutative and associative operation
$$ \otimes : \SQFT^V \times \SQFT^W \to \SQFT^{V\oplus W}$$
which is compatible with the actions by $\rO(V) \times \rO(W) \subset \rO(V\oplus W)$. We warn that the ``commutativity'' is subtle. Indeed, given $\cF \in \SQFT^V$ and $\cF' \in \SQFT^W$, to compare $\cF \otimes \cF'$ with $\cF' \otimes \cF$, one must use the isomorphism $\SQFT^{V\oplus W} \cong \SQFT^{W\oplus V}$ coming from the isomorphism $\Fer(V\oplus W) \cong \Fer(W \otimes V)$ that permutes the fermions. Even if $V = W$, this isomorphism is nontrivial, and may have a nontrivial anomaly.

Thus one can think of $\SQFT^\bullet$ all together as a sort of ``graded commutative monoid.'' With a bit of work, one can define a direct sum operation on each $\SQFT^n$, so that $\SQFT^\bullet$ is a graded commutative ring-without-negation. Rather than trying to define direct sums directly, we will see that each $\SQFT^n$ is in fact a commutative group up to homotopy: we will give $\SQFT^\bullet$ the structure of a (commutative orthogonal) $\Omega$-spectrum, with one small modification.

By definition, an \define{$\Omega$-spectrum} $\cE^\bullet$ is a sequence of spaces $\cE^0,\cE^1,\cE^2,\dots$, each equipped with a basepoint $0 \in \cE^n$, together with homotopy equivalences $\cE^n \isom \Omega\cE^{n+1}$, where $\Omega\cE^{n+1}$ means the space of loops in $\cE^{n+1}$ that start and end at the basepoint $0$. 
In particular, each $\cE^n$ is an infinite loop space (i.e.\ a homotopically coherent abelian group). The spectrum $\cE^\bullet$
 is \define{orthogonal} when the grading is not just by integers but by real vector spaces $V$ as above, and the homotopy equivalence $\cE^V \isom \Omega^k \cE^{V\oplus\bR^k}$ is compatible with the $\rO(V)\times \rO(k)$ action. Let $X$ be a space. The \define{$\cE^\bullet$-cohomology of $X$}
is by definition
$$ \cE^\bullet(X) = \pi_0 \operatorname{maps}(X, \cE^{\bullet}).$$
This is an abelian group because of the 
homotopy equivalence $\cE^n \isom \Omega\cE^{n+1}$, which provides, for any $s \geq 0$, an isomorphism
$$ \cE^\bullet(X) \cong \pi_s \operatorname{maps}(X, \cE^{\bullet+s}).$$

For our spectrum $\SQFT^\bullet$, we want to choose the basepoint $0 \in \SQFT^n$ to be the ``zero QFT.'' This is the TQFT that assigns ``$0$'' to every nonempty input: its partition function is zero, its Hilbert space is zero-dimensional, etc. This can be thought of as having any anomaly that one so chooses. When a physicist says ``supersymmetry is spontaneously broken in $\cF$,'' a mathematician should hear ``$\cF$ flows to $0$ under RG flow,'' where ``RG flow'' is a canonically-defined action by the monoid $\bR_{\geq 0}$ on $\SQFT^n$ (and, debateably, by the group $\bR$), and ``$\cF$ flows to $\cF_{\mathrm{IR}}$'' means that $\cF_{\mathrm{IR}}$ is the limit of the RG-flow starting at $\cF$. 

But some physicists will rightly quibble with the idea of the ``zero QFT,'' and will instead take the phrase ``supersymmetry is spontaneously broken'' as a primitive notion. Moreover, some mathematical attempts to define the notion of ``quantum field theory,'' including the functorial ones suggested in~\cite{MR2079378}, include this zero QFT as a point on $\SQFT^n$, but others of a more operator-algebraic nature (e.g.~\cite{MR2742433}) do not. (Indeed, if one follows the ideas of~\cite{MR2742433}, then the definition of $\SQFT^n$ should be the space of operator-algebraically-defined SQFTs equipped with a Morita equivalence to $\Fer(n)$. There is a ``zero'' operator algebra, but it is not Morita equivalent to a nonzero algebra.)
For this reason, we will modify our notion of spectrum to tolerate a subspace of basepoints, rather than a single basepoint. The loop space $\Omega\cE^{n+1}$ then should consist of loops that begin and end inside this subspace, and the homotopy groups defining $\cE^\bullet$-cohomology should be relative homotopy groups. Otherwise there is no real difference. And if the reader's model of ``quantum field theory'' includes the zero QFT, then the reader may use the usual notion of $\Omega$-spectrum in what follows.

Let us parameterize paths by the real line $\bR$: a point in $\Omega\SQFT^n$ is an $\bR$-family $x \mapsto \cF(x)$ in $\SQFT^n$ such that supersymmetry is spontaneously broken for all $x \ll 0$ and for all $x \gg 0$. (Or, for those readers who have a zero QFT, use instead families that approach $0$ as $x \to \pm \infty$. One can promote the former type of loop to the latter by turning on an RG flow whose strength increases as $x \to \pm \infty$.)

Then the map $\SQFT^n \to \Omega\SQFT^{n+1}$ couldn't be simpler. As above, let $\Fer(1)$ denote the holomorphic CFT of a single chiral Majorana--Weyl fermion $\lambda$. Above we promoted this CFT to an $\cN=(0,1)$ SCFT by declaring that the supersymmetry operator was $0$. But, at the cost of conformal invariance, we may give it other $\cN=(0,1)$ structures. Specifically, the supercurrent $\bar{Q} = x\lambda$ defines an $\cN=(0,1)$ supersymmetry on $\Fer(1)$, which is not superconformal. Let $\Fer(1)(x)$ denote the SQFT $(\Fer(1), \bar{Q} = x\lambda)$. Comparing with \S\ref{subsect.holomorphicanomaly}, $\Fer(1)(x)$ is exactly the ``fibre'' over $-x$ of the operator $\Phi = 0$ in  the vacuum theory $1 \in \SQFT^0$ (with one-dimensional Hilbert space and partition function identically equal to $1$).

Then the map $\SQFT^n \to \Omega\SQFT^{n+1}$ is:
$$ \cF \mapsto \cF \otimes \Fer(1)(x).$$
If $x \neq 0$, supersymmetry spontaneously breaks in $\Fer(1)(x)$ and hence in $\cF \otimes \Fer(1)(x)$. So this family $x \mapsto \cF \otimes \Fer(1)(x)$ is indeed a point in $\Omega\SQFT^{n+1}$. (In fact, it is a point even for the stricter version where ``$0$'' is a meaningful QFT: the action of RG flow on $\Fer(1)(x)$ simply rescales $x \mapsto e^s x$, where $s \to \infty$ is the IR limit, and so the $x \to \pm \infty$ limits of $\Fer(1)(x)$ are both the zero QFT.)

We must now prove that this map $\cF \mapsto \cF \otimes \Fer(1)(x)$ is a homotopy equivalence. Consider the ``dynamicalization'' map $\Omega\SQFT^{n+1} \to \SQFT^n$  that takes a family $x \mapsto \cF(x)$ in $\Omega\SQFT^{n+1}$ and promotes the parameter $x$ to a dynamical scalar multiplet, producing the SQFT that, as in \S\ref{subsect.holomorphicanomaly}, we will call $\int_{\phi,\psi} \cF(\phi)$. We claim without proof that $\int_{\phi,\psi} \cF(\phi)$ is compact for $(x \mapsto \cF(x)) \in \Omega\SQFT^{n+1}$: the justification is that, since supersymmetry is spontaneously broken, this is essentially a compactly-supported family; but more work would need to be done to justify this claim, and one may have to first modify the family by RG-flowing $\cF(x)$ by some finite amount that grows as $x \to \pm \infty$.

To prove that $\cF \mapsto \{x \mapsto \cF \otimes \Fer(1)(x)\}$ is a homotopy equivalence, it suffices to prove that its two compositions with $\int_{\phi,\psi}$ are both homotopic to the identity. We do not need to confirm any higher homotopy coherence: in particular, we do not need to show that the homotopies to the identity are compatible in any way. (We would need to prove such compatibilities if we wanted to claim that $\int_{\phi,\psi}$ was the homotopy-coherent inverse to tensoring with $\Fer(1)(-)$.)

First, consider the composition
$$ \cF \mapsto \cF \otimes \Fer(1)(x) \mapsto \int_{\phi,\psi} \cF \otimes \Fer(1)(\phi).$$
The copy of $\cF$ comes out of the integral, and so it suffices to show that $\int_{\phi,\psi} \Fer(1)(\phi)$ is continuously deformable to the vacuum theory $1 \in \SQFT^0$. This is a special case of the philosophy mentioned in \S\ref{subsect.holomorphicanomaly} that the total space of a family should be recoverable from dynamicalizing the parameter. In this case, the SQFT $\int_{\phi,\psi} \Fer(1)(\phi)$ contains the following fields. First, there is the chiral fermion $\lambda \in \Fer(1)$. Next, there is a full scalar boson~$\phi$, which is the bosonic component of the superfield that dynamicalizes $x$. Finally, there is the superpartner $\psi$ of $\phi$, which is an antichiral fermion. 
The supersymmetry operator in components is
$$ \bar{Q} = \lambda \bar\partial \phi + \psi \partial \phi.$$
The first summand is from the supersymmetry $\lambda x$ in $\Fer(1)(x)$, and the second summand says that $\psi$ is the superpartner of $\phi$.
The Lagrangian contains the standard massless terms $\lambda\bar\partial\lambda$, $\phi \Delta\phi$, and $\psi \partial\psi$. It also contains a correction coming from $\bar{Q}$, which ends up being $\lambda \psi + \phi^2$. (The Lagrangian for $\Fer(1)(x)$ had a correction like $\lambda \bar{Q}[x] + x^2$, and when we replace $x$ with $\phi$, $\bar{Q}[x]$ becomes $\psi$.) All together, we can recognize $\int_{\phi,\psi} \Fer(1)(\phi)$ as the free theory consisting of a massive Majorana fermion and a massive scalar boson. This free theory is well-known to flow to the trivial vacuum theory in the IR, which is to say that RG flow implements a homotopy $\int_{\phi,\psi} \Fer(1)(\phi) \simeq 1$.

The other composition is
$$ \cF(x) \mapsto \int_{\phi,\psi} \cF(\phi) \mapsto \int_{\phi,\psi} \cF(\phi) \otimes \Fer(1)(x).$$
We have not tried to be precise about the meaning of ``family of SQFTs.'' For the purposes of this article, let us suppose that the field content (and any other kinematical information) of $\cF(x)$ is independent of $x$, and only the Lagrangian and the supersymmetry (and any other dynamical information) varies with $x$. This is not unreasonable: if there is a field that exists only for certain values of $x$, one can extend it to a field that exists for all $x$ but is very massive except at the values of $x$ for which it was earlier defined. Assuming we have topologized the space of SQFTs in a way that cares primarily about the effective low-energy field theory, 
turning on very massive fields should be a very small deformation, and so should not change the homotopy type of the family $\cF(-)$. Then the field content on the right-hand side consists of: the fields in $\cF(-)$, the scalar $\phi$, the antichiral fermion $\psi$, and the chiral fermion $\lambda$. Writing $L_{\mathrm{LHS}}(x)$ and $\bar{Q}_{\mathrm{LHS}}(x)$ for the Lagrangian and supersymmetry operators in $\cF(x)$, and writing $\bar{Q}_{\mathrm{LHS}}'(x) = \frac{\partial}{\partial x}\bar{Q}_{\mathrm{LHS}}$, the Lagrangian and supersymmetry operators on the right-hand side are:
\begin{gather*}
  L_{\mathrm{RHS}} = L_{\mathrm{LHS}}(\phi) + \bar{Q}_{\mathrm{LHS}}'(\phi) \partial\psi  + \phi \Delta \phi +  \psi \partial \psi + \lambda \bar\partial \lambda, \qquad
  \bar{Q}_{\mathrm{RHS}} = \bar{Q}_{\mathrm{LHS}}(\phi) + \psi \partial \phi + x\lambda.
\end{gather*}
Now consider deforming this SQFT by the superpotential $W = f(\phi)\lambda$ for some polynomial $f \in \bR[x]$. This deformation is allowed: it does not destroy compactness, nor does it destroy the spontaneous supersymmetry breaking. The deformation changes the Lagrangian to:
$$ L_{\mathrm{deformed}} = L_{\mathrm{RHS}} + \bar{Q}_{\mathrm{RHS}}[W] + \left(\frac{\partial W}{\partial\lambda}\right)^2.$$
Since the original $\bar{Q}_{\mathrm{LHS}}(x)$ was a function in $x$, neither $W$ nor $\bar{Q}_{\mathrm{LHS}}(\phi)$ have any $\partial\phi$-dependence, and so commute. Thus we have:
$$ L_{\mathrm{deformed}} = L_{\mathrm{RHS}} + \psi f'(\phi) \lambda + x  f(\phi) + f(\phi)^2.$$
Taking $f(\phi) = -2\phi$ gives
$$ L_{\mathrm{deformed}} = \left[ L_{\mathrm{LHS}}(\phi) + \bar{Q}_{\mathrm{LHS}}'(\phi) \partial\psi\right] + \left[ \phi \Delta \phi +  \psi \partial \psi + \lambda \bar\partial \lambda + \psi \lambda + -2x \phi + \phi^2 \right].$$
Focusing on the second bracketed expression, we see that $\phi$ now has a mass with vacuum expectation value $x$ and the full fermion $(\lambda,\psi)$ is also massive. So, performing the path integral in those variables first, the IR behaviour of the deformed theory is described simply by setting $\phi = x$ and $\lambda=\psi = 0$, and we recover the original theory $\cF(x)$.

In summary, we have outlined a proof of the following result. We call a ``conjecture'' because we did not attempt to mathematically define or topologize the spaces $\SQFT^\bullet$, and because even at a physicists' level of rigour we left some questions about the details of these spaces. 
%A version of this result (with symmetric spectra in place of orthogonal spectra, and with a rather different explanation of the spectrum structure) is proposed in~\cite{MR2742432}, which also proposes a detailed mathematical definition of the spaces $\SQFT^n$.
\begin{conjecture}
  The spaces $\SQFT^V$ of compact unitary $\cN=(0,1)$ SQFTs with anomaly identified with the anomaly of $\Fer(V)$ compile into a commutative ring orthogonal $\Omega$-spectrum $\SQFT^\bullet$. \qed
\end{conjecture}

\subsection{Equivariant $\SQFT^\bullet$ and 't~Hooft anomalies} \label{subsec.anomalies}

Let $G$ be a finite group (or a Lie group or \dots, but we will need only the finite group case).
 The discussion in the previous section applies equally well if one considers SQFTs, and families thereof, which are equipped with a nonanomalous $G$-flavour symmetry. The corresponding spectrum $\SQFT^\bullet_G$ is a $G$-equivariant enhancement of the spectrum $\SQFT^\bullet$: using it, one can assign cohomology groups to spaces equipped with $G$-action. The fundamental reason that $\SQFT^\bullet$ admits an equivariant enhancement is that SQFTs admit automorphisms, and so the collection $\SQFT^n$ of SQFTs with a given gravitational anomaly is not merely a space, but rather a groupoid or stack. 
 (The groupoidal/stacky approach to genuinely equivariant spectra is formalized in \cite{0701916}.)
 If $\cX$ is any stack, then we can consider the space of maps of stacks from $\cX$ to $\SQFT^n$, and evaluate its homotopy groups. Taking $\cX = \bB G$ the classifying stack of the group $G$ gives:
\begin{gather*}
\SQFT^\bullet_G = \operatorname{maps}(\bB G, \SQFT^\bullet), \\ \SQFT^\bullet(\bB G) = \SQFT^\bullet_G(\pt) = \pi_0 \operatorname{maps}(\bB G, \SQFT^\bullet).\end{gather*}

We may also consider SQFTs with $G$-flavour symmetry and prescribed 't~Hooft anomaly~$\omega$. These compile into an orthogonal $\Omega$-spectrum $\SQFT^\bullet_{G,\omega}$.
Because 't~Hooft anomalies add under stacking (i.e.\ under tensor product of SQFTs), $\SQFT^\bullet_{G,\omega}$ is not a ring spectrum, but it is a module spectrum for $\SQFT^\bullet_{G}$.
The algebrotopologists' name for introducing an 't~Hooft anomaly $\omega$ is \define{twisting}: the homotopy groups 
$$ \SQFT^\bullet_{G,\omega}(\pt) = \SQFT^\bullet_\omega(\bB G) = \pi_0  \SQFT^\bullet_{G,\omega}$$
are the ``$\omega$-twisted $G$-equivariant $\SQFT^\bullet$-cohomology of a point.''

Where do 't~Hooft anomalies live? For $(1{+}1)$-dimensional fermionic QFTs, they live in the (extended) \define{supercohomology} $\SH^3(\bB G)$ of the group $G$~\cite{GuWen,WangGu2017}, which is the 3-layer spectrum described in \S5.4 of~\cite{MR3978827} with homotopy groups $$\SH^2(\pt) = \bZ_2, \qquad \SH^1(\pt) = \bZ_2, \qquad \SH^0(\pt) = 0, \qquad \SH^{-1}(\pt) = \bZ,$$ and Postnikov $k$-invariants $\Sq^2 : \bZ_2 \to \bZ_2$ and $\Box_\bZ \circ \Sq^2 : \bZ_2 \to \bZ$, where $\Sq^2$ denotes the second Steenrod squaring operation and $\Box_\bZ$ denotes the integral Bockstein (for the short exact sequence $\bZ \to \bZ \to \bZ_2$).

By definition there is a map $\H^{\bullet+1}(G;\bZ) \to \SH^\bullet(\bB G)$. A long exact sequence shows that this map is an injection in degree $\bullet \leq 3$ (but typically not for larger values of $\bullet$). It is a surjection if $\H^{\bullet-1}(G;\bZ_2)$ and $\H^{\bullet-2}(G;\bZ_2)$ both vanish. In particular, if $G$ is the Schur cover of simple group, then $\SH^3(\bB G) = \H^4(G;\bZ)$.

Actually, as emphasized in~\cite{MR3978827}, it is best to think of $\SH^3(G)$ instead as the \define{reduced} group cohomology of $G$ with coefficients in the 4-level spectrum $\mathrm{fGP}^\times_{\leq 4}$ mentioned in~\S\ref{subsec.spectrum}, since the basepoint $\pt \to \bB G$ gives a canonical isomorphism
$$ (\mathrm{fGP}^\times_{\leq 4})^3(\bB G) = \SH^3(\bB G) \oplus \SH^3(\pt),$$
and $\SH^3(\pt) \cong \bZ$ indexes the gravitational anomaly $n = 2(c_L - c_R)$.
This is consistent with the general story of twistings of cohomology theories: $\mathrm{fGP}^\times_{\leq 4}$ controls all the anomalies, both 't~Hooft and gravitational, for the spectrum $\SQFT^\bullet$, and so algebrotopologists sometimes write the twisted cohomology groups as $\SQFT^{\bullet+\omega}_G(-)$, with $\bullet+\omega$ a total class in $(\mathrm{fGP}^\times_{\leq 4})^3(\bB G)$.

To build $\SQFT^\bullet_{G,\omega}$ completely correctly, one should rigidify the 't~Hooft anomaly by choosing some representative SQFT $\cV_\omega$ with anomaly $\omega$ and then asking that points in $\SQFT^\bullet_{G,\omega}$ have 't~Hooft anomaly identified with the anomaly of $\cV_\omega$. If one just asks that the anomaly of an SQFT $\cF$ be equal to $\omega$ in cohomology, then there is an ambiguity in this identification (parameterized by the reduced cohomology group $\widetilde{\SH}{}^2(\bB G)$), analogous to the ambiguity in promoting an anomalous SQFT with $c_L = c_R$ to an absolute SQFT. 
Given that we modelled $\SQFT^\bullet$ as an orthogonal spectrum, identifying gravitational anomalies with anomalies of $\Fer(V)$ for real vector spaces $V$, one might try now to set $\cV_\omega \overset?= \Fer(V)$ for some real representation $V : G \to \rO(n)$. The anomaly of $G$ acting on $\Fer(V)$ is a characteristic class $\frac{p_1}2(V) \in \SH^3(\bB G)$ called the \define{fractional Pontryagin class}. (It is in the image of $\H^4(G;\bZ)$ whenever the representation $V$ is Spin.)

 Thus one can ask: How many of the classes in $\SH^3(\bB G)$ arises as fractional Pontryagin classes of real representations? This is a supercohomological version of the classical question of understanding which classes in $\H^{\mathrm{ev}}(G;\bZ)$ arise as Chern classes of complex representations. The answer is: it depends on the group $G$. 
 To illustrate this, consider the case when $G$ is a Schur cover of a sporadic group. The calculations in~\cite{JFT,MR3990846,MR3916985} show that $\SH^3(\bB G) = \H^4(G;\bZ)$ 
vanishes if $G$ is one of
$$\{\rM_{23}, 3\mathrm{McL}, \rJ_4, \mathrm{Ly}\}$$
and does not vanish but
consists entirely of fractional Pontryagin classes for $G$ in
$$ \{\rM_{11}, 2\rM_{12}, 6\rM_{22}, \rM_{24}, 2\rJ_2, \mathrm{Co}_3, \mathrm{Co}_2, 2\mathrm{Co}_1, 6\mathrm{Suz}, 3\rJ_3, \mathrm{He}\}.$$
On the other hand, for the groups $G$ in
$$ \{2\mathrm{HS}, \mathrm{Mon}\},$$
it is shown in those papers that
$\SH^3(\bB G) = \H^4(G;\bZ)$ is  \emph{not} generated by fractional Pontryagin classes. The calculations for the other sporadic groups have not been completed.

This might lead the reader to worry that perhaps there is no good representative $\cV_\omega$ in general. Fortunately, the main result of~\cite{EvansGannon} is that there is, for any finite group $G$ and anomaly $\omega \in \H^4(G;\bZ)$, a bosonic holomorphic conformal field theory $\cV_\omega$ with $G$-flavour symmetry and 't~Hooft anomaly~$\omega$.
The CFT $\cV_\omega$ is not canonical, and is of very high central charge. 
Although~\cite{EvansGannon} focuses on bosonic CFTs, the construction extends to the fermionic case for any $\omega \in \SH^3(\bB G)$. Any holomorphic conformal field theory can be thought of as an $\cN=(0,1)$ superconformal field theory by simply declaring the supersymmetry operator to be trivial. Thus we can construct representatives $\cV_\omega$ as required.

The two examples from the end of \S\ref{subsect.holomorphicanomaly}, $\overline{\Fer(3)}{}^{\oplus 24}$ and $\overline{V^{f\natural}} \otimes \overline{\Fer(3)}$, each carry natural $\rM_{24}$-actions. The action on the former permutes the $24$ summands, and so (as in the introduction) we will call it $\Perm \otimes \overline{\Fer(3)}$; we will write ``$\Perm$'' for both the standard degree-24 permutation representation of $\rM_{24}$ as well as its enhancement to a TQFT with $24$ massive vacua permuted by $\rM_{24}$.
 The action on the latter is the restriction of the action by $\Co_1 = \Aut(\overline{V^{f\natural}})$.
Thus we find classes $[\Perm \otimes \overline{\Fer(3)}] \in \SQFT^{-3}_{\omega}(\bB \rM_{24})$ and $[\overline{V^{f\natural}} \otimes \overline{\Fer(3)}] \in \SQFT^{-27}_{\omega'}(\bB \rM_{24})$, where $\omega,\omega'$ are the 't~Hooft anomalies of the various actions. In both cases the action of $\rM_{24}$ on $\overline{\Fer(3)}$ is trivial, and so these classes are the products of classes $[\Perm] \in \SQFT^{0}_{\omega}(\bB \rM_{24})$ and $[\overline{V^{f\natural}}] \in \SQFT^{-24}_{\omega'}(\bB \rM_{24})$ by the class $[\overline{\Fer(3)}] \in \SQFT^{-3}(\pt)$. (The ring $\SQFT^{\bullet}(\pt)$ acts on all twisted equivariant cohomologies $\SQFT^\bullet_{\omega}(X)$.)
One of the main results of~\cite{WittenTMF} is that there is a well-defined class in $\SQFT^{-n}(\pt)$ for each cobordism class of $n$-dimensional manifolds with String structure (and in particular for each class of $n$-dimensional framed manifolds), and the 3-sphere $S^3 = \mathrm{SU}(2)$ with its Lie group framing determines the class $[\overline{\Fer(3)}]$. In algebraic topology, generalized cohomology classes determined by $S^3$-with-its-Lie-group-framing are conventionally named ``$\nu$.'' Following this convention, we can write our SCFTs as
$$[\Perm] \nu \in \SQFT^{-3}_\omega(\bB \rM_{24}), \qquad [\overline{V^{f\natural}}] \nu \in \SQFT^{-27}_{\omega'}(\bB \rM_{24}).$$

What are the 't~Hooft anomalies $\omega,\omega'$? The answer in the former case is complicated, and so we will do it second. For $[\overline{V^{f\natural}}] \nu$, the anomaly of $\rM_{24}$ is restricted from the anomaly of the $\Co_1$-action on $\overline{V^{f\natural}}$. The main result of~\cite{JFT} implies that $\SH^3(\Co_1) \cong \bZ_{24}$ is cyclic, generated by the {fractional Pontryagin class} $\frac{p_1}2$ of the $24$-dimensional projective representation of $\Co_1$. (The paper~\cite{JFT} calculates instead the ordinary cohomology of the Schur cover $\Co_0 = 2.\Co_1$, but it is not hard to show that the canonical maps $\SH^3(\Co_1) \to \SH^3(\Co_0)$ and $\H^4(\Co_0; \bZ) \to \SH^3(\Co_0)$ are both isomorphisms.) Up to a sign convention in the definition of ``'t~Hooft anomaly,'' $\frac{p_1}2$ is precisely the anomaly of $\Co_1$ acting on $V^{f\natural}$~\cite{MR3916985}. Furthermore, Theorem~8.1 of \cite{JFT} asserts that, upon restriction to $\rM_{24} \subset \Co_1$, this anomaly restricts to $-\alpha$, where $\alpha$ is the anomaly of Mathieu Moonshine computed in~\cite{MR3108775}.
Actually, since different authors might reasonably disagree on the sign of ``the anomaly,'' it is useful that~\cite{MR3832169} has compared the multipliers for some elements acting in Mathieu Moonshine versus in $V^{f\natural}$ (see Table 3 therein). Multipliers depend linearly on the anomaly, and in all cases checked the anomaly for $V^{f\natural}$ restricts to minus the anomaly from~\cite{MR3108775}. Together with the computer calculation $\H^4(\rM_{24};\bZ) = \bZ_{12}$ from~\cite{SE09} (confirmed using elementary methods in Theorem 5.1 of~\cite{JFT}), and further calculations from~\cite{MR3108775}, these comparisons are enough to establish that 
$$ (\text{anomaly of }V^{f\natural})|_{\rM_{24}} = -(\text{anomaly of Mathieu Moonshine}). $$
But the anomalies of $V^{f\natural}$ and $\overline{V^{f\natural}}$ have opposite signs. Writing $\alpha$ for the Mathieu Moonshine anomaly from~\cite{MR3108775}, we find:
$$ [\overline{V^{f\natural}}] \nu \in \SQFT^{-27}_{\alpha}(\bB \rM_{24}).$$

Turning to $[\Perm] \nu \in \SQFT^{-3}_\omega(\bB \rM_{24})$, we must answer the question: What is the 't~Hooft anomaly of the $\rM_{24}$-action on the TQFT $\Perm$? The na\"ive answer, ``zero,'' misses an important subtlety, which is that the question is badly posed: 't~Hooft says that there is an \define{anomaly} when a partition function or other datum, which was expected to be $G$-invariant, in fact changes by a phase; but in our case those data are often zero because of the vacuum degeneracy. More precisely, the $\rM_{24}$-symmetry on $\Perm$ spontaneously breaks to a trivial $\rM_{23}$-symmetry. This trivial $\rM_{23}$-symmetry is definitely nonanomalous. But we may consider the total $\rM_{24}$-symmetry to have any anomaly that we choose in the kernel of $\SH^3(\rM_{24}) = \H^4(\rM_{24};\bZ) \to \SH^3(\rM_{23}) = \H^4(\rM_{23};\bZ)$. Remarkably, $\H^4(\rM_{23};\bZ) = 0$~\cite{MR1736514} (that paper in fact shows that $\H^\bullet(\rM_{23};\bZ) = 0$ for $\bullet \leq 5$, and provides further information about $\H^\bullet(\rM_{23};\bZ)$; the low-cohomology results are confirmed computationally in~\cite{SE09}, and the $\H^4$ calculation is confirmed with elementary methods in~\cite{JFT}). Thus we may consider the $\rM_{24}$-action on $[\Perm] \nu$ as having any anomaly that we want:
$$ [\Perm] \nu \in \SQFT^{-27}_{\omega}(\bB \rM_{24}) \text{ for any desired } \omega \in \SH^3(\rM_{24}) = \H^4(\rM_{24};\bZ) = \bZ_{12}.$$

The same argument can be rephrased algebrotopologically in terms of \define{pushforwards}. 
If $f : H \to G$ is a homomorphism of finite groups, then an SQFT with $G$-symmetry and 't~Hooft anomaly $\omega \in \SH^3(\bB G)$ determines, by forgetting some information, an SQFT with $H$-symmetry and 't~Hooft anomaly $f^*\omega \in \SH^3(\bB H)$. This provides a \define{pullback} map
$$ f^* : \SQFT^\bullet_{G,\omega} \to \SQFT^\bullet_{H,f^*\omega}.$$
This map has an ``adjoint'' $f_* : \SQFT^\bullet_{H,f^*\omega} \to \SQFT^\bullet_{G,\omega}$. 
To construct it, note that any map $f : H \to G$ factors canonically as a surjection followed by an injection:
$$ H \epi \im(f) \mono G.$$
Thus it suffices to describe $f_*$ when $f : H \to G$ is either surjective or injective.

Suppose first that $f : H \to G$ is a surjection with kernel $K = \ker(f)$. Then $f^*\omega \in \SH^3(\bB H)$ restricts trivially to $K$, and so if $\cF$ is an SQFT with $H$-symmetry, the $K$-action is nonanomalous and may be gauged. Furthermore, because the anomaly $f^*\omega$ of the $H$-action is pulled back from $G$, there is no ``mixed anomaly.'' It follows that the gauged theory $\cF \sslash K$ carries a $G$-action with anomaly $\omega$. 
The pushforward map $f_*$ is
$$ f_*:\cF \mapsto \cF \sslash K.$$
Note the repeated use of the fact that the anomaly is pulled back from $G$. If all we knew was that $H$ acted on $\cF$ with some anomaly $\omega' \in \SH^3(H)$, and that $\omega'|_K = 0$, then there would be an ambiguity in the meaning of the gauged theory: there would be $\widetilde{\SH}{}^2(K)$-many theories that deserve the name ``$\cF \sslash K$,'' parameterized by the $\widetilde{\SH}{}^2(K)$-many trivializations of $\omega'|_K$. (Here and throughout, $\widetilde{\SH}{}^\bullet$ denotes reduced supercohomology.) 
In our case,
we can choose a canonical gauging because, since $f^*\omega$ is restricted from $G$, it trivializes canonically on $K$.
Gauging  uses up the $K$-symmetry, but produces a new ``magnetic dual'' action by $\widetilde{\SH}{}^1(K) \cong \hom(K,\rU(1))$, and in general the remaining $G$-action could be extended by this symmetry (c.f.~\cite{BT2017} or \S2.3 of~\cite{MR3916985}). 
In our case 
the extension is trivial because $f^*\omega$ is pulled back from $G$. If the extension were trivializable but not canonically so, then the different trivializations might lead to different $G$-actions with different anomalies. But, again because $f^*\omega$ is pulled back from $G$, the extension is canonically trivializable, and the resulting $G$-action has anomaly $\omega$.

Suppose now that $f : H \to G$ is an injection, and let $X = G/H$ denote the space of cosets. If $\cF$ is an SQFT with $H$-symmetry and anomaly $f^*\omega$, then the direct sum (aka disjoint union) of $X$-many copies of $\cF$ can be given a $G$-action that permutes the copies, and that acts as $H$ on each copy. Physically, this is an SQFT where the $G$-symmetry spontaneously breaks to an $H$-symmetry. This is the pushforward map.

In both cases, the pushforward $f_*$ can be described as a type of finite path integral. Indeed, gauging a $K$-symmetry is the same as integrating over $K$-gauge fields, which are maps from the worldsheet to $\bB K$, which is the fibre of $\bB H \to \bB G$ in the case when $H \to G$ is an injection.  
When $H \to G$ is a surjection, then the fibre of $\bB H \to \bB G$ is the set $X$, and again we are taking an integral over maps to this fibre. 
This explains the general structure: $f_*$ implements a finite path integral over the space of maps from the worldsheet to the fibre $X$ of the map $\bB H \to \bB G$.

As an example, suppose that $f : H \mono G$ is an inclusion, and $\omega \in \SH^3(G)$ is an anomaly such that $f^*\omega = 0 \in \SH^3(H)$. Then we have a pushforward map
$$ f_* : \SQFT^\bullet(\bB H) \to \SQFT^\bullet_\omega(\bB G).$$
The domain is a commutative ring (the codomain is not, if $\omega \neq 0$), with unit $1 \in \SQFT^\bullet(\bB H)$ represented by the trivial ``vacuum'' SQFT with trivial $H$-symmetry. The pushforward $f_*(1)$ is simply the $(1{+}1)$-dimensional TQFT with $X = G/H$ many ground states, permuted by the $G$-symmetry, and no other structure. (In terms of functorial field theories valued in the 2-category of algebras and bimodules, $f_*(1)$ corresponds to the algebra $\bigoplus_X \bC$.) For $G = \rM_{24}$ and $H = \rM_{23}$, this is the TQFT that we called ``$\Perm$'' above.
If instead we had chosen some $\cF \in \SQFT^\bullet(\pt)$, equipped with the trivial $H$-action (equivalently, pulled back along $\bB H \to \pt$), then $f_*(\cF) = f_*(1) \otimes \cF = \Perm \otimes \cF$.

For the purposes of explaining Mathieu Moonshine, this looks pretty good. When restricted along $\pt \subset \rM_{24}$, the SQFT $\Perm \otimes \overline{\Fer(3)}$ restricts to $\overline{\Fer(3)}{}^{\oplus 24}$, which we already saw is nullhomotopic and produces the mock modular form $H(\tau)$. If $\Perm \otimes \overline{\Fer(3)}$ were $\rM_{24}$-equivariantly nullhomotopic, then we would produce mock modular forms $H_{g,h}(\tau)$ as desired. Unfortunately, it is not:

\begin{proposition}\label{prop.permDNW}
  $\Perm \otimes \overline{\Fer(3)}$ is not $\rM_{24}$-equivariantly nullhomotopic, for any 't~Hooft anomaly $\omega$.
\end{proposition}
\begin{proof}
  If $[\Perm \otimes \overline{\Fer(3)}] = [\Perm] \otimes \nu$ were trivial in $\SQFT^{-3}_\omega(\bB \rM_{24})$, then its restriction to $\rM_{23}$ would also be trivial. Since $\SH^3(\rM_{23}) = 0$, this restriction has trivial gauge anomaly, and so we may gauge the $\rM_{23}$-action. If $\Perm \otimes \overline{\Fer(3)}$ were equivariantly nullhomotopic, then so would be this gauged theory (by gauging the $\rM_{23}$-action on the nullhomotopy).
   In algebrotopological language, writing $f : \rM_{23} \mono \rM_{24}$ and $p : \rM_{23} \to \pt$, we wish to compute $p_* f^* f_* 1 \otimes \nu$.
   
   This is purely a TQFT computation. 
   %For the remainder of the proof, we will write $\mathbf{24}$ for the degree-24 permutation representaion of $\rM_{24}$, and for its restrictions to subgroups thereof. 
   Our goal is to compute the $(1{+}1)$-dimensional TQFT which counts maps from the worldsheet to the quotient stack
   $$ \mathbf{24} \sslash \rM_{23}.$$
   But, restricted to $\rM_{23}$, $\mathbf{24}$ splits as $\mathbf{1} \sqcup \mathbf{23}$, and so
   $$ \mathbf{24} \sslash \rM_{23} = \bB \rM_{23} \sqcup \mathbf{23} \sslash \rM_{23} = \bB \rM_{23} \sqcup \bB \rM_{22}.$$
   In other words, the TQFT $p_* f^* f_* p^* 1$ is the direct sum of two TQFTs: pure gauge theory for $\rM_{23}$, and pure gauge theory for $\rM_{22}$.
   
   For any finite group $G$, pure $G$-gauge theory in $(1{+}1)$-dimensions is described by the group algebra $\bC[G]$ of $G$, which is Morita equivalent to the direct sum of $\#(G/G)$-many copies of $\bC$, where $\#(G/G)$ means the number of conjugacy classes in $G$. The group $\rM_{23}$ has $17$ conjugacy classes, and the group $\rM_{22}$ has $12$ conjugacy classes. Thus
   $$ p_* f^* f_* p^* 1 = 17 + 12 = 29,$$
   and so $p_* f^* f_* p^* 1 \otimes \nu = 29\nu$, represented by the SQFT $\overline{\Fer(3)}{}^{\oplus 29}$. But $29$ is not divisible by $24$, and so $\overline{\Fer(3)}{}^{\oplus 29}$ is not nullhomotopic by~\cite{GJFmock}.
   
   One could wonder if perhaps the day would be saved by somehow squeezing in some discrete torsion, i.e.\ nontrivial Dijkgraaf--Witten action, into the $\rM_{22}$ gauge theory, since $\widetilde{\SH}{}^2(\rM_{22}) = \rH^3(\rM_{22};\bZ) = \bZ_6$ is nontrivial. This effects a change from the group algebra $\bC[\rM_{22}]$ to a twisted group algebra. The twisted group algebras of $\rM_{22}$ are Morita equivalent to a sum of 10 or 11 copies of $\bC$, depending on the twisting, and neither $17 + 10$ nor $17+11$ is divisible by $24$.
\end{proof}

\subsection{Twisted and twined shadows}

Proposition~\ref{prop.permDNW} means that we will not be able to answer Question~\ref{question.modularity} by working just with the permutation representation of $\rM_{24}$. There is another reason to reject it as an answer. Suppose, contradicting Proposition~\ref{prop.permDNW}, that $\Perm \otimes \overline{\Fer(3)}$ were $\rM_{24}$-equivariantly nullcobordant, and choose a nullcobordism $\cF$. For each commuting pair $g,h \in \rM_{24}$, we may twist and twine $\cF$, thereby producing a partition function $Z_{RR}(\cF)_{g,h}(\tau,\bar\tau)$. Following the logic of Conjecture~\ref{conj.gjfA}, the holomorphic part of $Z_{RR}(\cF)_{g,h}(\tau,\bar\tau)$, normalized with a factor of $\eta(\tau)$, will be a mock modular form (for some subgroup of $\SL_2(\bZ)$ depending on $g,h$) whose shadow is (the complex conjugate of) the torus one-point function of $\bar Q$ in $\Perm \otimes \overline{\Fer(3)}$, twisted and twined by $g$ and $h$.

Since $g$ and $h$ do not act on $\overline{\Fer(3)}$, this shadow factors as
$$ \text{putative shadow} = Z_{RR}(\Perm)_{g,h} \, \eta(\tau)^3.$$
The computation of $Z_{RR}(\Perm)_{g,h}$ is very easy, because $\Perm$ itself is very easy, being simply the TQFT of maps from the worldsheet to the standard permutation representation $\mathbf{24}$ of $\rM_{24}$. To have a map to this set from a torus twisted and twined by $g$ and $h$, the value of the map must be fixed by both $g$ and $h$, and we discover:
$$ Z(\Perm)_{g,h} \propto \text{number of common $g,h$ fixed points in }\mathbf{24}.$$
If we were treating $\Perm$ as a nonanomalous $\rM_{24}$-equivariant TQFT, then the two sides would be equal. We have written only that they are proportional because of the possibility of a nontrivial anomaly $\omega$. Indeed, the presence of $\omega$ means that the twisted-twined partition ``function'' is not really a function at all, but rather a section of a flat line bundle on the space of spin tori with $G$-bundles. Under modifying a $3$-cocycle representative of $\omega$ by $\d \xi$, for some $2$-cochain $\xi$ on $G$, the ``function'' $Z(\Perm)_{g,h}$ changes by a factor of $\frac{\xi(g,h)}{\xi(h,g)}$.

When $g = e$ is the identity, $Z_{RR}(\Perm)_{e,h}$ is simply the trace of the $h$-action on $\mathbf{24}$, which agrees with the shadows in Mathieu Moonshine (compare~\cite{MR3271175}). More generally, if the subgroup of $\rM_{24}$ generated by $g$ and $h$ is cyclic (for example, if $g$ and $h$ have coprime order), then $Z(\Perm)_{g,h} = Z(\Perm)_{1,x} = \tr_{\mathbf{24}}(x)$, where $x$ is any generator of the cyclic group. This is again consistent.

However, if the subgroup generated by $g$ and $h$ is not cyclic, then this putative shadow is not the shadow of the mock modular form $H_{g,h}(\tau)$ from (generalized) Mathieu Moonshine. Indeed,~\cite{MR3108775} finds that $H_{g,h}(\tau)$ has trivial shadow (i.e.\ it is holomorphic modular) as soon as $g$ and $h$ do not generate a cyclic group, and for most such pairs $H_{g,h}$ simply vanishes.
But there are many rank-2 subgroups of $\rM_{24}$ which do have fixed points. A list of all conjugacy classes of rank-2 abelian subgroups of $\rM_{24}$ is available in Table~1 of~\cite{MR3108775}. The first entry on that list, for example, is a Klein-4 subgroup $\bZ_2^2$ which acts with $8$ fixed points in $\mathbf{24}$.

Instead, as in Conjecture~\ref{conjecture.main}, we conjecture that $\overline{V^{f\natural}} \otimes \overline{\Fer(3)}$ is nullhomotopic. If it is, then the twisted and twined partition functions of its nullhomotopy would have, as their shadows, the functions $Z_{RR}(\overline{V^{f\natural}})_{g,h} \, \eta(\tau)^3$. The antiholomorphicity of $\overline{V^{f\natural}}$ means that $Z_{RR}(\overline{V^{f\natural}})_{g,h}$ is just an integer: the signed trace of $h$ acting on the ground states of the $g$-twisted Ramond sector of $V^{f\natural}$. When $g = 1$, these ground states form the Leech lattice representation $\mathrm{Leech} \otimes \bR$ of $\Co_0 = 2.\Co_1$. (The double cover comes from the ``Gu--Wen layer'' of the anomaly of the $\Co_1$-action on $\overline{V^{f\natural}}$.) This representation restricts over $\rM_{24}$ to the permutation representation, and so $Z_{RR}(\overline{V^{f\natural}})_{1,h} = \tr_{\mathbf{24}}(h) = Z_{RR}(\Perm)_{1,h}$, which is the desired value. More generally, if $g,h$ generate a cyclic group, with cyclic generator $x$, then $Z_{RR}(\overline{V^{f\natural}})_{g,h} = Z_{RR}(\overline{V^{f\natural}})_{1,x} = Z_{RR}(\Perm)_{1,x} = Z_{RR}(\Perm)_{g,h}$, simply because these two integers are related by a modular transformation. On the other hand, when $g,h$ generate a rank-2 group, then $Z_{RR}(\overline{V^{f\natural}})_{g,h}$ and $Z_{RR}(\Perm)_{g,h}$ may not agree. In fact:

\begin{theorem}
  If $g,h \in \rM_{24}$ generate a rank-2 abelian group, then $Z_{RR}(\overline{V^{f\natural}})_{g,h} = 0$. Thus Conjecture~\ref{conjecture.main} is consistent with the shadows found by~\cite{MR3108775}.
\end{theorem}

The calculations of \cite{MR3918493} suggest that there may be an elegant proof of this theorem, but the author did not find one. Rather, we will prove the theorem by computing all cases.

\begin{proof}
The integer $Z_{RR}(\overline{V^{f\natural}})_{g,h} = Z_{RR}(V^{f\natural})_{g,h}$ depends only on the conjugacy class of the rank-$2$ abelian group $\langle g,h\rangle$. It transforms with nontrivial multiplier under some congruence subgroup, and hence must be zero, as soon as the 't~Hooft anomaly restricts nontrivially to $\langle h\rangle$, or to any other generator of $\langle g,h\rangle$. This leaves only the groups where $\langle g,h\rangle$ consists of elements with $\rM_{24}$-conjugacy classes $1, 2\rA$, $3\rA$, $4\rB$, or $6\rA$. (The other nonanomalous elements in $\rM_{24}$ do not participate in rank-$2$ abelian groups.) In Table~1 of~\cite{MR3108775}, these are the entries numbered $1,2,3,17,18,19,25,26,28,29$.

In order to study these cases, we must understand the $g$-twisted Ramond-sectors of $V^{f\natural}$ for $g$ of $\rM_{24}$ conjugacy class $2\rA$, $3\rA$, $4\rB$, and $6\rA$. In fact, the only rank-2 subgroups in $\rM_{24}$ that include a $6\rA$ element are generated by a $6\rA$ element and a $2\rA$ element, and so, by switching $g$ and $h$, we do not need to consider the last case.

In order to study these twisted sectors, let us recall a bit about the holomorphic SVOA $V^{f\natural}$. It is a lattice SVOA for the \define{$D_{12}^+$ lattice}:
$$ D_{12}^+ = \left\{\lambda = (\lambda_1,\dots,\lambda_{12}) \in \bZ^{12} \sqcup {\textstyle \bigl(\bZ + \frac12\bigr){}^{12}} \text{ such that } \sum \lambda_i \in 2\bZ\right\}.$$
%, defined as those vectors $\lambda = (\lambda_1,\dots,\lambda_{12}) \in \bR^{12}$ such that $\sum \lambda_i$ is even and either all the $\lambda_i$ are integral or all the $\lambda_i$ are in $\bZ + \frac12$. 
It has a canonical translation coset inside $\bR^{12}$:
$$ (D_{12}^+)_R = \left\{\lambda = (\lambda_1,\dots,\lambda_{12}) \in \bZ^{12} \sqcup {\textstyle \bigl(\bZ + \frac12\bigr){}^{12}} \text{ such that } \sum \lambda_i \in 2\bZ + 1\right\}.$$
%
% $$, consisting of those vectors where again  the $\lambda_i$ are either all in $\bZ$ or all in $\bZ+\frac12$, but $\sum \lambda_i$ is odd. 
 The Ramond sector $V^{f\natural}_R$ is built from $(D_{12}^+)_R$ in the same way that the Neveu--Schwarz sector is built from $D_{12}^+$. Namely, $V^{f\natural}_R$ is generated over the Heisenberg algebra $\operatorname{Bos}(12)$ by a state $\Gamma_\lambda$ for each $\lambda \in (D_{12}^+)_R$.
There is no canonical way to assign a fermion parity operator ``$(-1)^f$'' to the R-sector of a holomorphic SVOA, but the relative parity is well-defined, and we will arbitrarily declare the absolute parity by saying that $\Gamma_\lambda$ is bosonic (resp.\ fermionic) if $\lambda \in \bZ^{12}$ (resp.\ $(\bZ+\frac12)^{12}$).

The $\cN{=}0$ automorphism group (i.e.\ the automorphism group as an SVOA, ignoring the supersymmetry) of $V^{f\natural}$ is the Lie group $\SO^+(12)$, defined as the image of $\Spin(12)$ in the positive half-spin representation. Because of the 't~Hooft anomaly, this group acts only projectively on the Ramond sector: the group that acts linearly on $V^{f\natural}_R$ is $\Spin(12)$ itself. Since $\SO^+(12)$ is connected, any $g \in \SO^+(12)$ can be conjugated into the maximal torus $T \cong \bR^{12} / D_{12}^+ \subset \SO^+(12)$. We will abusively also call this torus element ``$g$,'' but we will write the group law in $T$ additively.
 Any $g \in T$ determines a translated lattice $D_{12}^+ + g \subset \bR^{12}$, from which the $g$-twisted sector $V^{f\natural}_g$ is built. A special case is when $g = R$ is the central element in $\SO^+(12)$, in which $D_{12}^+ + R = (D_{12}^+)_R$ is the canonical translated lattice, and the notation ``$V^{f\natural}_R$'' is consistent. As an element of $\bR^{12} / D_{12}^+$, $R$ can be represented by the vector $R = (1,0,\dots,0) \pmod{D_{12}^+}$. More generally, the $g$-twisted R-sector deserves the name ``$V^{f\natural}_{R+g}$.''
For the symmetries $g \in \rM_{24} \subset \SO^+(12)$ that we are interested in, the vectors are:
$$ \begin{array}{c|c|c}
 \rM_{24}\text{ name} & \text{cycle structure} & g \in T = \bR^{12} / D_{12}^+ \\ \hline
 2\rA & 1^8 2^8 & (\frac12, \frac12, \frac12, \frac12, 0,0,0,0,0,0,0,0) \\
 3\rA & 1^6 3^6 & (\frac13, \frac13, \frac13, \frac13, \frac13, \frac13, 0,0,0,0,0,0) \\
 4\rB & 1^4 2^2 4^4 & (\frac12, \frac12, \frac12,  \frac14, \frac14,  \frac14, \frac14, 0, 0,0,0,0)
\end{array}
$$

In any sector of any lattice SVOA, $L_0$ acts on the state $\Gamma_\lambda$ with eigenvalue $|\lambda|^2/2$. If $g$ preserves a supersymmetry operator $Q$, then in the $g$-twisted R-sector the Hamiltonian $L_0 - c/24$ is a square (of the zero mode of $Q$, up to a normalization convention) and so takes only nonnegative values. We are interested in the space of ground states, which are thus in bijection with those $\lambda \in V^{f\natural}_{R+g}$ with $|\lambda|^2 = 1$ (since $c=12$ for $V^{f\natural}$). Such a state contributes a bosonic or fermionic mode according to whether $\lambda - g$ is integral or half-integral.
Each of the vectors $g \in T$ listed above contains at least five $0$s. Thus a vector $\lambda \in g + (\bZ+\frac12)^{12}$ will have at least five entries with absolute value $\geq \frac12$, and so $|\lambda|^2 \geq 5 \frac14 > 1$. It follows that the $g$-twisted R-sector has only bosonic ground states. The number of ground states is then, by modularity, equal to the trace of $g$ acting on the ground states in $V^{f\natural}_R$, which is just the $24$-dimensional representation of $\rM_{24}$. This trace~$\tr_{\mathbf{24}}(g)$ is easily read from the above table: it is the exponent of $1$ in the cycle structure.
(A priori, there could be both bosonic and fermionic ground states in the $g$-twisted R-sector, and only their signed count is equal to~$\tr_{\mathbf{24}}(g)$.)

For any $g$ acting on any holomorphic conformal field theory $V$, the $g$-twisted sectors $V_g$ and $V_{R+g}$ carry projective actions of the centralizer $C(g)$ inside the automorphism group of $V$.
As we have remarked already, the anomaly for the $\rM_{24}$-action on $V^{f\natural}$ is (up to a sign) the same as the anomaly computed in~\cite{MR3108775}. This anomaly determines the projectivity of the action of $C(g)$ on the $g$-twisted sectors. (See~\cite{MR3108775} for a nice explanation.) The centralizers $C(g) \subset \rM_{24}$ of the elements $g$ listed above, and their projective character tables,  are listed in the appendix of~\cite{MR3108775}.

When $g = 2\rA$, we have $C(2\rA) = 2^4.(2^3{:}\rL_3(2))$ in the ATLAS notation. Its action on $V^{f\natural}_{R+2\rA}$ is genuinely projective. The eight ground states must compile into an $8$-dimensional module. This is the smallest dimension of any projective representation on $C(2\rA)$. It remains to identify the correct representation: they are listed under the names $\chi_1,\chi_2,\chi_3,\chi_4$ in the appendix of~\cite{MR3108775}. But look at the class called $4\rA_1$ therein. It has a nontrivial anomaly, and so its trace vanishes. Thus we find that the ground states of $V^{f\natural}_{R+2\rA}$ correspond to the character $\chi_1$. The only nonzero entries in the $\chi_1$ correspond to elements $h \in C(2\rA)$ such that $h$ and $g = 2\rA$ together generate a cyclic group. This establishes the Theorem for the groups numbered $1,2,3,17,18,19, 28,29$ in Table~1 of~\cite{MR3108775}.

When $g = 3\rA$, we have $C(3\rA) = 3A_6$, the exceptional Schur cover of the alternating group $A_6$. The ground states of $V^{f\natural}_{R+3\rA}$ form a six-dimensional linear representation $M$; because $\H^1(C(3\rA);\rU(1)) = 0$, there is a canonical trivialization of the projectivity, and so no phase ambiguities in the actions of elements of $C(3\rA)$ on $M$. The centralizer of $g$ inside the full automorphism group $\Aut_{N=1}(V^{f\natural}) = \Co_1$ is a group of shape $3^2.\rU_4(3).2$~\cite{ATLAS,MR716777}. This acts through a double cover $(3^2\times 2).\rU_4(3).2$ on $V^{f\natural}_{R+3\rA}$, and $6$ is the smallest dimension of any simple representation thereof. The central $g \in (3^2\times 2).\rU_4(3).2$ acts on all $6$-dimensional representations with trace $\pm 3 \pm 3\sqrt{-3}$. It follows that $M$ breaks up over $3A_6$ as a sum of the characters labeled $\chi_2,\chi_3,\chi_4,\chi_5,\chi_8,\chi_9$ in~\cite{MR3108775}. For all of these modules, the element labeled $3\rA_2$ acts with trace $0$. This establishes the Theorem for the group numbered $33$ in Table~1 of~\cite{MR3108775}.

Finally, we have the groups $\langle g,h\rangle$ numbered $25$ and $26$, for which $g$ and $h$ are both of class $4\rB$. According to~\cite{MR3108775}, $C(4\rB)$ is a group of shape $((4\times4){:}4){:}2$, acting genuinely projectively on the four ground states of $V^{f\natural}_{R+4\rB}$. For both groups (numbers $25$ and $26$), the centralizer of $\langle g,h\rangle$ has order $16$. It follows that $h$ is one of the conjugacy classes named ``$4\rB_3$,'' ``$4\rB_4$,'' ``$4\rB_5$,'' and ``$4\rB_7$'' in the appendix of~\cite{MR3108775}. But for $h = 4\rB_5$ or $4\rB_7$, the group $\langle g,h\rangle$ contains an element of class $4\rA$, which has an anomaly (they correspond to the groups numbered $23$ and $24$ in Table~1 of~\cite{MR3108775}). The classes $h = 4\rB_3$ and $4\rB_4$ act with trivial trace on all genuinely projective representations of $C(4\rB)$.
\end{proof}

The mock modular forms $H_{g,h}$ from~\cite{MR3108775} are \define{integral} in the sense that their coefficients, as functions of $h$,  are all virtual characters of projective representations of the centralizer of $g$ in $\rM_{24}$. (Indeed, they are mostly zero.)
 Thus the equivariant version of the invariant from~\cite{GJFmock} vanishes for $\overline{V^{f\natural}} \otimes \overline{\Fer(3)}$: that invariant does not obstruct Conjecture~\ref{conjecture.main}.

It was observed early in the development of Mathieu Moonshine \cite{MR2955931,MR3539377} that the characters that appear (i.e.\ the coefficients of the $q$-expansion of $H_{1,h}(\tau)$) are all restrictions of virtual characters of projective represenations of $\Co_1$. However, it is unlikely that the functions $H_{1,h}$, let alone $H_{g,h}$, have \emph{(mock) modular} integral extensions to $\Co_1$. Said another way, it is likely that the invariant from~\cite{GJFmock} is strong enough to prove:

\begin{conjecture}
  $\overline{V^{f\natural}} \otimes \overline{\Fer(3)}$ is not $\Co_1$-equivariantly nullhomotopic.
\end{conjecture}

\subsection{$\TMF^\bullet$ and $\tmf^\bullet$}\label{subsec.tmf}

The main conjecture of \cite{MR2079378,MR2742432} (our Conjecture~\ref{conjecture.ST}) is that the spectrum $\SQFT^\bullet$ is equivalent to the ``universal elliptic cohomology'' spectrum $\TMF^\bullet$ called \define{Topological Modular Forms}. There is quite a lot of evidence in favour of this conjecture, and versions of it were predicted as early as \cite{MR885560,MR970288,MR992209}. Notably, Witten explained in the first of those papers that the then-recently-discovered ``elliptic genus'' of Landweber, Stong, and Ochanine arises as the $\bZ_2$-twisted partition function of the $\cN=(1,1)$ sigma model (with $\bZ_2$-action that breaks the the left-moving supersymmetry), and also introduced what is now known as the \define{Witten genus} by using instead an $\cN=(0,1)$ sigma model. 

Another piece of evidence supporting 
Conjecture~\ref{conjecture.ST}
is that the corresponding statement in $(0{+}1)$ dimensions is understood~\cite{CheungThesis,MR2648897,MR2653060,UlricksonSUSY}. Indeed, the construction from~\S\ref{subsec.spectrum}  with $\cN{=}1$ supersymmetric quantum mechanics models in place of $(1{+}1)$-dimensional QFTs  produces a spectrum $\SQM^\bullet$. If the SQM models are not required to support a time-reversal symmetry, then the resulting spectrum $\SQM^\bullet$ is known to model the complex K-theory spectrum $\mathrm{KU}^\bullet$. If the SQM models are equipped with a time-reversal symmetry then $\SQM^\bullet$ models orthogonal K-theory $\mathrm{KO}^\bullet$.
 (The time-reversal symmetry must satisfy $T^2 = +1$, corresponding to $\mathrm{Pin}^-$ under Wick rotation. The $\mathrm{Pin}^+$ case $T^2 = (-1)^f$ is not compatible with the dynamicalization procedure leading to a spectrum structure.)
% $\mathrm{KO}^\bullet$~is the universal cohomology theory of chromatic height $1$, whereas $\TMF^\bullet$ is the universal cohomology theory of chromatic height $2$. 

Quantum mechanics has a complete mathematical axiomatization in terms of Hilbert spaces and von Neumann algebras, and the statement ``$\SQM^\bullet = \mathrm{K}^\bullet$'' is a mathematical theorem. Presuming that a complete mathematical axiomatization of unitary, compact $(1{+}1)$-dimensional quantum field theory can be found, Conjecture~\ref{conjecture.ST} offers an analytic model of $\TMF^\bullet$, for which so far the only known models are homotopy-algebraic. (Progress towards proving Conjecture~\ref{conjecture.ST} is available in \cite{CheungThesis,BET1,BET2}, which establish versions ``over the Tate curve'' and, equivariantly, over $\bC$.)

There are in fact three closely-related spectra that go under the name ``topological modular forms,''  distinguished by their capitalizations. The first, $\TMF^\bullet$, is the space of ``weakly holomorphic topological modular forms'': it is a homotopical refinement of the ring $\MF^\bullet$ of integral modular forms that are holomorphic for finite values values of $\tau$, but possibly meromorphic at the cusp $\tau = i\infty$. (By ``integral,'' we mean that the $q$-expansion lives in $\bZ(\!(q)\!)$. Modular forms of weight $w$ are assigned cohomological degree $n=-2w$.)
The algebrotopological definition of $\TMF^\bullet$ is a ``derived'' version of $\MF^\bullet$. Specifically, there is a ``derived stack'' $\cM_{ell}^{der}$ which refines the stack $\cM_{ell}$ of smooth elliptic curves. It carries a ``derived structure sheaf'' $\cO^{der}$ whose fibre at an elliptic curve $E \in \cM_{ell}^{der}$ is the spectrum presenting $E$-elliptic cohomology. The homotopy sheaf $\pi_{2w} \cO^{der} = L^{\otimes w}$ is the line bundle whose sections are weight-$w$ modular forms. 
(Constructing these derived algebrogeometric objects is hard \cite{MR2125040,MR2597740,MR2648680,MR3328535,EllipticI,EllipticII,EllipticIII}.)
Then $\TMF^\bullet$ is the spectrum of derived global sections of $\cO^{der}$. This is the spectrum that appears in Conjecture~\ref{conjecture.ST}.

The second spectrum, $\Tmf^\bullet$, is the space of ``holomorphic topological modular forms,'' analogous to the ring $\mf^\bullet$ of modular forms that are bounded at $\tau = i\infty$. Its definition parallels $\TMF^\bullet$ and $\mf^\bullet$:  compactify $\cM_{ell}^{der}$ to a derived stack $\overline{\cM}_{ell}^{der}$ that allows elliptic curves with nodal singularities, extend the derived structure sheaf, and take derived global sections. 
Because of the derived nature of these constructions, the homotopy groups of both $\TMF^\bullet$ and $\Tmf^\bullet$ include information about the cohomology of the line bundles $L^{\otimes w}$. In particular, even though there are no holomorphic modular forms of negative weight (positive cohomological degree), the line bundles $L^{\otimes w}$ for negative even $w$ do have cohomology over $\overline{\cM}_{ell}$, leading to nontrivial classes in $\pi_{-\bullet}\Tmf = \Tmf^{\bullet}(\pt)$ for $\bullet>0$.

The physical significance of $\Tmf^\bullet$ is not yet clear: there probably is an analogue of Conjecture~\ref{conjecture.ST} for $\Tmf^\bullet$, but a satisfactory one has not yet been proposed. One approach is suggested in \cite{MR2653060}, but this author doubts that that method can be made physically sensible in $(1{+}1)$ dimensions. A more direct approach would involve a spectral constraint on the operator $L_0$ in the Ramond sector which is strong enough to assure that the adjusted partition function $Z'_{RR}(\tau) = Z_{RR}(\tau) \eta^{2(c_R - c_L)}(\tau)$ converges as $\tau \to i\infty$.  For example, it should rule out the holomorphic SCFT $V^{f\natural}$, since that SCFT represents the class $\{24\Delta^{-1}\} \in \TMF^{24}(\pt)$, which is not in the image of $\Tmf^\bullet \to \TMF^\bullet$. On the other hand, $\Tmf^{21}(\pt) \cong\bZ$ is generated by a class that deserves the name $\{24\Delta^{-1}\nu\}$, which is is in the kernel of $\Tmf^{\bullet} \to \TMF^\bullet$. The author believes that the generator $\{24\Delta^{-1}\nu\}$ should be represented by the $\cN{=}(1,1)$ SCFT $V^{f\natural} \otimes \overline{\Fer(3)}$, but it is not clear how to tune the constraint so as to allow this.

Third, the spectrum $\tmf^\bullet$ is defined to be the connective cover of $\Tmf^\bullet$: as an $\Omega$-spectrum, $\tmf^n = \Tmf^n = \Omega^{-n}\Tmf^0$ for $n \leq 0$, but $\tmf^n = B^n \Tmf^0$ for $n > 0$. Said another way, $\tmf^\bullet$ is built by keeping only the $0$-space $\Tmf^0$ of the $\Tmf^\bullet$-spectrum, which is automatically an infinite loop space, and then interpreting infinite loop spaces as a special class of spectra.
In homotopy, we have
$$ \tmf^\bullet(\pt) = \begin{cases} \Tmf^\bullet(\pt), & \bullet \leq 0, \\ 0, & \bullet > 0. \end{cases}$$
More generally, if $X$ is a space, then $\tmf^\bullet(X) = \Tmf^\bullet(X)$ if $\bullet \leq 0$, but not for $\bullet > 0$.
We are interested in a particular $\rM_{24}$-equivariant $\TMF$-class $[\overline{V^{f\natural}}]\nu$ of cohomological degree $\bullet = -27$.  Nonequivariantly,
 $[\overline{V^{f\natural}}] = \{24\Delta\} \in \TMF^{-24}(\pt)$ is in the image of $\Tmf^{-27}(\pt)$, and we believe that this holds $\rM_{24}$-equivariantly as well. 
 %Upon replacing the stack $\bB \rM_{24}$ with its classifying space $B \rM_{24}$, we in particular find that $\Tmf^{-27}_\alpha(B\rM_{24}) = \tmf^{-27}_\alpha(B\rM_{24})$, and there is no cost to working with $\tmf^\bullet$ instead of $\Tmf^\bullet$ in Section~\ref{sec.AHSS}.
% Moreover, 
 Nonequivariantly $\{24\Delta\} \nu = 0 \in \tmf^{-27}(\pt)$, and hence in $\TMF^{-27}(\pt)$. Together with Conjecture~\ref{conjecture.ST}, this implies that $[\overline{V^{f\natural}}]\nu = 0 \in \SQFT^{-27}(\pt)$.

It is not expected that $\tmf^\bullet$ itself will admit a natural physical description. The reason is that any physical description in terms of spaces of SQFTs will naturally lead to a genuinely equivariant enhancement (by working with SQFTs  with a given flavour symmetry), but $\tmf^\bullet$ is not expected to admit a genuinely equivariant enhancement.
A better calculation than we will attempt in Section~\ref{sec.AHSS} would be to work out the equivariant cohomology $\Tmf^{-27}_\alpha(\bB \rM_{24})$, and perhaps show, as suggested by Theorem~\ref{thm.main}, that it vanishes away from the prime $p=2$. (Perhaps it even vanishes at $p=2$.) But there is no Atiyah--Hirzebruch spectral sequence for computing equivariant cohomology groups like $\Tmf^{-27}_\alpha(\bB \rM_{24})$, and so we will not attempt such a calculation. Rather, in Section~\ref{sec.AHSS} we will attempt $\Tmf^{-27}_\alpha(B\rM_{24})$, approximating the classifying stack $\bB G$ by its classifying space $BG$, and the above remarks identify $\Tmf^{-27}_\alpha(B\rM_{24}) \cong \tmf^{-27}_\alpha(B\rM_{24})$, since $B\rM_{24}$ is just a space, not a stack.

Physically, the difference between $\bB G$ and $BG$ is the following. As explained in \S\ref{subsec.anomalies}, a class in $\SQFT^\bullet_\omega(\bB G)$ is represented by a compact SQFT with $G$ flavour symmetry (and anomaly $\omega$). A class in $\SQFT^\bullet_\omega(B G)$ is instead represented by a family of SQFTs over the space $BG$. There is a map $\SQFT^\bullet_\omega(\bB G) \to \SQFT^\bullet_\omega(B G)$ which uses the $G$-action on an element $\cF \in \SQFT^\bullet_\omega(\bB G)$ to prescribe the monodromies of a family over $BG$ which is locally constant with value $\cF$. 
This map is definitely not an isomorphism of spaces, because its image consists of families which are locally constant, whereas a typical family over $BG$ may vary a lot. 
But we don't need it to be, since we care only about homotopy classes. With some work, a family of SQFTs over $BG$ can be ``integrated'' to a $G$-equivariant SQFT, and so one might expect that $\SQFT^\bullet_\omega(\bB G) \to \SQFT^\bullet_\omega(B G)$ is a homotopy equivalence. The problem is that $BG$ is infinite-dimensional, and so this ``integral'' will usually fail to produce a \emph{compact} SQFT. (A topological space is ``finite-dimensional'' if it is homotopy equivalent to a finite cell complex. Except for spaces homotopy equivalent to finite sets, no space is both finite-dimensional and finite in homotopy.) 
Indeed, there are sequences of $G$-equivariant SQFTs which diverge in $\SQFT^\bullet_\omega(\bB G)$ because their limits are  noncompact, but which converge in $\SQFT^\bullet_\omega(B G)$ because this noncompactness can be concentrated ``near infinity'' in $BG$: as you go out along a cell decomposition of $BG$, the family stays compact but becomes larger and larger.
As such, one expects $\SQFT^\bullet_\omega(BG)$ to be a ``completion'' of $\SQFT^\bullet_\omega(\bB G)$, analogous to the famous result from \cite{MR259946} describing $\mathrm{KU}^\bullet(BG)$ as a completion of $\mathrm{KU}^\bullet(\bB G)$. A  completion statement for $\TMF$ is known for finite abelian groups $G$ \cite{EllipticIII}. But we warn that direct computations in \cite{2004.10254} show that for $G = \rU(1)$, the map $\TMF^\bullet(\bB \rU(1)) \to \TMF^\bullet(B \rU(1))$ is far from a completion. (There is a more sophisticated ``completion'' statement that holds for $\rU(1)$ at the level of sheaves over $\cM_{ell}$. The failure of $\TMF^\bullet(\bB \rU(1)) \to \TMF^\bullet(B \rU(1))$ can then be traced to the non-affineness of the stack of elliptic curves with $\rU(1)$-bundle.)

In addition to the rings $\MF^\bullet$ and $\mf^\bullet$ of weakly holomorphic and holomorphic modular forms, number theorists care also about the ideal $\mathrm{cf}^\bullet \subset \mf^\bullet$ of \define{cusp forms}, which are the holomorphic modular forms which vanish at $\tau = i\infty$. Like modular forms, cusp forms admit a topological enhancement to a spectrum $\mathrm{Tcf}^\bullet$ of \define{topological cusp forms}. Summarizing a fair amount of hard work, the idea is to promote the restriction map $\mf^\bullet \to \cO(\text{cusp}) = \bZ$ to a map $\Tmf^\bullet \to \mathrm{KO}^\bullet$, which was done in \cite{MR3455154}. Then $\mathrm{Tcf}^\bullet$ is defined as the homotopy fibre of $\Tmf^\bullet \to \mathrm{KO}^\bullet$. The topology literature seems to contain very little investigation of $\mathrm{Tcf}^\bullet$, and, just like for $\Tmf^\bullet$, the physical significance it not yet clear. We will mention one interesting fact about $\mathrm{Tcf}^\bullet$, which makes its behaviour different from the classical case of $\mathrm{cf}^\bullet$. Namely, $\mathrm{cf}^\bullet$ is the principal ideal inside $\mf^\bullet$ generated by $\Delta$, and as such it represents (up to suspension, aka degree-shift) the trivial class in the Picard group of $\mf^\bullet$. But $\Delta$ does not lift to an element in $\Tmf^\bullet(\pt)$, and $\mathrm{Tcf}^\bullet$ is not isomorphic to a suspension of $\Tmf^\bullet$. Rather, in unpublished work L.\ Meier has identified the class of $\mathrm{Tcf}^\bullet$ with the exotic $24$-torsion element in the Picard group of $\Tmf^\bullet$ called $\Gamma(\cJ)$ in \cite{MR3590352}. (That paper shows that $\Gamma(\cJ)$ and suspension together generate $\operatorname{Pic}(\Tmf^\bullet) \cong \bZ_{24} \times \bZ$. Other exotic elements are studied in \cite{MR3685599,MR4054878}.)

The genus-zero property in Monstrous Moonshine is reformulated in \cite{MR3449012} in terms of an \define{optimal growth condition} which provides the ``moonshine'' part of Umbral (and in particular Mathieu) Moonshine. The condition (for $\rM_{24}$) is simply that the mock modular forms $H_{g,h}(\tau)$ grow no worse than $q^{-1/8}$ as $\tau \to i\infty$, and are bounded near other cusps. Recall from \S\ref{subsect.holomorphicanomaly} our proposal that $H_{g,h}(\tau)$ is the holomorphic part of $\eta(\tau) Z_{RR}(\cF)_{g,h}(\tau,\bar\tau)$, for some SQFT $\cF$ with with cylindrical ends $\partial \cF = \overline{V^{f\natural}}\otimes \overline{\Fer(3)}$ of cohomological degree $n =-28$, whereas the homotopy theory convention for Witten genera is to work with the adjusted partition function $Z'_{RR}(\tau,\bar\tau) = Z_{RR}(\tau,\bar\tau) \eta(\tau)^{-n}$. The optimal growth condition then becomes the statement that the adjusted  function $H_{g,h}(\tau) \eta(\tau)^{27} = O(q^1)$, i.e.\ it is a ``mock cusp form.''
This leads us to propose the following answer to Question~\ref{question.growth}:
\begin{conjecture}\label{conjecture.growth}
  The $\rM_{24}$-equivariant class $[\overline{V^{f\natural}}]\nu$, refining the cusp form $\{24\Delta\}\nu = 0 \in \Tmf^{-27}(\pt)$, is $\rM_{24}$-equivariantly nullhomotopic in the spectrum $\mathrm{Tcf}^\bullet$ of topological cusp forms.
\end{conjecture}

\section{Ordinary cohomology of $\rM_{24}$}\label{sec.HM24}

In order to run the Atiyah--Hirzebruch spectral sequence for $\tmf^\bullet_\omega(B\rM_{24})[\frac12]$, we will need good control over the ordinary group cohomology rings $\H^\bullet(\rM_{24};\bZ[\frac12])$ and $\H^\bullet(\rM_{24};\bF_3)$.
We find it necessary to invert the prime $2$ simply because the $2$-local computations are too hard.
These rings were computed in~\cite{MR1263724,MR1373961}. 
We will review and extend that analysis: our goal is to have explicit control over the action of the first Steenrod cube $\cP$ on $\H^\bullet(\rM_{24};\bF_3)$.

As with any finite group, the computation of $\H^\bullet(\rM_{24};\bZ[\frac12])$ factors prime-by-prime, and at the prime $p$, the computation is controlled by the Sylow $p$-subgroup. The primes $p\geq 5$ are quite straightforward: the Sylow $p$-subgroup is cyclic, from which it already follows that $\H^\bullet(\rM_{24};\bZ[\frac16])$ is supported in even degrees. Since $\tmf^\bullet(\pt)[\frac16]$ is supported in even degrees and has no torsion, we learn immediately that $\tmf^\bullet_\omega(B\rM_{24})[\frac16]$, which is independent of the twisting $\omega$ since $\omega$ is $12$-torsion, is also supported only in even degrees. A stronger statement is proved in~\cite{MR1263724}: $\tmf^\bullet_\omega(B\rM_{24})[\frac16]$ is generated by elliptic Chern classes. (The cohomology theory $\tmf^\bullet[\frac16]$ is called $E\ell\ell^\bullet$ in~\cite{MR1263724,MR982399}.)

\subsection{Computing in $\H^\bullet(\rM_{24};\bF_3)$}

Thus the interesting computation is at the prime $3$, where the Sylow subgroup is nonabelian, being isomorphic to the extraspecial group $S = 3^{1+2}_+$ of order $27$ and exponent $3$. The rings $\H^\bullet(\rM_{24};\bZ_{(3)})$ and $\H^\bullet(\rM_{24};\bF_3)$ are computed in~\cite{MR1373961}. We will report the main result, but change some letters:
\begin{theorem}[\cite{MR1373961}]
\label{thm.Green}
  The graded commutative ring $\H^\bullet(\rM_{24};\bZ_{(3)})$ has a presentation with four generators, of cohomological degree and additive order as follows:
  $$\begin{array}{c|c|c}
    \text{ Name } & \text{ Degree } & \text{ Additive order } \\ \hline
    r & 4 & 3 \\
    s & 12 & 9 \\
    t & 16 & 3 \\
    u & 11 & 3 
  \end{array}$$
  The only relations are $u^2 = 0$ (which follows from the Koszul sign rules) and $rt = 0$.
  
  We will use the same names for classes in $\H^\bullet(\rM_{24};\bZ_{(3)})$ as for their mod-3 reductions inside $\H^\bullet(\rM_{24};\bF_3)$. The $\operatorname{Ext}$ term in the universal coefficient theorem implies that each generator ``$x$'' of $\H^\bullet(\rM_{24};\bZ_{(3)})$ of cohomological degree $n$ also determines a generator of $\H^\bullet(\rM_{24};\bF_3)$ of degree $(n-1)$, which will be denoted ``$X$.'' 
  They are related by the mod-3 Bockstein $\Box$ for the extension $\bZ_3 \to \bZ_9 \to \bZ_3$:
  $$ \Box R = r, \qquad \Box S = 3s = 0, \qquad \Box T = t, \qquad  \Box U = u.$$
  Note that $\Box S = 0$ in $\H^\bullet(\rM_{24};\bF_3)$ because $s \in \H^\bullet(\rM_{24};\bZ_{(3)})$ has additive order $9$ and not $3$.
  The following is a complete list of relations for $\H^\bullet(\rM_{24};\bF_3)$ as an $\bF_3$-algebra:
  \begin{gather*}
    Ru = rU, \qquad TS = Tu = tU, \qquad tS = tu, \\
    u^2 = R^2 = T^2 = U^2 = rt = rT = uU = tR = RU = RT = UT = 0, \\
    rS = uS = RS = US = S^2 = 0.
  \end{gather*}
\end{theorem}

A standard lemma (see e.g.\ Section XII.8 of~\cite{MR0077480}) implies that the restriction maps $\H^\bullet(\rM_{24};\bZ_{(3)}) \to \H^\bullet(S;\bZ_{(3)})$ and $\H^\bullet(\rM_{24};\bF_3) \to \H^\bullet(S;\bF_3)$ are injections onto direct summands, where $S = 3^{1+2}_+ \subset \rM_{24}$ is the Sylow 3-subgroup. The rings $\H^\bullet(S;\bZ_{(3)})$ and $\H^\bullet(S;\bF_3)$ are computed in~\cite{MR223430,MR1104598,MR1162933}. In particular, the third of those articles confirms the following result due to~\cite{MR1387653} (the order of final publication did not match the order in which the preprints were originally circulated):
\begin{proposition}[\cite{MR1387653}]\label{prop.MT}
  Let $A_1,A_2,A_3,A_4$ denote the four maximal abelian subgroups of~$S = 3^{1+2}_+$, each isomorphic to $\bZ_3 \times \bZ_3$. The total restriction map
  $$ \H^\bullet(S;\bF_3) \to \prod_i \H^\bullet(A_i;\bF_3)$$
  is an injection.
\end{proposition}

Since $\H^\bullet(\rM_{24};\bF_3) \subset \H^\bullet(S;\bF_3)$, it follows in particular that we can compute inside $\H^\bullet(\rM_{24};\bF_3)$ by computing restrictions to subgroups isomorphic to $\bZ_3 \times \bZ_3$. Let us describe these subgroups. 
First, $\rM_{24}$ has two conjugacy classes of elements of order $3$. 
Class~$3\rA$ consists of those elements in $\rM_{24}$ that act in the degree-24 permutation representation with cycle structure $1^6 3^3$;
 class~$3\rB$ acts with cycle structure~$3^8$. The central $\bZ_3 \subset S$ consists of $3\rA$-elements. 
There are also two conjugacy classes of subgroups $\bZ_3 \times \bZ_3 \subset \rM_{24}$. Both subgroups are maximal-abelian. One of them, which we will call simply $\bZ_{3\rA} \times \bZ_{3\rA}$, consists entirely of $3\rA$-elements. The other is  $\bZ_{3\rA} \times \bZ_{3\rB}$. In addition to the identity element, it contains two $3\rA$-elements (forming a $\bZ_3$-subgroup), and the other six elements are of class $3\rB$. The ``Weyl groups'' $W(A) = N(A)/A$ of these maximal abelian subgroups of $\rM_{24}$ are as large as possible given the conjugacy classes of elements:
$$ W(\bZ_{3\rA} {\times} \bZ_{3\rA}) = \frac{N(\bZ_{3\rA} {\times} \bZ_{3\rA})}{\bZ_{3\rA} {\times} \bZ_{3\rA}} \cong \GL_2(\bF_3), \qquad W(\bZ_{3\rA} {\times} \bZ_{3\rB}) = \frac{N(\bZ_{3\rA} {\times} \bZ_{3\rB})}{\bZ_{3\rA} {\times} \bZ_{3\rB}} \cong D_{12}.$$
By $D_{12}$ we mean the dihedral group of order $12$, isomorphic to the upper Borel $\bigl( \begin{smallmatrix} * & * \\ 0 & * \end{smallmatrix} \bigr) \subset \GL_2(\bF_3)$.
These and other claims about $\rM_{24}$ are easily checked in the computer algebra program GAP~\cite{GAP}.

For any finite group and any abelian subgroup $A \subset G$ and for any ring $R$, the restriction map $\rH^\bullet(G;R) \to \rH^\bullet(A;R)$ lands within the Weyl-invariant subring $\rH^0(W(A); \rH^\bullet(A;R))$. 
We therefore conclude:
$$ \H^\bullet(\rM_{24}; \bF_3) \subset \H^0(\GL_2(\bF_3); \H^\bullet(\bZ_{3\rA} \times \bZ_{3\rA}; \bF_3)) \times \H^0(D_{12}; \H^\bullet(\bZ_{3\rA} \times \bZ_{3\rB}; \bF_3)).$$
The next step is to understand the right-hand side. Let us choose ``coordinates'' on $\bZ_{3\rA} \times \bZ_{3\rA}$ and $\bZ_{3\rA} \times \bZ_{3\rB}$, writing $Y$, resp.\ $Z$, for the homomorphisms onto the ``fibre'' $\bZ_{3\rA}$, resp.\ and onto the ``base'' $\bZ_{3\rA}$ or $\bZ_{3\rB}$. After identifying $\bZ_3 = \bF_3$, these coordinates give classes in $\H^1(\bZ_3 \times \bZ_3; \bF_3)$. The full algebra $\H^\bullet(\bZ_3 \times \bZ_3; \bF_3)$ is then a graded polynomial algebra
$$ \H^\bullet(\bZ_3 \times \bZ_3; \bF_3) = \bF_3[ Y,Z,y,z] $$
where $y = \Box Y$ and $z = \Box Z$ are in degree $2$, and the only relations are the ones imposed by the Koszul sign rules: $Y^2 = Z^2 = YZ + ZY = 0$. Note also that mod-3 reduction identifies $\H^\bullet(\bZ_3\times\bZ_3; \bZ)$ with $\ker\Box \subset \H^\bullet(\bZ_3 \times \bZ_3; \bF_3)$ (except in degree $0$), and that this ring is generated by $y$ and $z$ and the degree-3 element $w = yZ - Yz = \Box(YZ)$. In particular, the subring of $\H^\bullet(\bZ_3 \times \bZ_3; \bF_3)$ consisting of even-degree elements with integral lifts is $\bF_3[y,z]$.

The $D_{12}$-action on $\bF_3[ Y,Z,y,z]$ is generated by the following three automorphisms:
\begin{gather*}
(Y,Z,y,z) \mapsto (-Y,z,-y,z), \qquad (Y,Z,y,z) \mapsto (Y,-z,y,-z), \\ (Y,Z,y,z) \mapsto (Y+Z,Z,y+z,z).
\end{gather*}
To generate the full $\GL_2(\bF_3)$-action, it suffices to include also the automorphism
$$(Y,Z,y,z) \mapsto (Z,Y,z,y).$$

\begin{lemma}\label{lemma.cz}
  The subring of $\bF_3[y,z]$ invariant under $(y,z) \mapsto (y+z,z)$ is $\bF_3[z,c]$ where $c = y(y+z)(y-z) = y^3 - yz^2$ is of cohomological degree $6$. The subring of $\bF_3[y,z,w]$ is $\bF_3[z,w,c]$.
\end{lemma}
Note that $\bF_3[y,z,w] = \H^\bullet(\bZ_3 \times\bZ_3;\bZ)$ except in degree $0$.
\begin{proof}
  Certainly $z$, $c$, and $w$ are invariant. Suppose $p(y,z) = p_0 z^n + p_{1} yz^{n-1} + \dots + p_n y^n$ is an invariant homogeneous polynomial of polynomial degree $n$ (cohomological degree $2n$). Then the largest $i$ with $p_i \neq 0$ must be divisible by $3$. Indeed, otherwise under $y \mapsto y+z$, the coefficient on $y^{i-1}z^{n-i+1}$ will change by $i p_i \neq 0$. So the space of invariant degree-$n$ polynomials is at most $(1+\lfloor \frac n 3 \rfloor)$-dimensional. But this is the dimension of the space of degree-$n$ polynomials in $\bF_3[z,c]$. The second claim follows from the first together with the fact that $\bF_3[y,z,w]$ = $\bF_3[y,z] \oplus w \bF_3[y,z]$, since $w^2 = 0$ by the Koszul sign rules.
\end{proof}

We note also that $z$ and $c$ are not invariant under $y\leftrightarrow z$, but that $z^6 + c^2 = y^6 + y^4 z^2 + y^2 z^4 + z^6$ 
and $z^2 c^2 = y^6 z^2 + y^4 z^4 + y^2 z^6$ are.

We can now work out the restrictions to $\bZ_{3\rA}\times\bZ_{3\rA}$ and $\bZ_{3\rA}\times\bZ_{3\rB}$ of the integral generators $r,s,t,u$. First, $r$ has cohomological degree $4$. There are no $\GL_2(\bF_3)$-invariant degree-4 classes, and so $r|_{\bZ_{3\rA}{\times}\bZ_{3\rA}} = 0$. Then Proposition~\ref{prop.MT} implies that $r|_{\bZ_{3\rA}{\times}\bZ_{3\rB}} \neq0$, and so is proportional to $z^2$, as that is the only degree-4 class invariant under $y \mapsto y+z$.
Changing the sign of $r$ if necessary, we learn:
$$ r|_{\bZ_{3\rA}{\times}\bZ_{3\rA}} = 0, \qquad r|_{\bZ_{3\rA}{\times}\bZ_{3\rB}} = z^2.$$
The extension class $R$ is also immediate, since it is an invariant degree-3 class satisfying $\Box R = r$.
$$ R|_{\bZ_{3\rA}{\times}\bZ_{3\rA}} = 0, \qquad R|_{\bZ_{3\rA}{\times}\bZ_{3\rB}} = Zz.$$
Note that $w$ is a degree-3 class invariant under $z \mapsto y+z$ and in the kernel of $\Box$, but it picks up a sign under some of the $D_{12}$ transformations, and so cannot appear here.

The next classes worth considering are the generators $T,t$, since $rT = rt = 0$. Thus both of these classes restrict trivially to $\bZ_{3\rA}\times\bZ_{3\rB}$, since $z^2$ is not a zero-divisor. Thus $T$ and $t$ restrict nontrivially to $\bZ_{3\rA}\times\bZ_{3\rA}$ by Proposition~\ref{prop.MT}. Since $t$ is an integral class of cohomological degree $16$, its restriction to $\bZ_{3\rA}\times\bZ_{3\rA}$ must a polynomial in $y$ and $z$ of polynomial degree $8$ invariant under all of $\GL_2(\bF_3)$. By using Lemma~\ref{lemma.cz}, it is easy to see that the only such polynomial is $c^2 z^2 = y^6z^2 + y^4z^4 + y^2 z^6$. Changing the sign of $t$ as necessary, we have:
$$ t|_{\bZ_{3\rA}{\times}\bZ_{3\rA}} = y^6z^2 + y^4z^4 + y^2 z^6, \qquad t|_{\bZ_{3\rA}{\times}\bZ_{3\rB}} = 0.$$
As for $T|_{\bZ_{3\rA}{\times}\bZ_{3\rA}}$, we need a $\GL_2(\bF_3)$-invariant degree-$15$ class satisfying $\Box T = t$. One could worry that there could be multiple choices. By Lemma~\ref{lemma.cz}, the only  degree-$15$ class in the kernel of $\Box$ which is $\GL_2(\bF_3)$-invariant \emph{up to a sign} is $w (c^2 + z^6)$, but this class changes sign under some of the involutions in $\GL_2(\bF_3)$. So $T$ is uniquely determined by invariance and $\Box T = t$. We have:
$$ T|_{\bZ_{3\rA}{\times}\bZ_{3\rA}} = Y(yz^6 - y^3 z^4) + Z(y^6z - y^4 z^3), \qquad T|_{\bZ_{3\rA}{\times}\bZ_{3\rB}} = 0.$$

We next consider the degree-$12$ generator $s$. 
It must have nontrivial restrictions to both $\bZ_{3\rA}\times\bZ_{3\rA}$ and $\bZ_{3\rA}\times\bZ_{3\rB}$, since neither $rs$ nor $st$ vanishes. The restriction to $\bZ_{3\rA}\times\bZ_{3\rA}$ is a polynomial in $y$ and $z$ of degree $6$, invariant under all of $\GL_2(\bF_3)$, and so must be $c^2 + z^6 = y^6 + y^4 z^2 + y^2 z^4 + z^6$, after possibly changing the sign of $s$. The restriction to $\bZ_{3\rA}\times\bZ_{3\rB}$ is some linear combination of $c^2$ and $z^6$. But note that, upon further restriction to the fibre $\bZ_{3\rA}$, the two restrictions must agree. On the other hand, we have the freedom to change $s \mapsto s\pm r^3$ without changing the presentation of $\H^\bullet(\rM_{24};\bZ)$ in Theorem~\ref{thm.Green}. We will choose the modification so that $s|_{\bZ_{3\rB}} = 0$.
 Thus we may assume:
$$ s|_{\bZ_{3\rA}{\times}\bZ_{3\rA}} = y^6 + y^4 z^2 + y^2 z^4 + z^6, \qquad s|_{\bZ_{3\rA}{\times}\bZ_{3\rB}} = y^6 + y^4 z^2 + y^2 z^4.$$

We will consider the two degree-$11$ generators $S$ and $u$ at the same time.
It will be convenient to replace $u$ with $v = u - S$.
 Note that $\Box S = \Box v = 0$, and so the restrictions of both must land in $\bF_3[y,z,w]$, where as above $w = \Box(YZ) = yZ - Yz$.
 Theorem~\ref{thm.Green} provides $rS = 0$ and $tv = 0$. Therefore $S|_{\bZ_{3\rA}{\times}\bZ_{3\rB}} = 0$ and $v|_{\bZ_{3\rA}{\times}\bZ_{3\rA}} = 0$.
By Lemma~\ref{lemma.cz}, the other restrictions are of the form $w p(z,c)$, where the polynomial $p(z,c)$ is of homogeneous degree $4$ in $y$ and $z$. After tracking the behaviour under $z \mapsto -z$ and $y \mapsto -y$, 
the only option is $p(z,c) = zc = y^3z - yz^3$, and so:
$$ S|_{\bZ_{3\rA}{\times}\bZ_{3\rA}} = Y(yz^4 - y^3 z^2) + Z(y^4z - y^2 z^3), \qquad S|_{\bZ_{3\rA}{\times}\bZ_{3\rB}} = 0.$$
The signs of $S$ and $u$ are not independent if we want to preserve the relation $tS = tu$ from Theorem~\ref{thm.Green}, and the author was not able to identify the correct sign for the restriction of $v = u-S$. We do have:
$$ v|_{\bZ_{3\rA}{\times}\bZ_{3\rB}} = 0, \qquad v|_{\bZ_{3\rA}{\times}\bZ_{3\rB}} = \pm \bigl( Y(yz^4 - y^3 z^2) + Z(y^4z - y^2 z^3)\bigr) .$$

Last, we have the degree-$10$ generator $U$, which must satisfy $\Box U = u = v+S$.
Since $U$ is of even degree and not in the kernel of $\Box$, its restriction must be of the form $p(y,z)YZ$, for some polynomial $p$ of polynomial degree $4$. Invariance then gives the answer:
$$ U|_{\bZ_{3\rA}{\times}\bZ_{3\rA}} = YZ(y^3 z - y z^3), \qquad U|_{\bZ_{3\rA}{\times}\bZ_{3\rB}} = \pm YZ(y^3 z - y z^3).$$
The sign is the same as above, set by $\Box U = v+S$.

In summary, we have shown:
\begin{proposition}\label{prop.restrictions}
 With notation as in Theorem~\ref{thm.Green}, and writing $v = u-S$, 
the generators of $\H^\bullet(\rM_{24};\bF_3)$ have the following  restrictions to the maximal abelian subgroups $\bZ_{3\rA}\times\bZ_{3\rA}$ and $\bZ_{3\rA}\times\bZ_{3\rB}$:
  $$\begin{array}{c|c|c}
  \text{Generator} & \bZ_{3\rA}\times\bZ_{3\rA} & \bZ_{3\rA}\times\bZ_{3\rB} \\ \hline
  R & 0 & Zz \\
  r & 0 & z^2 \\
  U & YZ(y^3 z - y z^3) & \pm YZ(y^3 z - y z^3) \\
  v & 0 & \pm \bigl( Y(yz^4 - y^3 z^2) + Z(y^4z - y^2 z^3)\bigr)\\
  S & Y(yz^4 - y^3 z^2) + Z(y^4z - y^2 z^3) & 0 \\
  s & y^6 + y^4 z^2 + y^2 z^4 + z^6 & y^6 + y^4 z^2 + y^2 z^4\\
  T & Y(yz^6 - y^3 z^4) + Z(y^6z - y^4 z^3)& 0\\
  t & y^6z^2 + y^4z^4 + y^2 z^6 & 0\\
  \end{array}
  $$
The sign $\pm$ is the same throughout.\end{proposition}

\subsection{Cohomology of $\cP + \epsilon r$ acting on $\H^\bullet(\rM_{24};\bF_3)$}\label{subsec.cohomofD}

Our next goal is to understand in detail the action of the first Steenrod cube $\cP$ on $\H^\bullet(\rM_{24};\bF_3)$. This is a degree-$4$ universal cohomology operation defined on all $\bF_3$-cohomology rings. Often the notation $\cP^k$ is used to denote the $k$th Steenrod cube. We will avoid that notation because it does not mean the $k$th power of $\cP$, writing instead $\cP^{(k)}$ for the $k$th Steenrod cube, and $\cP^{\circ k}$ for the $k$-fold composition.

 Among the defining properties of $\cP$ are that it is a derivation, that it vanishes in cohomological degree $1$, and that it is the cube in cohomological degree $2$. For example, on $\bZ_3 \times \bZ_3$, it satisfies
 $$ \cP(Y) = \cP(Z) = 0, \qquad \cP(y) = y^3, \qquad \cP(z) = z^3.$$
Propositions~\ref{prop.MT} and~\ref{prop.restrictions} then provide enough information to work out the action of $\cP$ on the generators of $\H^\bullet(\rM_{24};\bF_3)$:
  $$\begin{array}{c|c|c}
  \;\;g\;\; & \cP(g)|_{\bZ_{3\rA}\times\bZ_{3\rA}} & \cP(g)|_{\bZ_{3\rA}\times\bZ_{3\rB}} \\ \hline
  R & 0 & Zz^3 \\
  r & 0 & -z^4 \\
  U & 0 & 0 \\
  v & 0 & \pm \bigl( Y(-y^3 z^4 + y z^6) + Z(y^6 z - z^4 z^3) \bigr) \\
  S & Y(-y^3 z^4 + y z^6) + Z(y^6 z - z^4 z^3) & 0 \\
  s & y^6 z^2 + y^4 z^4 + y^2 z^6 & y^6 z^2 + y^4 z^4 + y^2 z^6 \\
  T & 0 & 0 \\
  t & 0 & 0 \\
  \end{array}
  $$
From this we learn:
\begin{proposition}\label{prop.Paction}
 With notation as in Proposition~\ref{prop.restrictions}, 
 the first Steenrod cube $\cP$ acts as:
 \begin{gather*} \cP(R)  = Rr, \qquad \cP(U)  = 0, \qquad \cP(S)  = T, \qquad \cP(T)  = 0,
 \\ 
 \cP(r) = -r^2, \qquad \cP(v)  = vr \pm Rs, \qquad \cP(s)  = sr + t, \qquad \cP(t)  = 0.
 \end{gather*}
\end{proposition}
Recalling that $u = v+S$, we note that $$\cP(u) = vr \pm Rs + T = (u-S)r \pm Rs + T = ur + T \pm Rs,$$ since $Sr = 0$. 
\begin{proof}
  The most interesting case is $\cP(v)$; the other cases are left to the reader. Recall that $v$, and hence $\cP(v)$, vanishes when restricted to $\bZ_{3\rA} \times \bZ_{3\rA}$. Their other restrictions are
  \begin{gather*}
   v|_{\bZ_{3\rA} \times \bZ_{3\rB}} = \pm \bigl( Y(yz^4 - y^3 z^2) + Z(y^4z - y^2 z^3)\bigr),\\
   \cP(v)|_{\bZ_{3\rA} \times \bZ_{3\rB}} = \pm \bigl( Y(yz^6 - y^3 z^4) + Z(y^6 z - y^4 z^3) \bigr),
  \end{gather*}
  where for example we use $\cP(yz^4 - y^3 z^2) = y^3z^4 + 4yz^6 - 3y^5 z^2 - 2y^3 z^4 = yz^6 - y^3 z^4$. Comparing $\cP(v)|_{\bZ_{3\rA} \times \bZ_{3\rB}}$ with $vr|_{\bZ_{3\rA} \times \bZ_{3\rB}} = vz^2$, we find a discrepency:
  $$ (\cP(v) - vr)|_{\bZ_{3\rA} \times \bZ_{3\rB}} = \pm Z\bigl(  (y^6 z - y^4 z^3) - z^2(y^4z - y^2 z^3)\bigr) = \pm Z(y^6z + z^4 z^3 + y^2 z^5).$$
  Factoring out $Zz = R|_{\bZ_{3\rA} \times \bZ_{3\rB}}$ gives $\pm s|_{\bZ_{3\rA} \times \bZ_{3\rB}}$.
\end{proof}

Fix some $\epsilon \in \bF_3$. The specific understanding that we seek is the following: we will calculate  the ``cohomology'' of the $\cD = \cP + \epsilon r$. 
This operator is not a differential in the usual sense: $\cD\circ\cD \neq 0$. Rather, it is a \define{$3$-differential} in the sense that its {cube} is zero:
\begin{lemma}\label{lemma.Dis3dif}
  Let $X$ be a space, and choose $x \in \H^4(X;\bF_3)$. Then the operator $\cD = \cP + x$ on $\H^\bullet(X;\bF_3)$ satisfies $\cD^{\circ 3} = 0$.
\end{lemma}
\begin{proof}
  Expanding $(\cP + x)^{\circ 3}$, we have:
  \begin{align*}
    \cD^{\circ 3} & = \cP^{\circ 3} + \cP^{\circ 2} \circ x + \cP \circ x \circ \cP + x \circ \cP^{\circ 2} + \cP \circ x^2 + x \circ  \cP \circ x + x^2 \circ \cP + x^3.
  \end{align*}
  The Adem relations provide $\cP^{\circ 3} = 0$  and  $\cP^{\circ 2} = -\cP^{(2)}$, where $\cP^{(2)}$ denotes the second Steenrod power.  The statement ``$\cP$ is a derivation'' may be summarized as:
  $$\cP \circ x = x \cP + \cP(x).$$
  Here and in the sequel, by ``$x\cP$'' we mean of course $x \circ \cP$, i.e.\ apply $\cP$ and then multiply by~$x$, whereas ``$\cP(x)$'' means multiplication by $\cP(x)$.
  Thus we find:
  \begin{multline*}
   \cD^{\circ 3}  = 0 + (\cP^{\circ 2}(x) + 2 \cP(x) \cP + x \cP^{\circ 2}) + (\cP(x) \cP + x \cP^{\circ 2}) + x \cP^{\circ 2} \\
   + (2x\cP(x) + x^2 \cP) + (x \cP(x) + x^2 \cP) + x^2 \cP + x^3.\end{multline*}
   Since we are in characteristic $3$, everything cancels to
   $$ \cD^{\circ 3} = \cP^{\circ 2}(x) + x^3.$$
   But $\cP^{\circ 2}(x) = -\cP^{(2)}(x) = -x^3$ since $x$ has degree $4$.
\end{proof}

 Appendix~\ref{appendix-3complex} contains an extended discussion of $3$- and higher differentials. 
For our purposes, it suffices to record the following. A usual differential on a $\bK$-vector space makes it into a module for the algebra $\bK[\cD]/(\cD^2)$; the usual cohomology is the result of decomposing the module as a direct sum of indecomposables, and discarding the free summands. We have instead an action of $\bF_3[\cD]/(\cD^3)$ on a vector space $V$, and its \define{cohomology} $\rH^*(V,\cD)$ is the result of decomposing $V$ as a sum of indecomposable $\bF_3[\cD]/(\cD^3)$-modules and discarding the free summands. Whereas in the usual case the non-free indecomposable module is unique, over $\bF_3[\cD]/(\cD^3)$ there are two isomorphism classes of non-free indecomposable modules, of $\bF_3$-dimensions $1$ and $2$. So $\rH^*(V,\cD)$ is not just a vector space, but rather picks up a $\bZ_2$-grading in addition to any cohomological grading on $V$. One of the punchlines of Appendix~\ref{appendix-3complex} is that this $\bZ_2$-grading is really a fermionic grading: $\rH^*$ converts the two-dimensional indecomposable $\bF_3[\cD]/(\cD^3)$-module into an odd line $\bF_3^{0|1}$. The other punchline is that $\rH^*$ is functorial and symmetric monoidal.

To illustrate this, and as a warm-up to the $\rM_{24}$-case that we care about, let us study the cohomology of $\cP + \epsilon a^2$ on $\H^\bullet(\bZ_3;\bF_3) = \bF_3[A,a]$, where $A$ has degree $1$ and $a = \Box A$ has degree $2$. Since $\cP(A) = 0$, multiplication by $A$ is an isomorphism of $3$-complexes between the even-degree cohomology $\bF_3[a]$ and the odd-degree cohomology $A\bF_3[a]$. So it suffices to understand the cohomology of $\cD = a^3\frac\partial{\partial a} + \epsilon a^2$ on $\bF_3[a]$, where $\epsilon \in \bF_3$. On monomials, we have $\cD(a^i) = (i+\epsilon)a^{i+2}$. For $\epsilon = 0$, this complex looks like:
$$\begin{tikzpicture}[anchor=base]
  \path
  (0,1) node (a0) {$a^0$}
  (1,0) node (a1) {$a^1$}
  (2,1) node (a2) {$a^2$}
  (3,0) node (a3) {$a^3$}
  (4,1) node (a4) {$a^4$}
  (5,0) node (a5) {$a^5$}
  (6,1) node (a6) {$a^6$}
  (7,0) node (a7) {$a^7$}
  (8,1) node (a8) {$a^8$}
  (9,0) node (a9) {$a^9$}
  (10,1) node (a10) {$a^{10}$}
  (11,0) node (a11) {$a^{11}$}
  (12,1) node (a12) {$\cdots$}
  (13,0) node (a13) {$\cdots$}
  ;
  \draw[->] (a1) -- node[auto] {$\scriptstyle +1$} (a3);
  \draw[->] (a2) -- node[auto] {$\scriptstyle -1$} (a4);
  \draw[->] (a4) -- node[auto] {$\scriptstyle +1$} (a6);
  \draw[->] (a5) -- node[auto] {$\scriptstyle -1$} (a7);
  \draw[->] (a7) -- node[auto] {$\scriptstyle +1$} (a9);
  \draw[->] (a8) -- node[auto] {$\scriptstyle -1$} (a10);
  \draw[->] (a10) -- node[auto] {$\scriptstyle +1$} (a12.west |- a10);
  \draw[->] (a11) -- node[auto] {$\scriptstyle -1$} (a13.west |- a11);
\end{tikzpicture}$$
The non-free summands---the cohomology---are $\{a^0\}$ and $\{a^1 \to a^3\}$.
For $\epsilon=1$, we have instead:
$$\begin{tikzpicture}[anchor=base]
  \path
  (0,1) node (a0) {$a^0$}
  (1,0) node (a1) {$a^1$}
  (2,1) node (a2) {$a^2$}
  (3,0) node (a3) {$a^3$}
  (4,1) node (a4) {$a^4$}
  (5,0) node (a5) {$a^5$}
  (6,1) node (a6) {$a^6$}
  (7,0) node (a7) {$a^7$}
  (8,1) node (a8) {$a^8$}
  (9,0) node (a9) {$a^9$}
  (10,1) node (a10) {$a^{10}$}
  (11,0) node (a11) {$a^{11}$}
  (12,1) node (a12) {$\cdots$}
  (13,0) node (a13) {$\cdots$}
  ;
  \draw[->] (a0) -- node[auto] {$\scriptstyle +1$} (a2);
  \draw[->] (a1) -- node[auto] {$\scriptstyle -1$} (a3);
  \draw[->] (a3) -- node[auto] {$\scriptstyle +1$} (a5);
  \draw[->] (a4) -- node[auto] {$\scriptstyle -1$} (a6);
  \draw[->] (a6) -- node[auto] {$\scriptstyle +1$} (a8);
  \draw[->] (a7) -- node[auto] {$\scriptstyle -1$} (a9);
  \draw[->] (a9) -- node[auto] {$\scriptstyle +1$} (a11.west |- a9);
  \draw[->] (a10) -- node[auto] {$\scriptstyle -1$} (a12.west |- a10);
\end{tikzpicture}$$
The cohomology is just $\{a^0 \to a^2\}$. Finally, for $\epsilon = -1$, we see:
$$\begin{tikzpicture}[anchor=base]
  \path
  (0,1) node (a0) {$a^0$}
  (1,0) node (a1) {$a^1$}
  (2,1) node (a2) {$a^2$}
  (3,0) node (a3) {$a^3$}
  (4,1) node (a4) {$a^4$}
  (5,0) node (a5) {$a^5$}
  (6,1) node (a6) {$a^6$}
  (7,0) node (a7) {$a^7$}
  (8,1) node (a8) {$a^8$}
  (9,0) node (a9) {$a^9$}
  (10,1) node (a10) {$a^{10}$}
  (11,0) node (a11) {$a^{11}$}
  (12,1) node (a12) {$\cdots$}
  (13,0) node (a13) {$\cdots$}
  ;
  \draw[->] (a0) -- node[auto] {$\scriptstyle -1$} (a2);
  \draw[->] (a2) -- node[auto] {$\scriptstyle +1$} (a4);
  \draw[->] (a3) -- node[auto] {$\scriptstyle -1$} (a5);
  \draw[->] (a5) -- node[auto] {$\scriptstyle +1$} (a7);
  \draw[->] (a6) -- node[auto] {$\scriptstyle -1$} (a8);
  \draw[->] (a8) -- node[auto] {$\scriptstyle +1$} (a10);
  \draw[->] (a9) -- node[auto] {$\scriptstyle -1$} (a11);
  \draw[->] (a11) -- node[auto] {$\scriptstyle +1$} (a13.west |- a11);
\end{tikzpicture}$$
The cohomology is $\{a^1\}$. In all cases, the ``tails'' are exact: the only cohomology is near the ``head'' $a^0$.

We now turn to the case we care about, which is the cohomology of $\cD = \cP + \epsilon r$ on $\H^\bullet(\rM_{24};\bF_3)$. Note that $\cD$ preserves the cohomological degree mod $4$. It also preserves an \define{auxiliary degree} defined by assigning auxiliary degree $0$ to $R$ and $r$ and auxiliary degree $+1$ to $U,v,S,s,T,$ and $t$.

The following monomials are a basis of the submodule of $\H^\bullet(\rM_{24};\bF_3)$ in cohomological degree $0\pmod 4$ and auxiliary degree $j$:
$$s^j, \qquad s^j r^i, \; 0 < i, \qquad s^{j-i} t^i, \; 0 < i \leq j.$$
The action of $\cD = \cP + \epsilon r$ is:
$$\begin{tikzpicture}[anchor=base,xscale=2,auto]
\path
(.25,0) node (s) {$s^j$}
(1,1) node (sb1) {$s^jb$} (2,1) node (sb2) {$s^jb^2$} (3,1) node (sb3) {$s^jb^3$} (4,1) node (sb4) {$s^jb^4$}  (5,1) node (sb5) {$\cdots$}
(1,-1) node (st1) {$s^{j-1}t$} (2,-1) node (st2) {$s^{j-2}t^2$} (3,-1) node (st3) {$s^{j-3}t^3$} (4,-1) node (st4) {$s^{j-4}t^4$} (5,-1) node (st5) {$\cdots$} (5.75,-1) node (st6) {$t^j$}  
;
\draw[->] (s) -- node {$\scriptstyle j+\epsilon$} (sb1);
\draw[->] (sb1) -- node {$\scriptstyle j+\epsilon-1$} (sb2);
\draw[->] (sb2) -- node {$\scriptstyle j+\epsilon-2$} (sb3);
\draw[->] (sb3) -- node {$\scriptstyle j+\epsilon-3$} (sb4);
\draw[->] (sb4) -- node {$\scriptstyle j+\epsilon-4$} (sb5.west |- sb4);
\draw[->] (s) -- node {$\scriptstyle j$} (st1);
\draw[->] (st1) -- node {$\scriptstyle j-1$} (st2);
\draw[->] (st2) -- node {$\scriptstyle j-2$} (st3);
\draw[->] (st3) -- node {$\scriptstyle j-3$} (st4);
\draw[->] (st4) -- node {$\scriptstyle j-4$} (st5.west |- st4);
\draw[->] (st5.east |- st6) -- node {$\scriptstyle 1$} (st6);
\end{tikzpicture}$$
As in the warm-up example, the tails are exact. The cohomology near the head depends on the values of both $j$ and $\epsilon$ mod $3$. Going through all nine cases, we find $\cD$-cohomology in the following cohomological degrees:
$$ \begin{array}{cr||c|c|c}
 && \multicolumn{3}{c}{j \pmod 3} \\
      && 0 & 1 & -1 \\ \hline \hline
\multirow{3}{*}{$\epsilon$} &  0  & \{12j\} & \{12j \to 12j+4\} \oplus \{12j+4\} & \{12j+4 \to 12j+8\} \\ \cline{2-5}
  & 1 & \{12j \to 12j+4\} & \{12j+4\} & \emptyset \\ \cline{2-5}
 & -1  & \emptyset & \{12j \to 12j+4\} & \{12j+4\} \\ 
\end{array}$$
For example, when $\epsilon = 0$ and $j = 1 \pmod 3$, cutting off the manifestly exact tails returns
$$\begin{tikzpicture}[anchor=base,xscale=2,auto]
\path
(.25,0) node (s) {$s^j$}
(1,1) node (sb1) {$s^jb$} (2,1) 
(1,-1) node (st1) {$s^{j-1}t$} (2,-1) 
;
\draw[->] (s) -- node {$\scriptstyle 1$} (sb1);
\draw[->] (s) -- node {$\scriptstyle 1$} (st1);
\end{tikzpicture}$$
The submodule $\{s^j \to s^jb + s^{j-1}t\}$ splits off as a direct summand, and the other summand is one-dimensional, generated by the cohomology class $[s^jb]$, which is cohomologous to $-[s^{j-1}t]$.

Multiplication by $U$ is an isomorphism between the subspace of $\H^\bullet(\rM_{24};\bF_3)$ of cohomological degree $0\pmod 4$ and the subspace of cohomological degree $2\pmod 4$. Since $\cP(U) = 0$, this isomorphism is in fact an isomorphism of $\bF[\cD]/(\cD^3)$-modules. Thus we immediately learn that the $\cD$-cohomology in cohomological degree $2\pmod 4$ consists of:
$$ \begin{array}{cr||c|c|c}
 && \multicolumn{3}{c}{j \pmod 3} \\
      && 0 & 1 & -1 \\ \hline \hline
\multirow{3}{*}{$\epsilon$} &  0  & \{12j+10\} & \{12j+10 \to 12j+14\} \oplus \{12j+14\} & \{12j+14 \to 12j+18\} \\ \cline{2-5}
  & 1 & \{12j+10 \to 12j+14\} & \{12j+14\} & \emptyset \\ \cline{2-5}
 & -1  & \emptyset & \{12j+10 \to 12j+14\} & \{12j+14\} \\ 
\end{array}$$

Theorem~\ref{thm.Green} implies that $\H^\bullet(\rM_{24};\bF_3)$ vanishes in cohomological degree $1\pmod 4$. The only remaining case is cohomological degree $3\pmod 4$. Monomials of cohomological degree $3\pmod 4$ are divisible by exactly one of $R,v,S,T$. The first two vanish when restricted to $\bZ_{3\rA} \times \bZ_{3\rA}$, and the second two vanish when restricted to $\bZ_{3\rA} \times \bZ_{3\rB}$. This allows us to split the $\bullet = 1\pmod 4$ subcomplex of $\H^\bullet(\rM_{24};\bF_3)$ into two summands: the kernel of restriction to $\bZ_{3\rA} \times \bZ_{3\rA}$ and the kernel of restriction to $\bZ_{3\rA} \times \bZ_{3\rB}$. (These kernels are disjoint by Proposition~\ref{prop.MT}.)

The first summand consists of those terms divisible by $R$ or $v$.
In auxiliary degree $j+1 \geq 1$, it looks like (the totalization of):
$$\begin{tikzpicture}[anchor=base,auto,scale=1.5]
\path
(0,0) node (v0) {$s^jv$} ++(2,0) node (v1) {$s^jvr$} ++(2,0) node (v2) {$s^jvr^2$} ++(2,0) node (v3) {$s^jvr^3$} ++(2,0) node (v4) {$\cdots$}
(0,-1) node (R0) {$s^{j+1}R$} ++(2,0) node (R1) {$s^{j+1}Rr$} ++(2,0) node (R2) {$s^{j+1}Rr^2$} ++(2,0) node (R3) {$s^{j+1}Rr^3$} ++(2,0) node (R4) {$\cdots$}
;
\draw[->] (v0) -- node {$\scriptstyle \pm1$} (R0);
\draw[->] (v1) -- node {$\scriptstyle \pm1$} (R1);
\draw[->] (v2) -- node {$\scriptstyle \pm1$} (R2);
\draw[->] (v3) -- node {$\scriptstyle \pm1$} (R3);
\draw[->] (v0) -- node {$\scriptstyle j+\epsilon+1$} (v1);
\draw[->] (v1) -- node {$\scriptstyle j+\epsilon$} (v2);
\draw[->] (v2) -- node {$\scriptstyle j+\epsilon-1$} (v3);
\draw[->] (v3) -- node {$\scriptstyle j+\epsilon-2$} (v4.west |- v3);
\draw[->] (R0) -- node {$\scriptstyle j+\epsilon+2$} (R1);
\draw[->] (R1) -- node {$\scriptstyle j+\epsilon+1$} (R2);
\draw[->] (R2) -- node {$\scriptstyle j+\epsilon$} (R3);
\draw[->] (R3) -- node {$\scriptstyle j+\epsilon-1$} (R4.west |- R3);
\end{tikzpicture}$$
The sign is the same in all vertical arrows. The reader is invited to check that the cohomology of  this complex:
\begin{itemize}
  \item Vanishes when $j+\epsilon = -1 \pmod 3$.
  \item Has one-dimensional cohomology in degree $\{12j+15\}$ when $j+\epsilon = 1 \pmod 3$.
  \item Has odd-one-dimensional cohomology in degree $\{12j+11 \to 12j+15\}$ when $j+\epsilon = 0 \pmod 3$.
\end{itemize}
In auxiliary degree $j+1=0$, we just see
$$R \overset{\epsilon+1}\longto Rr \overset{\epsilon}\longto Rr^2 \overset{\epsilon-1}\longto \cdots$$
which instead has cohomology in degrees:
\begin{itemize}
  \item $\{3\to 7\}$ if $\epsilon=0$. Since $j+1=0$, this replaces the $j+\epsilon=-1$ entry.
  \item $\emptyset$ if $\epsilon = 1$. Since $j+1=0$, this replaces the $j+\epsilon=0$ entry, which would have been the nonsensical cohomological degrees $\{-1 \to 3\}$ in any case.
  \item $\{3\}$ if $\epsilon = -1$. Since $j+1=0$, this agrees with the $j+\epsilon = 1$ entry.
\end{itemize}

Finally, we have the summand 
 consisting of those monomials divisible by $S$ or $T$. Note that $\cP(S) = T$ and $\cP(T) = 0$, and $Sr = Tr = 0$. So, in auxiliary degree $j+1$, we may factor the total complex as a tensor product:
$$\{s^j \overset j \longto s^{j-1}t \overset{j-1} \longto \cdots \overset 2 \longto st^{j-1} \overset 1 \longto t^j\} \otimes \{S \overset{+1}\longto T\}
$$
As explained in Appendix~\ref{appendix-3complex}, the functor $\rH^*$ that takes cohomology of $3$-complexes is symmetric monoidal. $\rH^*(\{S \overset{+1}\longto T\}) = \{S\to T\}$ is supported in cohomological degrees $\{11 \to 15\}$. The first tensorand $\{s^j \to \dots \to t^j\}$ is exact except near the head, where its cohomology depends on $j$:
\begin{itemize}
  \item If $j = 2 \pmod 3$, then $\{s^j \to \dots \to t^j\}$ is exact, and so $\{s^j \to \dots \to t^j\} \otimes \{S \to T\}$ is exact.
  \item If $j = 1 \pmod 3$, then $\rH^*(\{s^j \to \dots \to t^j\}) = \{s^j \to s^{j-1}t\}$ is supported in degrees $\{12j \to 12j+4\}$. So $\rH^*(\{s^j \to \dots \to t^j\} \otimes \{S\to T\}) = \{12j+15\}$ in cohomological degree.
  \item If $j = 0 \pmod 3$, then $\rH^*(\{s^j \to \dots \to t^j\}) = \{12j\}$ and $\rH^*(\{s^j \to \dots \to t^j\}\otimes \{S\to T\}) = \{12j+11 \to 12j+15\}$.
\end{itemize}
These cohomologies are independent of $\epsilon$.

Summarizing the above computations, we have:
 
\begin{proposition}\label{prop.cohD}
  Pick $\epsilon \in \bF_3$, and consider the $3$-differential $\cD = \cP+\epsilon r$ acting on $V = \H^\bullet(\rM_{24};\bF_3)$. 
  Its cohomology $\rH^*(V,\cD)$ is supported in the following cohomological degrees modulo $36$:
  \begin{itemize}
   \item $\epsilon = 0:$ $\{0\}$, $\{2 \to 6\}$, $\{10\}$, $\{11\to15\}$, $\{11\to15\}$, $\{12 \to 16\}$, $\{16\}$, $\{22\to 26\}$, $\{26\}$, $\{27\}$, $\{27\}$, $\{28 \to 32\}$, $\{35\to39\}$.
   \item $\epsilon = 1:$ $\{0 \to 4\}$, $\{10\to14\}$, $\{11\to15\}$, $\{15\}$, $\{16\}$, $\{26\}$, $\{27\}$, $\{35 \to 39\}$.
   \item $\epsilon = -1:$ $\{2\}$, $\{3\}$, $\{11\to15\}$, $\{12\to 16\}$, $\{22\to26\}$, $\{23\to27\}$, $\{27\}$, $\{28\}$, $\{35\to39\}$.
  \end{itemize}
  The repeated terms in the $\epsilon=0$ line indicate that the cohomology contains two summands in those degrees.
  When $\epsilon=0$, there is also one nonperiodic cohomology class in degree $\{3\to7\}$. Except for that one non-periodic class, all other cohomology is periodic with periodicity element $s^3$.

  On the subalgebra $\bF_3[R,r] \subset \H^\bullet(\rM_{24};\bF_3)$, $\cD = \cP+\epsilon r$ has the following cohomology:
  \begin{itemize}
    \item $\epsilon = 0:$ $\{r^0\}$ in degree $0$ and $\{R \to Rr\}$ in degree $\{3\to 7\}$.
    \item $\epsilon = 1:$ $\{r^0 \to r^1\}$ in degree $\{0\to4\}$.
    \item $\epsilon = -1:$ $\{R\}$ in degree $\{3\}$.
  \end{itemize}
\end{proposition}

\section{The Atiyah--Hirzebruch spectral sequence for $\tmf^\bullet_\omega(B\rM_{24})$} \label{sec.AHSS}

With the ordinary $3$-local cohomology of $B\rM_{24}$ understood, we are now ready to analyze the $3$-local structure of $\tmf^\bullet_\omega(B\rM_{24})$. We will do so by analyzing its Atiyah--Hirzebruch spectral sequence.

\subsection{General comments about AHSSs}\label{ahss.general}

Any space $X$ and spectrum $\cE^\bullet$ determine an \define{Atiyah--Hirzebruch spectral sequence}. (The name refers to~\cite{MR0139181} but the proof there consists essentially of a reference to~\cite{MR0077480}, and~\cite{MR0402720} attributes the construction to unpublished work of Whitehead.
 Further important early developments are in~\cite{MR150765}.) Spectral sequences are an algebrotopolical version of perturbation theory. As with any perturbative calculation, the goal is to approximate some nonperturbative object. In the AHSS case, that nonperturbative object is $\cE^\bullet(X)$, the $\cE^\bullet$-cohomology of the space $X$.

The rough idea of the AHSS is the following. Imagine that $X$ is a CW complex and that you have fixed some cochain model for $\cE^\bullet$. Then a cochain for $\cE^n(X)$ assigns, to each $m$-cell in~$X$, an element of $\cE^{n-m}$. The cohomology $\cE^n(X)$ is the cohomology for some total differential on this set of cochains. The AHSS perturbatively approximates that total differential. The $0$th approximation entirely ignores the topology of~$X$: for a cell $x\in X$ and a cochain $x \mapsto e(x)$, the $0$th approximation is $(d_0e)(x) = \d_{\cE} (e(x))$, where $\d_{\cE}$ is the differential in (the cochain model for) $\cE^\bullet$. The $1$st approximation includes some attaching information for the cells in~$X$. By definition, the \define{$E_k$-page} of a spectral sequence is the cohomology of the $(k-1)$th approximation of the total differential. It is always bigraded by the dimension $m$ of a cell in $X$ and the $E$-degree $n$ of a cochain.
 In the AHSS case, the $E_2$ page is
$$E_2^{m,n} = \H^m(X; \cE^n(\pt)).$$
As one ``turns the page,'' one includes higher-order corrections, which take into account how the homotopy groups $\cE^\bullet(\pt)$ are connected. 
The result is an infinite sequence of finer and finer approximations to $\cE^\bullet(X)$.

Any time one works perturbatively, one must worry about two related problems:
\begin{itemize}
  \item Does the perturbative expansion converge at all?
  \item Does the perturbative expansion  converge to the  object one wishes to compute?
\end{itemize}
In particular, perhaps there are ``nonperturbative effects'' not seen in the perturbative expansion, so that it does not in fact calculate the desired result.

In the case of AHSSs, this second problem is present as soon as one tries to extend from spaces to stacks. Indeed, suppose $\bX$ is a stack with classifying space  $X$. Then there is an AHSS which tries to approximate $\cE^\bullet(\bX)$, with $E_2$ page $\H^\bullet(\bX; \cE^\bullet(\pt))$. But ordinary cohomology does not distinguish stacks from spaces: $\H^\bullet(\bX; \cE^\bullet(\pt)) = \H^\bullet(X; \cE^\bullet(\pt))$. The higher differentials also do not distinguish $\bX$ from $X$. As a result, this AHSS will at best converge to $\cE^\bullet(X)$, which in general is not isomorphic to $\cE^\bullet(\bX)$ (compare \S\ref{subsec.tmf}).

Most textbooks confirm convergence of AHSSs in only very limited circumstances: when the $\cE^\bullet(\pt)$ is bounded below, or when $\H^\bullet(X;\bZ)$ is bounded above. More general convergence results are available in~\cite{MR1718076}. In particular, Theorem~12.4 of that paper says that the AHSS for $\cE^\bullet(X)$ does indeed converges ``conditionally'' to $\cE^\bullet(X)$, for any spectrum $\cE$ and space $X$. (The convergence is in the ``colimit'' topology. We will not here discuss the different topologies in which the convergence might hold.)
Theorem~7.1 of that paper gives conditions under which this ``conditional'' convergence is in fact ``strong,'' which is what one wants for applications. 
In particular, as explained in the remark following Theorem~7.1 of that paper, for a conditionally convergent spectral sequence to converge strongly, it suffices if  each entry $E^{m,n}$ supports only finitely many nonzero differentials, i.e.\ if there are only finitely many $k$ for which $\d_k : E_k^{m,n} \to E_k^{m+k,n-k+1}$ is nonzero. In particular, we find:

\begin{lemma}\label{lemma.AHSS1}
  Suppose $\cE$ is a spectrum all of whose homotopy groups $\cE^\bullet(\pt)$ are finitely generated as abelian groups, and suppose that $X = BG$ is the classifying space of a finite group. Then the AHSS $\H^\bullet(BG; \cE^\bullet(\pt)) \Rightarrow \cE^{\bullet}(BG)$ converges strongly.
\end{lemma}
\begin{proof}
  Since $G$ is finite and $\cE^n(\pt)$ is finitely generated, $E_2^{m,n} = \H^m(BG; E^n(\pt))$ is finite if $m>0$. It therefore can support only finitely many differentials. 
  
  And, other than $\d_0$ and $\d_1$, the $m=0$ column supports no differentials at all! Indeed, the $\d_2$ and higher differentials in an AHSS are always \emph{stable} cohomology operations, and stable cohomology operations always vanish in degree $m=0$.
\end{proof}

The proof of Lemma~\ref{lemma.AHSS1} does not automatically apply for cohomology with twisted coefficients, because the differentials in the twisted case can involve multiplication and higher Massey products with the twisting parameter. (The $m>0$ part of the proof still applies.) The K-theory case is explored in detail in~\cite{MR2307274}, where all higher differentials are computed: for twisting parameter $\alpha \in \H^3(X;\bZ)$, the $\d_{2k}$ differentials in the AHSS $\H^\bullet(X; \operatorname{KU}(\pt)) \Rightarrow \operatorname{KU}^\bullet_\alpha(\pt)$ vanish, and the $\d_{2k+1}$ differential is a stable operation plus a $k$-fold Massey product with $\alpha$. But Massey products (other than the ordinary product) vanish in degree $m=0$, and so again the $m=0$ column supports only finitely many (namely, one) differential, and the AHSS converges strongly.

The overall message of~\cite{MR1718076} is that one should run spectral sequences without worrying too much about convergence, and then check convergence at the end. This is because, for a conditionally convergent spectral sequence, strong convergence is simply a property of the sequence itself. Following this advice, we will compute the first few differentials in the AHSS for $\tmf^\bullet_\omega(B\rM_{24})$. This will be enough for us to confirm in Corollary~\ref{cor.main2} that the AHSS converges.

\subsection{Review of $\tmf^\bullet(\pt)[\frac12]$}\label{subsec.tmfpt}

The first step towards constructing the AHSS $\H^\bullet(\rM_{24};\tmf^\bullet(\pt)[\frac12]) \Rightarrow \tmf^\bullet_\omega(B\rM_{24})[\frac12]$ is to understand the coefficient ring $\tmf^\bullet(\pt)[\frac12]$. An excellent reference for looking up information about this ring is the chapter~\cite{Andre-tmfpt} of~\cite{MR3223024}, and~\cite{Akhil-homotopytmf} provides a nice survey of how the computations are done.

Write $c_w$ for the weight-$w$ Eisenstein series, normalized so that $c_w(q) = 1 + O(q)$. Recall that the ring $\mf$ of integral $\mathrm{SL}(2,\bZ)$-modular forms, which are ``holomorphic'' in the sense of being bounded at the cusp $\tau = i\infty$, is 
$$ \mf = \bZ[c_4,c_6,\Delta] / (c_4^3 - c_6^2 - 1728\Delta).$$
In particular, after inverting $6$, we have $\mf[\frac16] = \bZ[\frac16][c_4,c_6]$. As in \S\ref{subsec.tmf}, we will think of $\mf$ as a cohomologically-graded ring $\mf^\bullet$, with the modular forms of weight $w$ in cohomological degree $-2w$. (Our insistence of working with cohomological gradings means that $\mf^\bullet$ is supported in nonpositive degrees.) Justifying the names, there is a ring map $\tmf^\bullet(\pt) \to \mf^\bullet$. It is an isomorphism away from $6$:
$$\textstyle \tmf^\bullet(\pt)[\frac16] \isom \mf^\bullet[\frac16] = \bZ[\frac16][c_4,c_6].$$
It follows that the map $\tmf^\bullet(\pt) \to \mf^\bullet$ has kernel exactly the torsion in $\tmf^\bullet(\pt)$. It is traditional to name non-torsion classes in $\tmf^\bullet(\pt)$ by their images in $\mf^\bullet$.

(So far as the author knows, there is no \emph{interesting} spectrum $\cE^\bullet$ with homotopy groups $\cE^\bullet(\pt) = \mf^\bullet$: the only such spectrum is  a product of Eilenberg--Mac Lane spaces, and represents $\H^\bullet(-;\mf^\bullet)$. The map on coefficients $\tmf^\bullet(\pt) \to \mf^\bullet$ does not lift to a spectrum map $\tmf^\bullet \to \H^\bullet(-;\mf^\bullet)$.)

We will keep $2$ inverted, and describe $\tmf^\bullet(\pt)[\frac12]$ in terms of its map to $\mf^\bullet[\frac12]$. This map is almost a surjection. In particular, its image contains the Eisenstein series $c_4,c_6$, and hence powers of $27\Delta = \frac1{64}(c_4^3 - c_2^6)$. (At the prime $2$, $c_4$ is in the image of $\tmf^\bullet(\pt)$, but $c_6$ is not: only $2c_6$ is.) In fact, $m\Delta^k$ is in the image of $\tmf^\bullet(\pt)[\frac12]$ if and only if $mk = 0\pmod 3$: $\tmf^\bullet(\pt)[\frac12]$ contains nontorsion classes $\{3\Delta\}$, $\{3\Delta^2\}$, and $\Delta^3$. The curly brackets remind that $\{3\Delta\}$ is not divisible by $3$ in $\tmf^\bullet(\pt)[\frac12]$.

The kernel of the map $\tmf^\bullet(\pt)[\frac12] \to \mf^\bullet[\frac12]$, i.e.\ the torsion in $\tmf^\bullet(\pt)[\frac12]$, is periodic of cohomological degree $72$, with periodicity element $\Delta^3$. 
All torsion in $\tmf^\bullet(\pt)[\frac12]$ has additive order $3$.
A framed compact manifold of dimension $n$ determines a class $[M] \in \tmf^{-n}(\pt)$. In the case of a group manifold for a connected simply connected compact group~$G$, the corresponding class is represented, assuming Conjecture~\ref{conjecture.ST}, by the antiholomorphic SCFT consisting of $n = \dim G$ antichiral free Majorana--Weyl fermions, with supersymmetry encoding the bracket on the Lie algebra of $G$ \cite{WittenTMF}. In the introduction, we mentioned already the class $\nu$ represented by the 3-sphere $\mathrm{SU}(2)$. 
Note that~\cite{Andre-tmfpt} calls this class ``$\alpha$'' in the section describing the $3$-local structure of $\tmf^\bullet(\pt)$ (and ``$\nu$'' in the section describing the $2$-local structure). Another important class is represented by the $10$-dimensional group manifold $\operatorname{Spin}(5)$. For want of a better name, we will call this class ``$\mu$''; it is called ``$\beta$'' in~\cite{Andre-tmfpt}. These classes satisfy $\nu^2 = \mu^2\nu = \mu^5 = 0$ in $\tmf^\bullet(\pt)[\frac12]$. Furthermore, there is a nontrivial Massey product 
$ \langle\nu,\nu,\nu\rangle = -2\mu = \mu,$
which the reader is invited to think through geometrically by decomposing $\operatorname{Spin}(5)$ into two pieces. (Hint: use the inclusion $\mathrm{SU}(2)^2 = \operatorname{Spin}(4) \subset \operatorname{Spin}(5)$.) There are two further torsion classes $\tmf^\bullet(\pt)[\frac12]$ not in the subring generated by $\nu$ and $\mu$. The first is in degree $27$, and is called ``$\{\nu\Delta\}$,'' because it is represented by the product $\nu\Delta$ in the elliptic spectral sequence (see \S\ref{subsec.ahssp3}). The second is $\{\nu\Delta\}\mu$. These are related to $\nu$ and $\mu$ by:
$$ \langle\nu,\nu,\mu^2\rangle = \{\nu\Delta\}, \qquad \nu \{\nu\Delta\} = \mu^3 .$$
Except for the powers of $\Delta^3$, the torsion and non-torsion classes in $\tmf^\bullet[\frac12]$ do not mix: for example, $\{3\Delta\} \mu = 0$.

In summary:
\begin{proposition}\label{prop.tmftorsion}
The torsion in $\tmf^\bullet(\pt)[\frac12]$ is $72$-periodic, with periodicity given by multiplication by $\Delta^3$. In the range $0 \geq \bullet \geq -71$, it looks as follows. The boxed class is nontorsion, and the remaining classes are torsion with additive order $3$. The southwest-to-northeast edges indicate multiplication by $\nu$, and the northwest-to-southeast edges indicate (up to sign) a nontrivial Massey product $\langle\nu,\nu,-\rangle$. (The $y$-axis is otherwise insignificant.)
$$
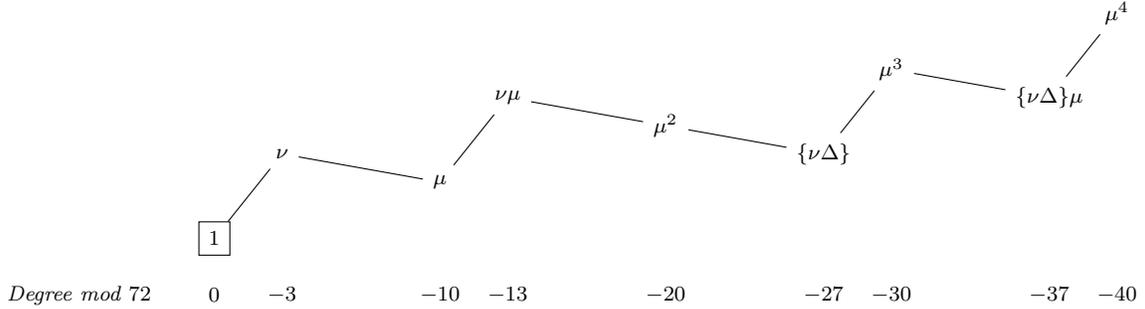
\begin{tikzpicture}[xscale=.3,yscale=.75]
\path (-6,-1) node {$\scriptstyle \text{Degree mod }72$};
\foreach \n in {3,10,13,20,27,30,37,40}
\path (\n,-1) node {$\scriptstyle -\n$};
\path (0,-1) node {$\scriptstyle 0$};
\path (0,0) node[draw] (one) {$\scriptstyle 1$};
\path (10,1) node (mu) {$\scriptstyle \mu$};
\path (20,2) node (mu2) {$\scriptstyle \mu^2$};
\path (30,3) node (mu3) {$\scriptstyle \mu^3$};
\path (40,4) node (mu4) {$\scriptstyle \mu^4$};
\path (3,1.5) node (nu) {$\scriptstyle \nu$};
\path (13,2.5) node (nu2) {$\scriptstyle \nu\mu$};
\path (27,1.5) node (nu3) {$\scriptstyle \{\nu\Delta\}$};
\path (37,2.5) node (nu4) {$\scriptstyle \{\nu\Delta\}\mu$};
\draw (one) -- (nu) -- (mu) -- (nu2) -- (mu2) -- (nu3) -- (mu3) -- (nu4) -- (mu4);
\end{tikzpicture}
$$
\end{proposition}

\subsection{Differentials for $p\geq 5$}

Our next task is to identify the early differentials in the AHSS for $\tmf^\bullet[\frac12]$. We will start first with the untwisted case and then add the twistings. This section studies the story after localizing at a prime $p\geq 5$; the $p=3$ story is in the next section. 

To warm up, let us review the analogous story for connective complex K-theory $\ku^\bullet$, due to~\cite{MR0139181,MR2172633,MR2307274}. After localizing at a prime $p\geq 3$, the coefficient ring is $\ku^\bullet_{(p)}(\pt) = \bZ_{(p)}[u]$, where $u$ has cohomological degree $-2$. As remarked in \S\ref{ahss.general}, with untwisted coefficients the differentials in the spectral sequence are necessarily stable cohomology operations. Working $p$-locally, the first stable cohomology operation is the composition
$$ \H^\bullet(-;\bZ_{(p)}) \overset{(\operatorname{mod} p)}\longto \H^\bullet(-;\bF_{p}) \overset{\cP}\longto \H^{\bullet+2p-2}(-;\bF_{p}) \overset{\Box_\bZ}\longto \H^{\bullet+2p-1}(-;\bZ_{(p)}).$$
Here by ``$\cP$'' we mean the first Steenrod $p$'th power operation, and by ``$\Box_\bZ$'' we mean the integral Bockstein (for the extension $\bZ \to \bZ \to \bZ_p$). Write $\cP_\bZ$ for this total composition. Then the first nontrivial differential in the AHSS $\H^\bullet(X;\ku^\bullet(\pt)) \Rightarrow \ku^\bullet(X)$ is
$$ \d_{2p-1} = \cP_\bZ \otimes u^{p-1}.$$
In the formula, we have identified the $E_2$ page as $\H^\bullet(X;\ku^\bullet(\pt)_{(p)}) \cong \H^\bullet(X;\bZ_{(p)}) \otimes \bZ_{(p)}[u^2]$, and  ``$u^{p-1}$''  means multiplication thereby.
In fact, the same formula works also at the prime $p=2$, with $\cP$ replaced by $\Sq^2$, so that $\cP_\bZ$ is the integral lift of $\Sq^3$. This gives the $\d_3$ differential identified in~\cite{MR0139181}. The higher differentials are similar, with $\cP$ replaced by higher Steenrod $p$th powers.

To see that $\d_{2p-1}$ is in fact a differential, note that $(\operatorname{mod} p) \circ \Box_\bZ = \Box_p$ is the mod-$p$ Bockstein (for the extension $\bZ_p \to \bZ_{p^2} \to \bZ_p$), and so:
$$ \cP_\bZ \circ \cP_\bZ = \Box_\bZ \circ \cP \circ \Box_p \circ \cP \circ (\operatorname{mod} p).$$
But an Adem relation says
$$ \cP \circ \Box_p \circ\cP = \cP^{(2)} \circ \Box_p + \Box_p \circ \cP^{(2)},$$
where $\cP^{(2)}$ denotes the second Steenrod power, and $\Box_p \circ (\operatorname{mod} p)$ and $\Box_\bZ \circ \Box_p$ both vanish.
The occurrence of $\cP^{(2)}$ here is related to the occurrence of $\cP^{(2)}$ in the next differential $\d_{4p-3}$.

The twisted story is only slightly more complicated. As explained in~\cite{MR2172633,MR2307274}, the K-theory of a space $X$ can be twisted by any class $\omega \in \H^3(X;\bZ)$ (and more generally, at the prime~$2$, by classes in the supercohomology $\SH^\bullet$ of~\cite{GuWen,WangGu2017}). The twisting modifies the $\d_3$ differential to
$$ \widetilde{\d}_3 = \d_3 - \omega \otimes u.$$
Higher differentials are also modified, now by Massey products with $\omega$. For example,
$$ \widetilde{\d}_5 = \d_5 - \langle\omega,\omega,-\rangle \otimes u^2.$$
Note that the Massey product $\langle\omega,\omega,-\rangle$ is not well-defined on the $E_2$-page of the AHSS, but it is well-defined on the $E_4$ page.

With the K-theory case understood, we can describe the tmf story. This is simplest if we work locally at a prime $p\geq 5$. There is a map of spectra $\cH : \tmf^\bullet \to \ku^\bullet\llbracket q\rrbracket$, which on homotopy groups takes a nontorsion class in $\tmf^\bullet(\pt)$ to its $q$-expansion (with the power of $u$ just recording the weight of the corresponding modular form). The name ``$\cH$'' is because of its physical interpretation as the map $\SQFT^\bullet \to \mathrm{KU}^\bullet(\!(q)\!)$ that sends an SQFT to its Hilbert space, with the parameter $q$ encoding the $S^1$-action that rotates the spatial circle. The existence of $\cH$ forces the values of some differentials in the AHSS for $\tmf^\bullet$, since the construction of AHSSs depends functorially on the spectrum. Indeed, suppose we are working $p$-locally, so that the $E_2$ page is $\H^\bullet(X;\bZ_{(p)})[c_4,c_6]$. As earlier, just because differentials are stable and because $\tmf^\bullet(\pt)_{(p)}$ has no torsion, the first possible nonzero differential is $\d_{2p-1}$.
Suppose that $\xi \in \H^\bullet(X;\bZ_{(p)})[c_4,c_6]$ is some class. If 
$$ \d_{2p-1}(\cH\xi) = (\cP_\bZ \otimes u^{p-1})(\cH\xi) $$
is not zero in $\ku^\bullet(X)_{(p)}\llbracket q\rrbracket$, then certainly $\d_{2p-1}(\xi)$ must also be nonzero. Indeed, we find that, in the AHSS for $\tmf^\bullet$:
$$ \d_{2p-1} = \cP_\bZ \otimes A$$
for some modular form $A$ of weight $p-1$ which maps to $u^{p-1}$ via $\cH$. Actually, that's not quite the requirement: if $\cH(A) = u^{p-1}$, then the $q$-expansion of $A$ is $1 \in \bZ_{(p)}\llbracket q\rrbracket$, which does not happen for a modular form of nonzero weight. The trick is that $\cP_\bZ$ factors through mod-$p$ reduction, and so its image consists just of $p$-torsion classes. Thus we do not need $\cH(A) = u^{p-1}$ on the nose, but only that $\cH(A) \equiv u^{p-1} \pmod p$. Said another way, the $q$-expansion of $A$ should be $1 \in \bF_p\llbracket q\rrbracket$.

Working over $\bF_p$, there is only one weight-$(p-1)$ modular form with trivial $q$-expansion, namely the \define{Hasse invariant}. When $p\geq 5$, it is liftable to an integral modular form. The standard lift is the weight-$(p-1)$ Eisenstein series $c_{p-1}$, and so we could set:
$$ \d_{2p-1} = \cP_\bZ \otimes c_{p-1}.$$
But we don't in fact need to choose a lift: different lifts differ by multiples of $p$, whereas the image of $\cP_\bZ$ is $p$-torsion, so different lifts give the same $\d_{2p-1}$ differential. Indeed, all we need is a criterion for checking liftability. Sufficient conditions are:
\begin{lemma}\label{lemma.liftability}
  A mod-$p$ modular form of weight $w$ admits an integral lift if $\H^1(\overline{\cM}_{ell}; L^{\otimes w})$ has no $p$-torsion, where $\overline{\cM}_{ell}$ is the compactified moduli stack of elliptic curves, and $L^{\otimes w}$ is the line bundle whose sections are weight-$w$ modular forms.
\end{lemma}
\begin{proof}
  This is automatic from the cohomology long exact sequence\\[6pt]
  \phantom{\ensuremath\Box}\hfill $\displaystyle \dots \overset p \longto \H^0(\overline{\cM}_{ell}; L^{\otimes k}) \overset{\operatorname{mod}p} \longto \H^0(\overline{\cM}_{ell}; L^{\otimes k}/p) \overset{\Box_\bZ}\longto \H^1(\overline{\cM}_{ell}; L^{\otimes k}) \overset p \longto  \dots.$ \hfill
\end{proof}
But the only torsion in $\H^\bullet(\overline{\cM}_{ell}; L^{\otimes k})$ is at the primes $2$ and $3$. Thus, for $p\geq 5$, in fact all mod-$p$ modular forms admit integral lifts.

\subsection{Differentials when $p=3$}\label{subsec.ahssp3}

If we try to repeat the $p\geq 5$ story at the prime $p=3$, we run into the following issue. Suppose $f \in \tmf^\bullet(\pt)$ is nontorsion, and that $x \in \H^\bullet(X; \bZ_{(3)})$. Then the $E_2$ page of the AHSS contains the class $x \otimes f$. The map $\cH$ sends this class to the class $x \otimes \cH(f)$ on the $E_2$ page of the AHSS for $\ku^\bullet(X)_{(3)}\llbracket q\rrbracket$, which supports a $\d_5$ differential sending it to
$$ x \otimes \cH(f) \mapsto \Box_\bZ\cP(x) \otimes u^2\cH(f),$$
where $\cP$ now denotes the first Steenrod cube, and we will leave out from the notation the mod-3 reduction. As above, this suggests that $x \otimes f$ should support a $\d_5$ differential of the form
$$\d_5 : x \otimes f \overset?\mapsto \Box_\bZ\cP(x) \otimes (\text{integral lift of }Af),$$
where $A$ denotes the mod-$3$ Hasse invariant.

By Lemma~\ref{lemma.liftability}, the obstruction to lifting $Af$ is measured by the class $\Box_\bZ(Af) = \alpha f \in \H^1(\overline{\cM}_{ell}; L^{\otimes k})_{(3)}$, where $\alpha = \Box_\bZ(A)$. This cohomology group turns out to vanish except when $k=2+12j$, in which case it is a $\bZ_3$ generated by $\alpha\Delta^j$. We therefore find that the above $\d_5$ differential is well-defined if $f$ is a multiple of $3$, $c_4$, $c_6$, or their translates by powers of $\Delta$.

Modulo this ideal, the nontorision subring is just $\bF_3[\Delta^3]$, and our $\d_5$ differential is not defined on classes of the form $x\otimes \Delta^{3j}$. Conveniently, it doesn't need to be. The presence of $\nu \in \tmf^{-3}(\pt)$ means that the $\tmf^\bullet$-AHSS may contain a $\d_4$-differential equal (up to an irrelevant sign) to
$$ \d_4 = \cP \otimes \nu.$$
Indeed, the fact that the map from the sphere spectrum to $\tmf^\bullet$ is an equivalence in low degrees forces the existence of such a differential. (The sphere spectrum is initial among $E_\infty$ ring spectra. This implies that, for the sphere spectrum, any differential which is allowed to be nonzero is in fact nonzero.) The $\d_5$ differential needs only to be defined on the cohomology of $\d_4$, and if $\Box_\bZ\cP(x) \neq 0$ so that ``$\d_5(x\otimes \Delta^{3j}) = \Box_\bZ\cP(x) \otimes (\text{lift of }A\Delta^{3j})$'' is undefined, then 
$$ \d_4(x\otimes \Delta^{3j}) = (-1)^{\deg x} \cP(x) \otimes \nu\Delta^{3j} \neq 0.$$
The sign comes from the Koszul sign rules, since $\nu$ has odd degree.

These $\d_4$ and $\d_5$ differentials are closely related. Indeed, the class $\alpha = \Box_\bZ(A)$ represents the class $\nu$ in the following sense. There is a \define{elliptic spectral sequence} converging to $\tmf^\bullet(\pt)_{(3)}$ whose $E_2$ page is $\H^\bullet(\overline{\cM}_{ell}; L^{\otimes \bullet})_{(3)}$. In this spectral sequence, $\alpha$ is a permanent cocycle, and its image on the $E_\infty$ page is the associated graded element to $\nu$.

Returning to the $\d_5$ differential, we must work out $\d_5(x\otimes f)$ whenever $f \in \tmf^\bullet(\pt)$ satisfies $\nu f = 0$. The discussion above about lifts of multiplication by the Hasse invariant implies:
\begin{gather*}
  \d_5(x \otimes c_4) = \Box_\bZ\cP(x) \otimes c_6, \\
  \d_5(x \otimes c_6) = \Box_\bZ\cP(x) \otimes c_4^2, \\
  \d_5(x \otimes \{3\Delta\})  = 0.
\end{gather*}
These almost completely determine the behaviour of $\d_5$, since it must be linear for the action by $\tmf^\bullet(\pt)$, and so, if $f_1,f_2 \in \tmf^\bullet(\pt)$ with $\nu f_1 = 0$, then
$$ \d_5(x \otimes f_1f_2) = \d_5(x\otimes f_1) f_2.$$
Note that this is consistent with the above rules because $\Box_\bZ\cP(x)$ is $3$-torsion. Indeed, for $f = c_4c_6$, we would a priori face a discrepancy like
$$ \d_5(x \otimes c_4)c_6 - \d_5(x \otimes c_6)c_4 = \Box_\bZ\cP(x) \otimes (c_6^2 - c_4^3) = \Box_\bZ\cP(x) \otimes 576 \{3\Delta\},$$
but this vanishes since $576$ is divisible by $3$.

These rules do not quite determine the behaviour of $\d_5$ on the whole $E_5$ page: there is the possibility of a differential of the form
$$ \d_5(x\otimes \mu^2) = \Box_\bZ\cP(x) \otimes f$$
for some $f \in \tmf^{-24}(\pt) = \operatorname{Span}(\{3\Delta\}, c_4^3)$. Because $\Box_\bZ\cP(x)$ is always $3$-torsion, this map only depends on the class of $f$ modulo $3$. Furthermore, $f$ must be in the kernel of $\cH : \tmf^{-24}(\pt)/3 \to \ku^{-24}\llbracket q \rrbracket/3 = \bF_3\llbracket q\rrbracket$. So the only possibility is, up to sign,
$$ \d_5 : x \otimes \mu^2 \overset?\mapsto \Box_\bZ\cP(x) \otimes \{3\Delta\}.$$
One hint that there is in fact such a differential comes from the elliptic spectral sequence. The class $\mu^2 \in \tmf^\bullet(\pt)$ is represented on the $E_2$ page by a class $\beta^2 \in \H^4(\overline{\cM}_{ell}; L^{\otimes 12})_{(3)}$. Although $\nu\mu^2 = 0$ in $\tmf^\bullet(\pt)$, $\alpha\beta^2 \neq 0$ in $\H^5(\overline{\cM}_{ell}; L^{\otimes 14})_{(3)}$; it is instead the image of a $\d_5$-differential emitted by $\Delta$. Rather than exploring the elliptic spectral sequence in more detail, we will give an alternate proof:

\begin{proposition}
  The AHSS for $\tmf^\bullet$ includes a nontrivial $\d_5$ differential supported by classes of the form $x \otimes \mu^2$.
\end{proposition}
\begin{proof}
Our strategy is to compare our AHSS with the computations from~\cite{MR2341960}, which computes the $3$-local $\tmf_\bullet$-homology of the symmetric group $S_3$. More specifically, that paper computes the $3$-local homology $\tmf_\bullet(\Sigma BS_3)_{(3)}$, where $BS_3$ is the classifying space of $S_3$, and $\Sigma BS_3$ is its suspension. We will focus on the specific value
$$ \tmf_{29}(\Sigma BS_3) = 0.$$

Since~\cite{MR2341960} computes homology, not cohomology, in this proof only we will work with homological, rather than cohomological, gradings. The homology and cohomology of a point are related by
$$ \tmf_\bullet(\pt) = \tmf^{-\bullet}(\pt).$$
But note that the same formula does not hold with $\pt$ replaced by other spaces. 

The AHSS for homology reads:
$$ \H_m(\Sigma B S_3; \tmf_n(\pt)) \Rightarrow \tmf_{m+n}(\Sigma BS_3).$$
Homology AHSSs converges strongly by Theorem 12.2 of~\cite{MR1718076}.
The differentials in homology AHSSs are essentially the same as the differentials in cohomology AHSSs, because in both cases they come form the Postnikov tower of the coefficient spectrum. The only difference is to understand the cohomology operations instead as homology operations. This is easy: the groups $\H_\bullet(X;\bF_3)$ and $\H^\bullet(X;\bF_3)$ are dual, and so one uses the dual map.

We have the following $\bZ_{(3)}$- and $\bF_3$-homology of $\Sigma BS_3$:
\begin{gather*}
  \H_\bullet(\Sigma BS_3; \bZ_{(3)}) = \begin{cases} \bZ, & \bullet = 0, \\ \bF_3 t_k, & \bullet = 4k, k>0, \\ 0, & \text{else},\end{cases} 
  \\
  \H_\bullet(\Sigma BS_3; \bF_3) = \begin{cases} %\bF_3, & \bullet = 0, \\ 
  \bF_3 t_k, & \bullet = 4k, \\ \bF_3 T_k, & \bullet= 4k+1, k>0, \\ 0, & \text{else},\end{cases} 
\end{gather*}
These are easily seen by noting that $S_3 = C_3 \rtimes C_2$. We have named basis vectors $t_k,T_k$, with $t_k$ denoting both the integral class and its mod-$3$ reduction. The Bockstein is $\Box(T_k) = t_k$, and the first Steenrod power is
$$ \cP(t_k) = (k-1)t_{k-1}, \qquad \cP(T_k) = k T_{k-1}.$$

In total degree $29$, the only nonzero entries on the $E_2$ page are:
$$ t_4 \otimes \nu\mu \in \H_{16}(\Sigma BS_3; \tmf_{13}(\pt)), \qquad T_2 \otimes \mu^2 \in \H_9(\Sigma BS_3; \tmf_{13}(\pt)).$$
The former is the image of a $\d_4$ differential:
$$ \d_4(t_5 \otimes \mu) = \cP(t_5) \otimes \nu\mu = t_4 \otimes \nu\mu.$$

The latter class must also be killed by some differential in order to confirm the computation $\tmf_{29}(\Sigma BS_3) = 0$ from~\cite{MR2341960}. It is not the image of a differential. Indeed, the only classes of total degree $30$ on the $E_2$ page that could emit differentials to $T_2 \otimes \mu^2$ are $t_5 \otimes \mu$, which we already saw does not survive $\d_4$, and 
$$ T_5 \otimes \nu\mu = \d_4(T_6 \otimes \mu),$$
and so also does not survive $\d_4$. Thus $T_2 \otimes \mu^2$ must emit a differential. But the only degree possible is $\d_5$, and so we conclude
$$ \d_5(T_2 \otimes \mu^2) \neq 0.$$
And so general AHSSs for $\tmf$ include a $\d_5$ differential supported by classes of the form $x\otimes \mu^2$.
\end{proof}

After $\d_4,\d_5$, the next differential allowed by general considerations about degrees of stable operations is $\d_8$. To derive its formula, compare with the analysis in~\cite{MR2307274}: $\d_8$ arises as the ``reason'' that $\d_4 \circ \d_4 = 0$. What is this reason? The vanishing of
$$ \d_4^{\circ 2} = \cP^{\circ 2} \otimes \nu^2 $$
has nothing to do with the ``$X$'' part of the differential, because $\cP^{\circ 2} \neq 0$. Rather, it vanishes because $\nu^2 = 0$ in cohomology. 
This means that the $\d_8$ differential will include a Massey product. Indeed, suppose that $x \otimes f \in \ker(\d_4)$ simply because $\nu f = 0$, while perhaps $\cP(x) \neq 0$. Then, if we imagine working at cochain level, we instead have 
$$ \d_4(x \otimes f) = (-1)^{\deg x} \cP(x) \otimes \d_1(F),$$
where $\d_1$ is the differential computing $\tmf^\bullet(\pt)$ and $F$ is some cochain for which $\d_1(F) = \nu f$. Imagine a ``total differential'' $\d_1 + \d_4$. One can find cochain formulas so that $\d_1$ and $\d_4$ commute, but $(\d_1 + \d_4)^2$ will not vanish. Rather, $(\d_1 + \d_4)^2 = \d_4^2$, which at cochain level is
$$ \d_4^2(x \otimes f) = \cP^{\circ 2}(x) \otimes \nu \d_1(F) = -\cP^{\circ 2}(x) \otimes \d_1(\nu F).$$
To correct this, we include a $\d_8$ differential whose commutator with $\d_1$ is $x\otimes f \mapsto -\cP^{\circ 2}(x) \otimes \d_1(\nu F)$. I.e.\ we should have
$$ \d_8(x \otimes f) \overset?= -\cP^{\circ 2}(x) \otimes (\nu F + N f),$$
where $N$ is some cochain such that $\d_1(N) = \nu^2$. This is an okay thing to write because $\d_1(\nu F + Nf) = -\nu \d_1(F) + \d_1(N)f = -\nu^2 f + \nu^2 f = 0$ at cochain level.

The combination $\nu F + N f$ is, by definition, the \define{Massey product} $\langle \nu,\nu,f\rangle$. We find:
$$ \d_8 = -\cP^{\circ 2} \otimes \langle\nu,\nu,-\rangle.$$
We emphasize that this is only well-defined on the $\d_4$-cohomology. Indeed, $\langle \nu,\nu,f\rangle$ doesn't exist unless $\nu f = 0$. Furthermore, the class $F$ is defined only modulo cocycles, but if you change $F$ by a cocycle, then you change $\langle \nu,\nu,f\rangle$ by a multiple of $\nu$, and so do not change $\d_8$ on the $\d_4$-cohomology. There is no essential ambiguity in the choice of $N$ because $\tmf^{-7}(\pt) = 0$. In the AHSS for some generic $3$-local $E_\infty$ ring spectrum $\cE$, there is room for one further term in the $\d_8$ differential, equal to $-\cP^{\circ 2} \otimes \lambda$ for some $\lambda \in \cE^{-7}(\pt)$. But again we use that $\tmf^{-7}(\pt) = 0$ to rule out that possibility here.

Although we will not need it, we mention that an analysis as in the $p\geq 5$ case identifies a $\d_9$-differential of the form $\cP^{(2)} \otimes (\text{integral lift of }A^2)$. Note that, although the Hasse invariant $A$ itself does not have an integral lift, $A^2$ lifts to $c_4$. Finally, if we twist by an 't~Hooft anomaly $\omega \in \H^4(X;\bZ_{(3)})$, the only difference is that the Steenrod operator $\cP$ is replaced by the operator $\cD = \cP - \omega$. All together, we find:
\begin{proposition}\label{prop.AHSSdifs}
  The AHSS $\H^\bullet(X; \tmf^\bullet(\pt)) \Rightarrow \tmf^\bullet_\omega(X)_{(3)}$ includes the following differentials, with $\cD = \cP - \omega$:
  \begin{itemize}
    \item There is a $\d_4$ differential of the form
    $$ \d_4 = \cD \otimes \nu.$$
    \item The multiples of classes $c_4$, $c_6$, and $\mu^2$ support a $\d_5$ differential of the form
    $$ \d_5 = \Box_\bZ\cD \otimes \begin{cases} fc_4 \mapsto fc_6, \\ fc_6 \mapsto fc_4^2, \\ f\mu^2 \mapsto f\{3\Delta\}. \end{cases}$$
    \item There is a $\d_8$ differential of the form
    $$ \d_8 = -\cD^{\circ 2} \otimes \langle\nu,\nu,-\rangle.$$
    \item There is a $\d_9$ differential of the form
    $$ \d_9 = -\Box_\bZ \cD^{\circ 2} \otimes c_4.$$
  \end{itemize}
\end{proposition}
There are also higher differentials which we will not work out.

\subsection{Running the spectral sequence}\label{subsec.runningAHSS}

We are now ready to understand the AHSS $\H^\bullet(B\rM_{24}; \tmf^\bullet(\pt)[\frac12]) \Rightarrow \tmf^\bullet_\omega(B\rM_{24})[\frac12]$, where $\omega \in \H^4(\rM_{24};\bZ[\frac12]) = \bF_3 r$, with $r$ as in Section~\ref{sec.HM24}. Given Conjecture~\ref{conjecture.main}, we are interested in the value of $\tmf^\bullet_\omega(B\rM_{24})[\frac12]$ for $\bullet \equiv 1 \pmod 4$, and specifically $\bullet = -27$.

As observed in~\cite{MR1263724}, at the primes $p\geq 5$ the $E_2$ page is $\H^\bullet(B\rM_{24};\bZ_{(p)})[c_4,c_6]$, where the ring $\H^\bullet(B\rM_{24};\bZ_{(p)})$ vanishes if $p \not\in\{5,7,11,13\}$, and for $p\in\{5,7,11,13\}$ it 
is a polynomial ring in a generator $x_p$ of cohomological degree $2(p-1)$ and additive order $p$. Thus all differentials vanish for degree reasons, and the spectral sequence stabilizes on the $E_2$ page. The differentials do play a role, however: they lead to extensions. Indeed, let $\tilde x_p = x_p c_{p-1}$, which is of total degree $0$ on the $E_\infty$ page. Then, by the $q$-expansion map $\cH$ to K-theory, we see that the translates of $\tilde x_p \bF_p[\tilde x_p]$ on the $E_\infty$-page compile to  copies of the $p$-adic integers $\bZ_{(p)}$ in $\tmf^\bullet(B\rM_{24})_{(p)}$.
For the purposes of this paper, all that we care about is that, for $p\geq 5$, $\tmf^\bullet(B\rM_{24})_{(p)}$ is supported in degrees $\bullet \equiv 0 \pmod 4$. %, and so vanishes in degree $\bullet \equiv 1 \pmod 4$.

We put aside the prime $p=2$ for being too complicated, leaving only the prime $p=3$. The first few differentials for the AHSS $\H^\bullet(B\rM_{24}; \tmf^\bullet(\pt)) \Rightarrow \tmf^\bullet_\omega(B\rM_{24})_{(3)}$ are summarized in Proposition~\ref{prop.AHSSdifs}. A typical term on the $E_2$ page has shape $x\otimes f$, where $f\in\tmf^\bullet(\pt)$ and $x \in \H^\bullet(\rM_{24};\bZ_{(3)})$ if $f$ is nontorsion and $x \in  \H^\bullet(\rM_{24};\bF_3)$ if $f$ is torsion. 
With some caveats, the differentials sort into two sets. If $f$ is nontorsion (and not a power of $\Delta^3$), then $x \otimes f$ only supports $\d_5$ and $\d_9$ differentials. If $f$ is torsion (and not a translate of $\mu^2$) then $x\otimes f$ only supports $\d_4$ and $\d_8$ differentials.

By Theorem~\ref{thm.Green}, $\H^\bullet(\rM_{24};\bZ_{(3)})$ is supported only in degrees $\bullet \equiv 3,4 \pmod 4$. Therefore the $E_2$-page entries $x\otimes f$ with $f$ nontorsion are also only in degrees $\bullet \equiv 3,4\pmod 4$. We remark that the differentials here are nontrivial, because Proposition~\ref{prop.Paction} implies:
$$ \Box_\bZ\cP(u) = \Box_\bZ(ur + T \pm Rs) = t \pm rs \neq 0.$$
But, since we care mostly about the case $\bullet \equiv 1 \pmod 4$, these differentials don't affect us. 

Therefore we may focus on classes of the form $x\otimes f$ with 
$$f \in \{1, \nu, \mu, \nu\mu,\mu^2, \{\nu\Delta\}, \mu^3, \{\nu\Delta\}\mu, \mu^4\}$$
or the translates thereof by powers of $\Delta^3$. Except for $\mu^2$, these classes only support $\d_4$ and $\d_8$ differentials. Since $\d_5(x\otimes \mu^2)$ is an integral class times $\{3\Delta\}$, it has degree $3$ or $4\pmod 4$, and so does not interact with the $\bullet \equiv 1 \pmod 4$ case. Thus for the purposes of computing $\tmf^\bullet_\omega(B\rM_{24})[\frac12]$ for $\bullet \equiv 1 \pmod 4$, we may work just with the subring of the $E_2$ page of the form
$$ \H^\bullet(\rM_{24};\bF_3) \otimes \bF_3\{1, \nu, \mu, \nu\mu,\mu^2, \{\nu\Delta\}, \mu^3, \{\nu\Delta\}\mu, \mu^4\},$$
and just with the differentials
$$ \d_4 = \cD \otimes \nu, \qquad \d_8 = -\cD^{\circ 2} \otimes \langle \nu,\nu,-\rangle,$$
where $\cD = \cP - \omega$.

Recall from Proposition~\ref{prop.tmftorsion} the action of the maps $\nu$ and $\langle \nu,\nu,-\rangle$. Writing $H = \H^\bullet(\rM_{24};\bF_3)$, we therefore find ourselves interested in the following total complex:
$$ H1 \overset\cD\longto H\nu \overset{\cD^2}\longto H\mu \overset\cD\longto H\nu\mu \overset{\cD^2}\longto H\mu^2 \overset{\cD^2}\longto H\{\nu\Delta\} \overset\cD\longto H\mu^3 \overset{\cD^2}\longto H\{\nu\Delta\}\mu \overset\cD\longto H\mu^4$$
This is an ordinary cochain complex because $\cD$ is a $3$-differential by Lemma~\ref{lemma.Dis3dif}. Its cohomology is closely related to the cohomology of $\cD$ itself, which is listed in Proposition~\ref{prop.cohD}. Indeed, both $\ker(\cD) / \im(\cD^2)$ and $\ker(\cD^2) / \im(\cD)$ vanish whenever $\cD$ is exact. More generally, both $\ker(\cD) / \im(\cD^2)$ and $\ker(\cD^2) / \im(\cD)$ are of the same total dimension as the cohomology of $\cD$ (with both $\{*\}$ and  cohomology $\{*\to*\}$ thought of as $1$-dimensional; this is the total dimension of the super-vector-space-valued cohomology from Appendix~\ref{appendix-3complex}). The precise cohomological degrees of cohomology classes depends on whether we use $\ker(\cD) / \im(\cD^2)$ or $\ker(\cD^2) / \im(\cD)$, but their degrees mod $4$ do not depend, since $\cD$ preserves the cohomological degree mod $4$.

The cohomology of the above total complex at the entries $H1$, $H\mu^2$, and $H\mu^4$ is more complicated, but for our purposes irrelevant. Indeed, $1$, $\mu^2$, and $\mu^4$ have cohomological degrees $\bullet \equiv 0 \pmod 4$, and, by Theorem~\ref{thm.Green}, $H$ vanishes in degrees $\bullet \equiv 1 \pmod 4$. Thus the cohomologies at those entries also vanish in degrees $\bullet \equiv 1 \pmod 4$.
Combining all of our calculations, we find the following:

\begin{theorem}\label{thm.tmf-3}
  Let $\omega = -\epsilon r$ with $\epsilon \in \bF_3$ and notation for $\H^\bullet(\rM_{24};\bZ_{(3)})$ as in Theorem~\ref{thm.Green}. On the $E_9$ page of the AHSS $\H^\bullet(B\rM_{24}; \tmf^\bullet(\pt)[\frac12]) \Rightarrow \tmf^\bullet_\omega(B\rM_{24})[\frac12]$, 
  the cohomology in total degree $\bullet \equiv 1 \pmod 4$ is the following. Write $\Upsilon = U\otimes \mu$, and note that $\Upsilon^2 = 0$. Then we have:
  \begin{itemize}
    \item $\epsilon = 0:$ A free $\bF_3[s^3,\Delta^3, \Upsilon]/(\Upsilon^2)$-module generated in cohomological degrees (repeated entries indicate multiplicity in the generating set):  
     $$25, 13, 9, 5, -3, -11, -11, -27,$$
%    $$-3, -27, 13, -11, 9, -11, 25, 5,$$
    plus a free $\bF_3[s^3,\Delta^3]$-module generated in degrees
     $$ 29, 17, 17, 5, 5, -3, -3, -19,$$
%    $$ -19, 5, 17, -3, 17, -3, 29, 5, $$
    plus a free $\bF_3[\Delta^3]$-module generated in degrees
    $$ -3, -27.$$
    \item $\epsilon = 1:$ 
    A free $\bF_3[s^3,\Delta^3, \Upsilon]/(\Upsilon^2)$-module generated in cohomological degrees  
     $$13, -3, -11, -23,$$
%    $$13, -11, -3, -23,$$
    plus a free $\bF_3[s^3,\Delta^3]$-module generated in degrees
     $$ 29, 29, 17, 5, 5, 5, -3, -15.$$
%    $$ 5, -15, 29, 5, 17, -3, 29, 5.$$
    \item $\epsilon = -1:$ 
    A free $\bF_3[s^3,\Delta^3, \Upsilon]/(\Upsilon^2)$-module generated in cohomological degrees  
    $$25, 9, 1, -11,$$
    plus a free $\bF_3[s^3,\Delta^3]$-module generated in degrees
     $$ 29, 29, 17, 17, 9, 5, -3, -7.$$
%    $$ 17, -7, 29, 9, 17, -3, 29, 5.$$
  \end{itemize}
\end{theorem}
\begin{proof}
Each entry in Proposition~\ref{prop.cohD} produces two free $\bF_3[s^3,\Delta^3]$-modules the $E_9$ page. 

For example, the $\epsilon=0$ entry in Proposition~\ref{prop.cohD} listed as ``$\{28 \to 32\}$'' is represented by $[s^2b\mapsto s^2b^2]$. (It is also represented by $[st \mapsto t^2]$.) The $E_2$ page includes the degree $1\pmod 4$ classes $s^2b \otimes \nu$, $s^2 b^2 \otimes \nu$, $s^2b \otimes \{\nu\Delta\}$, and $s^2b^2 \otimes \{\nu\Delta\}$. We have
$$ \d_4 : s^2b \otimes 1 \mapsto s^2b^2 \otimes \nu$$
and so $s^2b^2 \otimes \nu$ does not contribute to cohomology on the $E_9$ page, but the presence of $\cD$-cohomology means that $s^2b\otimes \nu$ is not in the image of $\d_4$ nor in the kernel of $\d_8$, and so does contribute cohomology, as do its translates by $s^3$ and $\Delta^3$. If we look instead at $s^2b \otimes \{\nu\Delta\}$ and $s^2b^2 \otimes \{\nu\Delta\}$, we see that the former supports a $\d_4$-differential but the latter survives to~$E_9$.

Consider now the $\epsilon=0$ entry in Proposition~\ref{prop.cohD} listed as ``$\{2 \to 6\}$.'' It is represented by $[Us^2b \overset\cP\mapsto Us^2b^2]$, which can be moved to have total degree $\bullet\equiv 1\pmod 4$ by multiplying by $\nu\mu$ or $\{\nu\Delta\}\mu$. In this way we see $E_9$-page classes represented by $Us^2b \otimes \nu\mu$ and $Us^2b^2 \otimes \{\nu\Delta\}\mu$. Indeed, while proving Proposition~\ref{prop.cohD}, we noted that the $\cD$-cohomology in degree $n\equiv 0 \pmod 4$ was isomorphic to the $\cD$-cohomology in degree $n+10$, with the isomorphism given by multiplication by $U$, of degree $\bullet=10$, whereas the odd-degree torsion classes listed in Proposition~\ref{prop.tmftorsion} are related by multiplication by $\mu^2$, of degree $\bullet=-10$. Thus we find that each entry in Proposition~\ref{prop.cohD} of degree $0\pmod 4$ contributes not just a pair of copies of $\bF_3[s^3,\Delta^3]$, but a pair of copies of $\bF_3[s^3,\Delta^3,\Upsilon]$. The general rule is that an entry like ``$\{n\}$'' with $n\equiv 0 \pmod 4$ produces generators in degrees $n-3$ and $n-27$, whereas an entry like ``$\{n\to n+4\}$'' produces generators in degrees $n-3$ and $n+4-27$.

Since $\H^\bullet(\rM_{24};\bF_3)$ vanishes in degree $1\pmod 4$, the only other entries in Proposition~\ref{prop.cohD} are of degree $3\pmod 4$. These can be combined with $\mu$ or $\mu^3$ and a similar analysis can be performed, but now multiplication by $U$ (and hence by $\Upsilon$) is zero. Thus we find merely a copy of $\bF_3[s^3,\Delta^3]$. The entry ``$\{n\}$'' with $n \equiv 3 \pmod 4$ in Proposition~\ref{prop.cohD} produces generators in degrees $n-10$ and $n-30$, whereas an entry ``$\{n \to n+4\}$'' produces generators in degrees $n+4-10$ and $n-30$.
\end{proof}

Reading off the most interesting degrees, we have:
\begin{corollary}\label{cor.main}
  In cohomological degree $\bullet = -27$, 
    both twisted cohomology groups $\tmf^{-27}_{\pm r}(B\rM_{24})_{(3)}$ vanish.

  Let $\Phi = s^6 \otimes \Delta^6$.  
  The $E_9$-page approximation to the untwisted cohomology  $\tmf^{-27}(B\rM_{24})_{(3)}$ is
  $$ \tmf^{-27}(B\rM_{24})_{(3)} \approx \bF_3[\Phi]\{1 \otimes \{\nu\Delta\}, U \otimes \{\nu\Delta\}\mu\} \oplus \bF_3 \{r \otimes \mu^3\}.$$
  By this we mean the abelian group isomorphic to $\bF_3[\Phi]^{ 2} \oplus \bF_3$, where the $\bF_3[\Phi]^{ 2}$ summand is generated over $\bF_3[\Phi]$ by the elements on the $E_9$ page represented by  $1 \otimes \{\nu\Delta\}$ and $U \otimes \{\nu\Delta\}\mu$, and where the $\bF_3$ summand is generated by $r\otimes \mu^3$.

  For comparison, in cohomological degree $\bullet = -3$, the $E_9$-page approximations to $\tmf^{-3}_{\epsilon r}(B\rM_{24})_{(3)}$ are:
  \begin{align*}
   \tmf^{-3}(B\rM_{24})_{(3)} & \approx \bF_3[\Phi]\{1 \otimes \nu, U \otimes \nu \mu, s^2R \otimes \mu^3, sT \otimes \mu^3\} \oplus \bF_3 \{rR \otimes \mu^2\}, \\[3pt]
   \tmf^{-3}_{-r}(B\rM_{24})_{(3)} & \approx \bF_3[\Phi]\{1 \otimes \nu, U \otimes \nu \mu, sT \otimes \mu^3\}, \\[3pt]
   \tmf^{-3}_{+r}(B\rM_{24})_{(3)} & \approx \bF_3[\Phi]\{sT \otimes \mu^3\}.
  \end{align*}
  Here we have listed just $E_2$-page representatives of the $E_9$-page generators, and those are of course ambiguous. For example, $sT$ and $St$ are $\cD$-cohomologous.
  
  In cohomological degree $\bullet = 1$, $\tmf^{-3}_{-\epsilon r}(B\rM_{24})_{(3)}$ vanishes for $\epsilon = 0,1$, and
  $$ \tmf^1_{+r}(B\rM_{24})_{(3)} \approx \bF_3[\Phi]\{ s^2 b \otimes \{\nu\Delta\}\}.$$
  Again note the ambiguity that $s^2b$ and $st$ are $\cD$-cohomologous in this case. \qed
\end{corollary}

To complete the proof of Theorem~\ref{thm.main}, we have:
\begin{corollary}\label{cor.main2}
  For all values $\omega = \epsilon r \in \H^4(\rM_{24};\bZ_{(3)})$, the AHSS for $\tmf^\bullet_\omega(B\rM_{24})_{(3)}$ converges. When $\omega = 0$ or $-r$, the class $1 \otimes \nu$ on the $E_2$ page is a permanent cycle, and represents a class in $\tmf^{-3}_\omega(B\rM_{24})_{(3)}$ with nontrivial restriction to $\tmf^{-3}(B\rM_{24})$. When $\omega = +r$, the restriction map $\tmf^{-3}_\omega(B\rM_{24})_{(3)} \to \tmf^{-3}_\omega(B\rM_{24})_{(3)}$ vanishes.
\end{corollary}
\begin{proof}
  The first two statements follow from the claim that, other than $\d_4$, all differentials in the AHSS $\H^m(B\rM_{24};\tmf^n(\pt)_{(3)}) \Rightarrow \tmf^{m+n}_\omega(B\rM_{24})_{(3)}$ vanish when $m=0$. The last statement is already clear from Corollary~\ref{cor.main} together with the fact that the restriction map along $\pt \to B\rM_{24}$ is restriction to the $m=0$ column, and annihilates $r,s,t,u$ and hence $sT \otimes \mu^3$.
  
  That the claim implies convergence is explained in \S\ref{ahss.general}.
   And if all such differentials vanish, then any class with $m=0$ that survives $\d_4$ is permanent, and so must represent a class in $\tmf^\bullet_\omega(B\rM_{24})_{(3)}$. But when $\omega = 0,-r$, the class $1 \otimes \nu$ itself survives $\d_4$, and has nontrivial restriction along $\pt \to B\rM_{24}$. As explained already in \S\ref{ahss.general}, all differentials vanish on the $m=0$ column for the untwisted cohomology $\omega = 0$.
   
   The only way a $\d_k$-differential can be nontrivial on the $m=0$ column is if it contains a term of that simply multiplies by an element on the $E_k$-page of total cohomological degree $1$. For $\omega = \epsilon r$ 
   with $\epsilon = 0,1$,  Corollary~\ref{cor.main} implies that there are no such elements. For $\epsilon = -1$, there are elements on the $E_9$ page of cohomological degree $1$, and so we need to know that they cannot appear.
   
   However, the universality of the AHSS means that the multiplying element must be of the form $x \otimes f$ where $f$ is arbitrary but where $x$ is produced from $\omega$ by a cohomology operation. The algebra of $3$-local cohomology operations is generated by the Steenrod powers and the Bockstein. The Bockstein vanishes on $r$ and the first Steenrod power acts as $\cP(r) = -r^2$. The second Steenrod power is simply $r \mapsto r^3$ for degree reasons, and the higher powers annihilate~$r$. These, together with the Cartan relation (which says that each Steenrod operator is a derivation modulo lower Steenrod operators), imply that the only cohomology classes that can be produced by $r$ are in the polynomial ring $\bF_3[r]$.
   
  The degree-$1$ classes $\Phi^k s^2b$ are not cohomologous to anything in this ring. Indeed, as shown in Proposition~\ref{prop.cohD}, for $\omega = -\epsilon r$ with $\epsilon = -1$ the action of $\cD = \cP+\epsilon r$ is exact on $\bF_3[r]$. Thus, by repeating the proof of Theorem~\ref{thm.tmf-3} just on this subring, we see that nothing of total degree $1$ in $\bF_3[r] \otimes \tmf^\bullet(\pt)$ survives to the $E_9$-page, and so there are no elements that could appear as multipliers in higher differentials, and so all higher differentials vanish on the $m=0$ column.
  \end{proof}

\appendix

\section{Higher complexes}\label{appendix-3complex}

Our analysis of the AHSS for $\tmf^\bullet_{(3)}$ in \S\ref{subsec.ahssp3} and \S\ref{subsec.runningAHSS} relied on the fact (Lemma~\ref{lemma.Dis3dif}) that the operator $\cD = \cP + \epsilon r$ is a \define{$3$-differential} in the sense that $\cD^{\circ 3} = 0$.
The goal of this Appendix is to tell the general story of higher differentials, and to point out an intriguing connection to Verlinde rings that the author has not seen stated directly in the literature.

%\subsection{Idea of cohomology of an $\ell$-differential}

Let $\bK$ be a field, perhaps of positive characteristic, and choose a positive integer $\ell$. An \define{$\ell$-differential} on a $\bK$-vector space $V$ is a linear endomorphism $D$ such that $D^\ell = 0$. For example, a $1$-differential is the zero map, and a $2$-differential is an ordinary differential. 
In the ordinary case, the cohomology of a $2$-differential is $\H^*(V,D) = \ker(D)/\im(D)$. There is also a theory of cohomology of $\ell$-differentials for higher $\ell$, which dates as early as~\cite{MR6514}; see the introduction of~\cite{MR3742439} for some history and a number of relevant references. Various authors have tried to define the cohomology of an $\ell$-differential as, for example, $\H^*(V,D) = \ker(D) / \im(D^{\ell-1})$ or $\ker(D^{\ell-1})/\im(D)$, and these definitions are fine for basic purposes. But there is a somewhat richer story, that may be especially entertaining for quantum field theorists.

The idea is the following. A vector space with an $\ell$-differential is equivalently a module for the algebra $\bK[D]/(D^\ell)$. This algebra has $\ell$ indecomposable modules, indexed by their $\bK$-dimensions $1,\dots,\ell$: the one-dimensional module is simple, and the $\ell$-dimensional module is free. Every $\bK[D]/(D^\ell)$-module splits as a direct sum of indecomposable modules. We will say that $D$ is \define{exact} when the module is free, i.e.\ all of its indecomposable summands are $\ell$-dimensional. The {cohomology} of an $\ell$-differential should measure its failure to be exact. Both $\ker(D) / \im(D^{\ell-1})$ or $\ker(D^{\ell-1})/\im(D)$ measure this failure coarsely: those two vector spaces are (noncanonically) isomorphic, and their dimension counts the number of non-free indecomposable summands. But we can measure things more finely, by recording which non-free summands appear.
We will define the {cohomology} $\rH^*(V,D)$ of an $\ell$-differential $D$ on $V$ to be the result of:
\begin{itemize}
  \item Decomposing $(V,D)$ as a direct sum of indecomposable $\bK[D]/(D^\ell)$-modules.
  \item Discarding the free summands.
  \item Converting the other summands into simple objects of a semisimple category.
\end{itemize}
For example, in the ordinary case, there is one non-free indecomposable $\bK[D]/(D^2)$-module, namely the one-dimensional one. Thus the cohomology in our sense is an object of a semisimple category with one simple object, i.e.\ the category of vector spaces.

This procedure is functorial, although it doesn't look so from our description. To make it cleaner, we will use the technology of \define{semisimplification} developed in~\cite{MR1686423,EOsimpl}. Although we care most about the case $\bK = \bF_3$ and $\ell = 3$, we will first tell the story when $\ell$ is not divisible by the characteristic of $\bK$.

\subsection{$\ell$-complexes in characteristic not dividing $\ell$}

The category of $\bK[D]/(D^\ell)$-modules is not naturally monoidal (if $\ell$ is not a power of the characteristic of $\bK$). This is clear already when $\ell=2$: the tensor product of ordinary  complexes is $D(v\otimes w) = D(v) \otimes w + (-1)^{\deg v} v \otimes D(w)$, which requires at least a $\bZ_2$-grading. To correct this, let us say that a (periodic) \define{$\ell$-complex} is a $\bZ_\ell$-graded vector space with an $\ell$-differential that increases the grading by $1$.
By field-extension if necessary, suppose that $\bK$ contains a primitive $\ell$'th root of unity $q$. Then we may define the tensor product of two $\ell$-complexes to be their usual tensor product as $\bZ_\ell$-graded vector spaces, equipped with the differential
$$ D(v\otimes w) = D(v) \otimes w + q^{\deg v} v \otimes D(w).$$
We will write $\cC_q$ for this monoidal category. The choice of $q$ identifies it with the category of modules for the semidirect product Hopf algebra 
\begin{gather*}
 H_q = \bK[D]/(D^\ell) \rtimes \bZ_\ell = \bK\langle D, K\rangle / (D^\ell, K^\ell-1, K D - q D K),\\
 \Delta(K) = K \otimes K, \qquad \Delta(D) = D \otimes 1 + K \otimes D,
\end{gather*}
which is nothing but the upper Borel inside Lusztig's small quantum group for $\mathrm{SL}(2)$.

The monoidal category $\cC_q$ is not braided, but it is spherical, and so has a \define{semisimplification} $\overline{\cC_q}$. 
The defining property of $\overline{\cC_q}$ is that it is the universal semisimple monoidal category receiving a monoidal (but neither left- nor right-exact) functor $\cC_q \to \overline{\cC_q}$. To construct it, one follows~\cite{MR1686423} and defines a monoidal ideal $\cN \subset \cC_q$ of ``negligible morphisms,'' which are those morphisms in the kernel of the trace pairing $\hom(X,Y) \otimes \hom(Y,X) \to \bK$ (which exists in any spherical monoidal category). Then $\overline{\cC_q}$ is defined to be the quotient category $\cC_q / \cN$. Although not obvious, this quotient category is semisimple, and the simple objects are indexed by the indecomposable objects in $\cC_q$ of nonzero quantum dimension.

It is not hard to show that an indecomposable $H_q$-module of $\bK$-dimension $n$ has quantum dimension $[n]_q = (q^n - 1)/(q-1)$. In particular, the objects in the kernel of $\cC_q \to \overline{\cC_q}$ are precisely the modules which are free over $\bK[D]/(D^\ell)$. 
We are therefore justified in using the name ``$\rH^*$'' for the functor $\cC_q \to \overline{\cC_q}$, and calling $\rH^*(V,D)$ the \define{cohomology} of the $\ell$-complex~$(V,D) \in \cC_q$.

When $\bK$ is of characteristic $0$, Theorem~5.2 of~\cite{EOsimpl} identifies the fusion rules for $\overline{\cC_q}$: the fusion ring is isomorphic to the fusion ring of
$$ \cat{Vec}[\bZ_\ell] \boxtimes \cat{Ver}_q \boxtimes \cat{Rep}(\mathrm{PGL}(2)).$$
Here $\cat{Ver}_q$ is the Verlinde category of $\mathrm{SL}(2)$ at level $k = \ell-2$. 
Corollary~5.3 of~\cite{EOsimpl} shows that $\overline{\cC_q}$ does indeed contain $\cat{Ver}_q$ as a subcategory.
The fusion ring for $\cat{Rep}(\mathrm{PGL}(2))$ is the same as that of $\cat{Rep}(\mathrm{OSp}(1|2))$, which also appears as a subcategory of $\cC_q$. Finally, $\cat{Vec}[\bZ_\ell] \subset \overline{\cC_q}$ is spanned by the images of $1$-dimensional $H_q$-modules. It is therefore conjectured in that paper that there is an equivalence of spherical fusion categories
$$ \overline{\cC_q} \overset?\cong \cat{Vec}[\bZ_\ell] \boxtimes \cat{Ver}_q \boxtimes \cat{Rep}(\mathrm{OSp}(1|2)).$$
To check this conjecture requires checking that there are no interesting associators between objects coming from the various tensorands on the right-hand side.

There are two extremal cases of this story. When $\ell=1$, the Hopf algebra $H_q$ is trivial, and $\cC_q = \overline{\cC_q} = \cat{Vec}$. More interesting is the case $\ell = +\infty$. Then by ``primitive $\ell$'th root of unity'' we will mean that $q$ does not solve an algebraic equation with nonnegative  integer coefficients. The Hopf algebra $H_q$ is then simply $\bK\langle D, K^{\pm 1}\rangle / (K D - q D K)$. To impose that $D$ act nilpotently, we will say that an \define{$\infty$-complex} is a direct sum of finite-dimensional $H_q$-modules, and write $\cC_q$ for the category thereof. Again assuming that $\bK$ is of characteristic $0$, $\cC_q$ then contains no objects of zero quantum dimension (since the quantum dimension of any object is a polynomial in $q$ with nonnegative integer coefficients), and the semisimplification $\overline{\cC_q}$ of $\cC_q$ is identified in Proposition~5.1 of~\cite{EOsimpl}:
$$ \overline{\cC_q} \cong \cat{Rep}(\GL_q(2)),$$
where $\GL_q(2)$ is the Drinfeld--Jimbo quantum group. Intriguingly, the right-hand side is braided, even though the left-hand side has no reason to be. Note that when $q=1$,  we recover the symmetric monoidal category of $\GL(2)$-modules. Indeed, $\cC_1$ was the category of finite-dimensional modules for the Borel subgroup $B \subset \GL(2)$, and the passage $B \leadsto \GL(2)$ is an example of the \define{reductive envelope} of~\cite{MR1956434}.

\subsection{$\ell$-complexes in characteristic $\ell$}

Finally, we discuss the case of most importance in this paper, which is when $\ell$ is prime and $\bK$ has characteristic $\ell$. (We care specifically about the case $\ell=3$.) In this case, we do not need any gradings or $q$'s. More precisely, in characteristic $\ell$, the unique primitive $\ell$'th root of unity is $q=1$, because the definition of ``primitive $\ell$'th root'' (when $\ell$ is prime) is ``solution to $(q^\ell - 1)/(q-1) = q^{\ell-1} + \dots + q + 1$,'' which in characteristic $\ell$ factors as $q^{\ell-1} + \dots + q + 1 = (q-1)^{\ell-1}$. As such, the category $\cC_\ell$ of $\bK[D]/(D^\ell)$-modules is symmetric monoidal. 

As above, the free $\bK[D]/(D^\ell)$-module is the only indecomposable of (quantum) dimension~$0$. Thus we are justified in defining the \define{cohomology} of an object $(V,D) \in \cC_\ell$ to be its image under the semisimplification functor $\rH^* : \cC_\ell \to \overline{\cC_\ell}$. The codomain $\overline{\cC_\ell}$ is studied in detail in~\cite{Ostrik2015}. It is called therein the \define{universal Verlinde category} $\cat{Ver}_\ell$, because the fusion ring of $\overline{\cC_\ell}$ is precisely the fusion ring of the Verlinde category for $\mathrm{SL}(2)$ at level $k=\ell-2$. Note that $\cat{Ver}_\ell$ is symmetric monoidal.

(Actually,~\cite{Ostrik2015} uses a different symmetric monoidal structure on the category $\cC_\ell$ of $\bK[D]/(D^\ell)$-modules. The one we are using corresponds to the Hopf structure on $\bK[D]/(D^\ell)$ in which $D$ is primitive, i.e.\ $\Delta D = D\otimes 1 + 1 \otimes D$. But $(D+1)^\ell = D^\ell + 1^\ell = 0+1$ in characteristic~$\ell$, and so as a category $\cC_\ell \cong \cat{Rep}(\bZ_\ell)$; this corresponds to the Hopf structure in which $D+1$ is grouplike, i.e.\ $\Delta D = D\otimes D + D\otimes 1 + 1 \otimes D$. Although these symmetric monoidal structures on $\cC_\ell$ are different, they determine the same fusion rules on the semisimplification~$\overline{\cC_\ell}$. Theorem~1.5 of~\cite{Ostrik2015} then implies that the two versions produce equivalent symmetric monoidal structures on~$\overline{\cC_\ell}$.)

The main examples are the following. When $\ell=2$, $\cat{Ver}_2 = \cat{Vec}$, and we recover the usual cohomology of an ordinary complex. When $\ell = 3$, $\cat{Ver}_3 = \cat{SVec}$ is the category of supervector spaces: the fermionic line $\bK^{0|1} \in \cat{Ver}_3$ is the cohomology of the two-dimensional indecomposable $\bK[D]/(D^\ell)$-module. When $\ell = 5$, $\cat{Ver}_5$ factors as a tensor product $\cat{SVec} \boxtimes \cat{Fib}$, where $\cat{Fib}$ is the \define{Yang--Lee} or \define{Fibonacci} category, a (symmetric, in characteristic $5$) fusion category with simple objects $\{1,X\}$ and fusion rules $X \otimes X = 1 \oplus X$.
In general, when $\ell$ is an odd prime, $\cat{Ver}_\ell$ factors as $\cat{SVec} \boxtimes \cat{Ver}_\ell^+$, where $\cat{Ver}_\ell^+$ is the ``bosonic part'' of $\cat{Ver}_\ell$, and is spanned by the images under $\H^*$ of the indecomposable $\bK[D]/(D^\ell)$-modules of odd $\bK$-dimension.

Let us end by observing the following. Still working in characteristic $\ell$, with $\ell$ an odd prime, let us say that a (nonperiodic) \define{$\ell$-complex} is a $\bZ$-graded vector space equipped with an $\ell$-differential that raises degree by $1$. We will not use any Koszul signs when multiplying $\ell$-complexes: the underlying vector spaces are entirely bosonic. 
Let us decide that an indecomposable $\ell$-complex supported in degrees $m,\dots,m+n$ has \define{spin} the average degree $m+\frac n 2$. This is either integral or half-integral depending on whether the $\bK$-dimension $n+1$ of the complex is odd or even. This spin is additive under tensor product, and so provides a $\frac12\bZ$ grading to the semisimplification of the category of $\ell$-complexes, in which the ``bosonic'' objects are precisely the ones of integral spin, and the ``fermionic'' objects are the ones of half-integral spin.
The factorization $\cat{Ver}_\ell \cong \cat{SVec} \boxtimes \cat{Ver}_\ell^+$ means that these ``fermionic'' objects really are fermionic in the sense of Koszul signs. In this way, the category of $\ell$-complexes secretly knows about the Koszul sign rules: fermions have ``emerged'' during the passage from the category $\ell$-complexes (a ``UV'' object) to its semisimplification (the ``IR'').

%
%\bibliography{ReferencesWithLinks}{}
%\bibliographystyle{alpha}
%

\newcommand{\etalchar}[1]{$^{#1}$}

\end{document}